\RequirePackage{rotating}
\documentclass{amsart}
\usepackage{array}
\usepackage{graphicx,verbatim, amsmath, amssymb, amsthm, amsfonts, epsfig, amsxtra,ifthen,mathtools,
caption,enumerate,hhline,bbm,capt-of,longtable}	
\usepackage[bookmarks=true, bookmarksopen=false,%
    colorlinks=true,%
    linkcolor=darkblue,%
    citecolor=darkblue,%
    filecolor=darkblue,%
    menucolor=darkblue,%
    linktoc=all,%
    urlcolor=darkblue
]{hyperref}
\usepackage{rotating}

\newtheorem{proposition}{Proposition}[section]
\newtheorem{theorem}[proposition]{Theorem}

\newtheorem{lemma}[proposition]{Lemma}
\newtheorem{prop}[proposition]{Proposition}
\newtheorem{cor}[proposition]{Corollary}
\newtheorem{thm}[proposition]{Theorem}

\theoremstyle{definition}
\newtheorem{example}[proposition]{Example}

\theoremstyle{remark}
\newtheorem{remark}[proposition]{Remark}

\numberwithin{equation}{section}
\usepackage[usenames]{color}
\definecolor{darkblue}{cmyk}{1,0.4,0,0.4}  

\usepackage{float}

\makeatletter
\newcommand\fs@boxedtopcap{\def\@fs@cfont{\bfseries}\let\@fs@capt\floatc@plain
  \def\@fs@pre{\setbox\@currbox\vbox{\hbadness10000
    \moveleft3.4pt\vbox{\advance\hsize by6.8pt
      \hrule \hbox to\hsize{\vrule\kern3pt
        \vbox{\kern3pt\box\@currbox\kern3pt}\kern3pt\vrule}\hrule}}}%
  \def\@fs@mid{\kern2pt}%
  \def\@fs@post{}\let\@fs@iftopcapt\iftrue}
\makeatother

\floatstyle{boxedtopcap}
\newfloat{listing}{pt}{}
\floatname{listing}{List}

\usepackage{color}

\newcommand{\margincolor}{red}      
\definecolor{darkgreen}{rgb}{0,0.7,0}

\newcounter{margincounter}
\newcommand{\marginnum}{
\ifnum\value{margincounter}<10
\textcolor{\margincolor}{\begin{picture}(0,0)\put(2.2,2.4){\circle{9}}\end{picture}\footnotesize\arabic{margincounter}}
\else\ifnum\value{margincounter}<100
\textcolor{\margincolor}{\begin{picture}(0,0)\put(4.256,2.5){\circle{11}}\end{picture}\footnotesize\arabic{margincounter}}
\else
\textcolor{\margincolor}{\begin{picture}(0,0)\put(6.8,2.5){\circle{14}}\end{picture}\footnotesize\arabic{margincounter}}
\fi\fi
}

\makeatletter
\newcommand{\switchmargin}{
\if@reversemargin
\normalmarginpar
\else
\reversemarginpar
\fi
}
\makeatother

\newcommand{\newword}[1]{\textbf{\emph{#1}}}

\newcommand{\integers}{\mathbb Z}

\newcommand{\reals}{\mathbb R}

\newcommand{\covered}{\lessdot}
\newcommand{\set}[1]{{\left\lbrace #1 \right\rbrace}}
\newcommand{\br}[1]{{\langle #1 \rangle}}

\renewcommand{\P}{{\mathcal P}}
\newcommand{\Q}{{\mathcal Q}}

\newcommand{\R}{{\mathcal R}}
\newcommand{\e}{\mathbf{e}}
\newcommand{\ubar}[1]{\text{\b{$#1$}}}
\newcommand{\eb}{\bar{\e}}
\newcommand{\eub}{\ubar{\e}}

\newcommand{\ck}{\spcheck}

\newcommand{\fin}{\mathrm{fin}}

\newcommand{\doub}{\mathrm{doub}}
\newcommand{\out}{\mathrm{out}}
\newcommand{\tNC}{\widetilde{NC}}

\newcommand{\tNCDc}{\tNC\mathstrut^D_{\!c}}
\newcommand{\tNCDetac}{\tNC\mathstrut^D_{\!\eta(c)}}
\newcommand{\tNCDcircc}{\tNC\mathstrut^{D,\circ}_{\!c}}
\newcommand{\tNCDcircetac}{\tNC\mathstrut^{D,\circ}_{\!\eta(c)}}

\newcommand{\tNCBc}{\tNC\mathstrut^B_{\!c}}
\newcommand{\tNCBcircc}{\tNC\mathstrut^{B,\circ}_{\!c}}
\newcommand{\afftype}[1]{{\widetilde{\raisebox{0pt}[6pt][0pt]{#1}}}}
\newcommand{\Stilde}{\raisebox{0pt}[0pt][0pt]{$\,{\widetilde{\raisebox{0pt}[6.1pt][0pt]{\!$S$}}}$}}
\newcommand{\Stildes}{\Stilde^{\mathrm{s}}_{\!2n}} 
\newcommand{\Stildeses}{\Stilde^{\mathrm{ses}}_{2n}} 
\newcommand{\Stildedes}{\Stilde^{\mathrm{des}}_{2n}} 
\newcommand{\Stildejes}{\Stilde^{\mathrm{jes}}_{2n}} 
\newcommand{\Stildebes}{\Stilde^{\mathrm{bes}}_{2n}} 
\newcommand{\nbig}{\#\mathrm{big}}
\newcommand{\nneg}{\#\mathrm{neg}}
\newcommand{\perm}{\mathsf{perm}}

\newcommand{\McSul}{\mathsf{McSul}}
\newcommand{\New}{\mathsf{New}}
\newcommand{\C}{\mathsf{C}}

\newcommand{\vr}{{\varrho}}

\newcommand{\fuu}{f_{\uparrow\uparrow}}
\newcommand{\fud}{f_{\uparrow\downarrow}}
\newcommand{\fdu}{f_{\downarrow\uparrow}}
\newcommand{\fdd}{f_{\downarrow\downarrow}}

\renewcommand{\mod}[1]{\ (\mathrm{mod}\ #1)}
\newcommand{\cmod}[1]{(\mathrm{mod}\ #1)}  

\newcommand{\curve}{\mathsf{curve}}

\title[Noncrossing partitions of an annulus with double points]{Symmetric noncrossing partitions of an annulus with double points\\\vspace{-10pt}}
\author{Nathan Reading\\\vspace{-10pt}}
\thanks{Nathan Reading was partially supported by the Simons Foundation under award number 581608 and by the National Science Foundation under award number DMS-2054489.  }

\begin{document}

\begin{abstract}
For affine Coxeter groups of affine types $\afftype D$ and $\afftype B$, we model the interval $[1,c]_T$ in the absolute order by symmetric noncrossing partitions of an annulus with one or two double points.
In type $\afftype B$ (and \emph{almost} in type $\afftype D$), the diagrams also model the larger lattice defined by McCammond and Sulway.
\end{abstract}

\maketitle

\vspace{-20pt}

\setcounter{tocdepth}{2}
\tableofcontents

%
%
%
%

\section{Introduction}
The object of this paper is to construct combinatorial models for certain intervals in the absolute order on a Coxeter group~$W$ and analogous intervals in supergroups of~$W$.
Specifically, we consider classical affine Coxeter groups and study the interval $[1,c]_T$ between the identity element and a Coxeter element.
In finite type, this interval is a Garside structure for the associated Artin group~\cite{Bessis,Bra-Wa}.
In affine type, McCammond and Sulway~\cite{McSul} constructed a supergroup of $W$ and showed that the analogous interval between $1$ and $c$ is a lattice and is a Garside structure for a supergroup of the Artin group.
In~\cite{affncA} and this paper, we construct combinatorial models for the intervals $[1,c]_T$ and for McCammond and Sulway's lattice in classical affine type, types $\afftype A$ and $\afftype C$ in \cite{affncA} and types $\afftype D$ and $\afftype B$ here.

The combinatorial model consists of symmetric noncrossing partitions of an annulus with two double points (type $\afftype D$) or one double point (type $\afftype B)$.
See Figures~\ref{nc ex fig} and~\ref{nc ex fig B}.
In fact, the McCammond-Sulway lattice in type $\afftype D$ is \emph{slightly} beyond the reach of our combinatorial model, but close enough to make the combinatorial model useful:
We need to add a finite number of group elements, which we identify explicitly as permutations, to the model to obtain the entire interval.
(Indeed, the interval $[1,c]_T$ contains all but finitely many elements of the McCammond-Sulway lattice, but the combinatorial model leaves out fewer elements.)
In type~$\afftype B$, the poset of symmetric noncrossing partitions is isomorphic to the McCammond-Sulway lattice.
In both types, the model captures the interval $[1,c]_T$ as the poset of symmetric noncrossing partitions with no ``dangling annular blocks''.
These planar diagrams ultimately come from projecting a small orbit to Coxeter plane, in the spirit of~\cite{plane}.
The constructions were generalized in \cite{surfnc} to a notion of symmetric noncrossing partitions of marked surfaces with double points.

The planar diagrams in this paper are analogous to the noncrossing partition diagrams in finite type D due to Athanasiadis and Reiner~\cite{Ath-Rei}.
There is also a construction~\cite{NicOan} of annular noncrossing partitions of types B and D, but these live in the finite Coxeter groups of these~types.

The isomorphisms from noncrossing partitions to $[1,c]_T$ and larger posets involve reading cycles in permutations from the boundaries of blocks.  
To prove the isomorphisms in type $\afftype D$, we must describe the cycles in great detail.
In type~$\afftype B$, we shortcut such details by ``folding'' the type $\afftype D$ results in the usual sense.

Many of the results in type $\afftype D$ appeared in Laura Brestensky's thesis \cite{BThesis}.
Here, we give a new account of these results (retaining some key ideas and some arguments), make the connection to the lattice of McCammond and Sulway, and prove the analogous results in type $\afftype B$.
(Along the way, we correct an error in \cite{BThesis} that was repeated in early versions of~\cite{surfnc}.
See Remark~\ref{error} and \cite[Section~4.6]{surfnc}.)

\section{Affine Coxeter groups of type $\afftype{D}$}
\label{aff type d}
In this section, after some general background on Coxeter groups and the affine case, we construct a root system, a reflection representation, and a permutation representation for a Coxeter group of type $\afftype{D}_{n-1}$.
We also construct a supergroup of the Coxeter group.
Some of what follows requires the assumption that $n\ge5$.

\subsection{General background on affine Coxeter groups}\label{gen sec}
We assume many standard definitions and facts about Coxeter groups and root systems.
Much of the assumed background, including justification of some facts given below, can be found in \cite[Section~2]{affncA}.
Here, we highlight some aspects of our approach that may be different from what some readers expect.

We use the usual reflection representation of a Coxeter group $W$ on a vector space $V$ spanned by the simple roots $\alpha_i$.
However, we depart from the typical Lie theoretic conventions as follows:
We take the simple co-roots $\alpha_i\ck$ to be scalings of the simple roots and take the fundamental weights $\rho_i$ to be the basis of the dual space $V^*$ dual to the simple co-roots.

A Coxeter element $c$ of $W$ is the product of the simple reflections $S$ in some order.
The choice of Coxeter element determines a skew-symmetric bilinear form $\omega_c$ on $V$, given by
\begin{equation}\label{omega def}
\omega_c(\alpha_i\ck,\alpha_j)=
\begin{cases}
a_{ij}&\text{if }s_i\text{ follows }s_j\text{ in }c,\\
0&\text{if }i=j\text{ or}\\
-a_{ij}&\text{if }s_i\text{ precedes }s_j\text{ in }c.\hfill\qedhere
\end{cases}
\end{equation}
Here, the $a_{ij}$ are entries of the Cartan matrix.
It follows from \cite[Lemma~3.8]{typefree} that $\omega_c(cx,cy)=\omega_c(x,y)$ for any $x,y\in V$.

We write $\ell_T$ for the length function with respect to the set $T$ of reflections, $\le_T$ for the \newword{absolute order} on~$W$, and $[1,c]_T$ for the interval from $1$ to $c$ in that order.

When $W$ is of affine type and rank $n$, we construct a rank-$n$ root system $\Phi$ and an $n\times n$ Cartan matrix~$A$ from a finite root system $\Phi_\fin$ in the standard way.
(See, for example, \cite[Chapter~4]{Humphreys}.)
Write $\delta$ for the shortest positive imaginary root.
The real roots are precisely the vectors $\beta+k\delta$ for $\beta\in\Phi_\fin$ and $k\in\integers$.

In affine type, the action of $c$ on $V$ has an eigenvalue $1$ with algebraic multiplicity~$2$ and with a $1$-dimensional fixed space spanned by $\delta$.
There is a unique generalized $1$-eigenvector $\gamma_c$ associated to $\delta$ with $\gamma_c$ contained in the span of~$\Phi_\fin$.
(The fact that $\gamma_c$ is a generalized $1$-eigenvector associated to $\delta$ means that ${c\gamma_c=\gamma_c+\delta}$.)
The action of $c$ on $V^*$ has a $1$-eigenvector $\omega_c(\delta,\,\cdot\,)$ with an associated generalized $1$-eigenvector $\omega_c(\gamma_c,\,\cdot\,)$.
This means that $c\cdot\omega_c(\gamma_c,\,\cdot\,)=\omega_c(\gamma_c,\,\cdot\,)+\omega_c(\delta,\,\cdot\,)$.
The plane in $V^*$ spanned by $\omega_c(\delta,\,\cdot\,)$ and $\omega_c(\gamma_c,\,\cdot\,)$ is the \newword{Coxeter plane in $V^*$}.

\subsection{Affine signed permutations}\label{C sec}
To define the Coxeter group of type $\afftype D_{n-1}$, it is convenient to first define the Coxeter group of type $\afftype C_{n-1}$.
Consider $\reals^{n+1}$ with basis $\e_1,\ldots,\e_{n+1}$ and take $V$ to be the set of vectors in~$\reals^{n+1}$ whose $\e_n$-coordinate is zero.
Define $\delta$ to be $\e_{n+1}+\e_{n-1}$ and define vectors $\e_i\in V$ for $i\in\integers\setminus\set{\ldots,-n,0,n,\ldots}$ by two rules:
$\e_{-i}=-\e_i$ and $\e_{i+2n}=\e_i+\delta$.
Define a symmetric bilinear form~$K$ on $V$ as the usual inner product on $\e_1,\ldots,\e_{n-1}$ with $K(\e_{n+1},x)=-K(\e_{n-1},x)$ for all $x\in V$.
In particular, $K(\e_{n+1},\e_{n+1})=-K(\e_{n-1},\e_{n+1})={K(\e_{n-1},\e_{n-1})=1}$.

Define simple roots $\alpha_0=\e_1-\e_{-1}=2\e_1$, $\alpha_i=\e_{i+1}-\e_i$ for $i=1,\ldots,n-2$ and $\alpha_{n-1}=\e_{n+1}-\e_{n-1}$.
The simple roots have squared lengths $K(\alpha_i,\alpha_i)=2$ except that $K(\alpha_0,\alpha_0)=4$ and $K(\alpha_{n-1},\alpha_{n-1})=4$.
The simple coroots are thus $\alpha_0\ck=\e_1$, $\alpha_i\ck=\e_{i+1}-\e_i$ for $1\le i\le n-2$, and $\alpha_{n-1}\ck=\frac12(\e_{n+1}-\e_{n-1})$.
One can check that these roots and coroots determine the correct Cartan matrix for a root system of type $\afftype{C}_{n-1}$.
Also, $\delta=\alpha_0+2\sum_{i=1}^{n-2}\alpha_i+\alpha_{n-1}$, as expected for type~$\afftype{C}_{n-1}$.

Each simple root $\alpha_i$ defines a simple reflection $s_i$ on $V$ by $s_i(x)=x-K(\alpha\ck_i,x)\alpha_i$.
Each $s_i$ acts by permuting the vectors $\set{\e_i:i\in\integers,\,i\not\equiv0\mod{n}}$, so we identify them with the corresponding permutations of indices.
The Coxeter group generated by the $s_i$ is the group $\Stildes$ of \newword{affine signed permutations}: the permutations $\pi:\integers\to\integers$ with $\pi(i+2n)=\pi(i)+2n$ and $\pi(-i)=-\pi(i)$ for all $i\in\integers$.
Affine signed permutations fix all multiples of $n$.
For details, see \cite[Section~8.4]{Bj-Br} (with slightly different conventions) and \cite[Section~4.1]{affncA} (with conventions as in this paper).

Because of the condition that $\pi(i+2n)=\pi(i)+2n$, we adopt the notation $(a_1\,\,\,a_2\,\cdots\,a_k)_{2n}$ for the infinite product $\prod_{\ell\in\integers}(a_1+2n\ell\,\,\,a_2+2n\ell\,\cdots\,a_k+2n\ell)$ of cycles.
For the same reason, the notation $(\cdots\,a_1\,\,\,a_2\,\cdots\,a_\ell\,\,\,a_1+2kn\,\cdots)$ uniquely specifies an infinite cycle, for any nonzero integer $k$.
Because of the additional condition that $\pi(-i)=-\pi(i)$, we adopt the notation $(\!(a_1\,\cdots\,a_k)\!)_{2n}$ for the infinite product of cycles $(a_1\,\cdots\,a_k)_{2n}\cdot(-a_1\,\cdots\,-a_k)_{2n}$.
Similarly, the notation $(\!(\cdots\,a_1\,\,\,a_2\,\cdots\,a_\ell\,\,\,a_1+2n\,\cdots)\!)$ means 
\[(\cdots\,a_1\,\,\,a_2\,\cdots\,a_\ell\,\,\,a_1+2n\,\cdots)(\cdots\,-\!a_1\,\,-\!a_2\,\cdots\,-a_\ell\,\,-\!a_1-2n\,\cdots).\]

With this notation, the simple reflections are $s_0=(-1\,\,\,1)_{2n}$, $s_i=(\!(i\,\,\,i+1)\!)_{2n}$ for $i=1,\ldots,n-2$, and $s_{n-1}=(n-1\,\,\,n+1)_{2n}$.
There are two types of reflections in $\Stildes$:
For each pair $i,j\in\integers\setminus\set{\ldots,-n,0,n,\ldots}$ with $j\not\equiv\pm i\mod {2n}$, there is a reflection $(\!(i\,\,\,j)\!)_{2n}$.
For each pair of indices $i$ and~$j$ in ${\integers\setminus\set{\ldots,-n,0,n,\ldots}}$ with $j\equiv-i\mod {2n}$, there is a reflection $(i\,\,\,j)_{2n}$.

\subsection{Affine doubly even-signed permutations}\label{D perm sec}
We build a root system of type~$\afftype{D}_{n-1}$ in the same vector space $V$ that was introduced in Section~\ref{C sec}. 
The simple roots are $\alpha_0=\e_2 + \e_1=\e_1-\e_{-2}$, $\alpha_{n-1}=\e_{n+1}-\e_{n-2} $ and  $\alpha_i={\e_{i+1}-\e_i}$ for $i=1\ldots n-2$.
The simple co-roots are $\alpha_i\ck=\alpha_i$ for $i=0,\ldots,{n-1}$.  
One can check that these simple roots/co-roots define the correct Cartan matrix.
The vector $\delta=\e_{n+1}+\e_{n-1}$ is $\alpha_0+\alpha_1+2\bigl(\sum_{i=2}^{n-3}\alpha_i\bigr)+\alpha_{n-2}+\alpha_{n-1}$, as expected.

The simple reflections $S=\set{s_0,\ldots,s_{n-1}}$ permute $\set{\e_i:i\in\integers,\,i\not\equiv0\mod{n}}$.
Using the cycle notation conventions from Section~\ref{C sec}, the corresponding permutations of $\set{i\in\integers:\,i\not\equiv0\mod{n}}$ are
$s_0=(\!(1 \,\, -2)\!)_{2n}$, $s_i=(\!(i \,\,\,i+1)\!)_{2n}$ for $i = 1,\ldots,n-2$, and $s_{n-1}=(\!(n-2 \,\,\,n+1)\!)_{2n}$.

The set $S$ generates a Coxeter group of type~$\afftype{D}_{n-1}$ (a subgroup of the group $\Stildes$ of affine signed permutations), namely the group $\Stildedes$ of \newword{affine doubly even-signed permutations}.
An affine signed permutation is doubly even-signed if it sends an even number of positive integers to negative integers and sends an even number of integers less than $n$ to integers greater than $n$.
Thus $\Stildedes$ is the set of permutation $\pi$ of the integers such that
\begin{itemize}
    \item $\pi(i + 2n) = \pi(i) + 2n$,
    \item $\pi(i) = -\pi(-i)$,
    \item $\set{i\in\integers:i>0,\pi(i)<0}$ has an even number of elements, and
    \item $\set{i\in\integers:i<n,\pi(i)>n}$ has an even number of elements.
\end{itemize}
The set $T$ of reflections in $\Stildedes$ is $\set{(\!(i \,\, j)\!)_{2n}:i\not\equiv\pm j\mod{2n}}$.
Correspondingly, the set of real roots in $\Phi$ is $\set{\e_j-\e_i:i\not\equiv\pm j\mod{2n}}$. 
The finite root system~$\Phi_\fin$ 
is $\set{\pm\e_j\pm \e_i:1\le i<j\le n-1}$. 
More details on this Coxeter group can be found in \cite[Section~8.6]{Bj-Br}, but with slightly different conventions.

\subsection{Affine jointly even-signed permutations}\label{D big sec}
We write $\Stildejes$ for the subgroup of $\Stildes$ (and supergroup of $\Stildedes$) consisting of \newword{affine jointly even-signed permutations}.
These are the permutations $\pi$ of $\mathbb{Z}$ such that:
\begin{itemize}
    \item $\pi(i + 2n) = \pi(i) + 2n$
    \item $\pi(i) = -\pi(-i)$
    \item the number of elements of $\set{i\in\integers:i>0,\pi(i)<0}$ plus the number of elements of $\set{i\in\integers:i<n,\pi(i)>n}$ is even.
\end{itemize}
Define permutations $\ell_i =(\!(\cdots\,i \,\,\, i+2n\,\cdots)\!)$ for $i=\pm1,\pm2,\ldots,\pm(n-1)$, noticing that $\ell_{-i}=\ell_i^{-1}$.
We call a permutation $\ell_i$ a \newword{loop} because eventually it will be associated to a noncrossing partition whose only nontrivial blocks resemble loops at~$i$ and~$-i$.
Let $L$ be the set $\set{\ell_{-n+1},\ldots\ell_{-1},\ell_1,\ldots\ell_{n-1}}$.

\begin{proposition}\label{big group gen}
The group $\Stildejes$ is generated by $S\cup\set{\ell_1}$.
It is also generated by~$T\cup L$.
\end{proposition}

Write $\nneg(\pi)$ for the number of elements of $\set{i\in\integers:i>0,\pi(i)<0}$ and $\nbig(\pi)$ for the number of elements of $\set{i\in\integers:i<n,\pi(i)>n}$.
The key to Proposition~\ref{big group gen} is the following lemma, whose straightforward but tedious proof is omitted here.

\begin{lemma}\label{loop big neg}
For $\pi\in\Stildes$ and $\ell_i\in L$, 
\[\nbig(\pi\ell_i)=\nbig(\pi)\pm1\quad\text{and}\quad\nneg(\pi\ell_i)=\nneg(\pi)\pm1.\]
\end{lemma}

\begin{proof}[Proof of Proposition~\ref{big group gen}]
Setting $\pi$ to be the identity in Lemma~\ref{loop big neg}, we see that $L\subseteq\Stildejes$.
Since $S$ generates $\Stildedes$ and $S\subseteq T$, it only remains to check that $\Stildedes\cup\set{\ell_1}$ generates~$\Stildejes$.
Suppose $\pi\in\Stildejes$, so that $\nbig(\pi)+\nneg(\pi)$ is even.
If $\nbig(\pi)$ and $\nneg(\pi)$ are both even, then $\pi\in\Stildedes$.
If $\nbig(\pi)$ and $\nneg(\pi)$ are both odd, then $\pi=(\pi\ell_1^{-1})\cdot\ell_1$, and Lemma~\ref{loop big neg} implies that $\pi\ell_1^{-1}\in\Stildedes$.
\end{proof}

We write $\ell_{T\cup L}(x)$ for the length of $x\in\Stildejes$ relative to the generating set $T\cup L$ (the length of a shortest expression for $x$ as a product of elements of $T\cup L$).
We also write $\le_{T\cup L}$ for the prefix/postfix/subword order on $\Stildejes$ relative to the alphabet $T\cup L$ and $[1,c]_{T\cup L}$ for the interval between the identity and $c$ in this order.

\section{The symmetric annulus with two double points}
\label{d annulus}
In this section, we initiate the construction of the planar model for $[1,c]_T$ in type~$\afftype{D}_{n-1}$ by defining the symmetric annulus with two double points.
This is a case of the definition of symmetric marked surfaces with double points \cite[Definition~3.1]{surfnc}.
Although the Coxeter-theoretic aspects of the model require the assumption that $n\ge5$, some of the combinatorics of the planar diagrams works for $n$ as small as $3$.
We implicitly assume that $n$ is large enough in every construction; always $n\ge5$ is safe.
We connect the symmetric annulus with two double points to Coxeter elements in two ways.
First, we show how a choice of Coxeter element is encoded by a choice of a specific symmetric annulus with two double points.
Second, we show how the symmetric annulus with two double points arises from the projection of an orbit in $V$ to the Coxeter plane in $V^*$.

\subsection{Coxeter elements in type $\afftype D$}\label{sym ann doub sec}

We write $D$ for the symmetric annulus with two double points and pass freely between two ways of viewing it.
Viewed more simply, $D$ is an annulus with an orientation-preserving symmetry $\phi$ whose two fixed points are called \newword{double points}, which we will distinguish as the \newword{upper double point} and \newword{lower double point}.
To make the symmetry $\phi$ definite, we temporarily view~$D$ as the cylindrical tube $\set{(x,y,z)\in\reals^3:x^2+y^2=1,\,-1\le z\le1}$.
The map $\phi$ is a half-turn symmetry of the tube about the $x$-axis, and the double points are $(1,0,0)$ and $(-1,0,0)$.
Moving forward, we picture $D$ as an annulus in the plane.

Viewed less simply, the two double points are literally ``doubled'' in the annulus, to obtain two copies of each double point, labeled $+$ and~$-$.
This could be accomplished, for example, by taking two disjoint copies of the annulus and identifying them pointwise, except the double points.
Crucially, when we consider the double points to be literally doubled, we let the map $\phi$ also swap $+$ and $-$ at each double point.
The left picture in Figure~\ref{sym ann fig} shows the annulus with two double points.
\begin{figure}
\scalebox{0.45}{\includegraphics{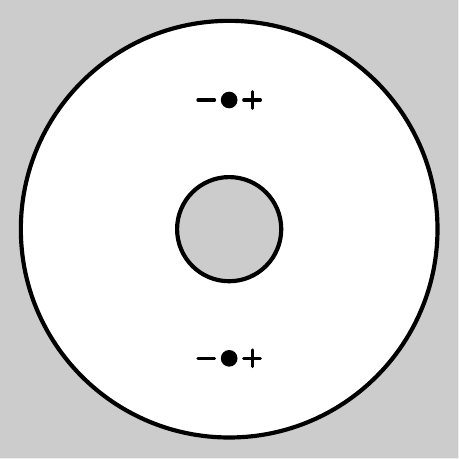}}
\quad
\scalebox{0.45}{\includegraphics{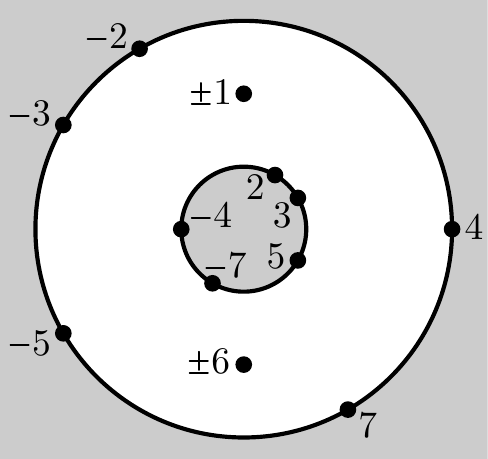}}
\quad
\scalebox{0.45}{\includegraphics{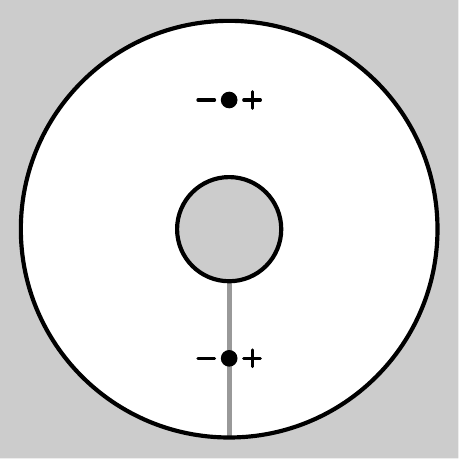}}
\caption{Left:  The symmetric annulus with two double points.
Center: Inner, outer and double points for $c=s_3s_6s_2s_0s_1s_5s_7s_4$.
Right:  The date line.}
\label{sym ann fig}
\end{figure}

A choice of Coxeter element $c$ of $\Stildedes$ corresponds to placing $\pm1,\ldots, {\pm(n-1)}$ on the boundary and double points of $D$ as follows.
Each double point gets two numbers with the same absolute value.
The numbers $\pm1$ are placed at the upper double point if and only if $s_0$ and $s_1$ either both precede or both follow $s_2$ in $c$.
Otherwise, $\pm2$ are at the upper double point.
Similarly, $\pm(n-1)$ are at the lower double point if and only if $s_{n-1}$ and $s_{n-2}$ either both precede or both follow $s_{n-3}$ in $c$.
Otherwise, $\pm(n-2)$ are placed at the lower double point.
We reuse the term \newword{double points} to describe the numbers that are placed at the double points of $D$.

The remaining numbers in $\set{\pm1,\ldots,\pm(n-1)}$ are placed on the inner or outer boundary of $D$ and called \newword{outer points} or \newword{inner points} accordingly.
For $i\in\set{2,\ldots,n-2}$, if $\pm i$ is not already on a double point, then $i$ is outer and $-i$ is inner if and only if $s_{i-1}$ precedes $s_{i}$ in $c$.
Otherwise, $i$ is inner and $-i$ is outer.
If $\pm1$ is not already on a double point, then $1$ is outer and $-1$ is inner if and only if $s_0$ precedes $s_2$ in $c$.
Otherwise, $1$ is inner and $-1$ is outer.
Similarly, if $\pm(n-1)$ is not on a double point, then $n-1$ is outer and $-(n-1)$ is inner if and only if $s_{n-3}$ precedes $s_{n-1}$ in $c$.
Otherwise $n-1$ is inner and $-(n-1)$ is outer.

The outer points are placed in increasing order clockwise, with negative numbers to the left and positive numbers to the right, and the inner points are similarly placed in increasing order clockwise.
The numbers are placed so that the angle from a negative number $-i$ to the bottom of the annulus equals the angle from the bottom to the positive number $i$.
In particular, $\phi$ maps each $i$ to $-i$.
An example is shown in the center picture of Figure~\ref{sym ann fig}.

\begin{lemma}\label{d annulus label}
Let $c$ be a Coxeter element of $\Stildedes$, represented as a partition of $\set{\pm1,\ldots,\pm(n-1)}$ into outer, inner, and double points.
If $a_1, \ldots, a_{n-3}$ are the outer points in increasing order, $p_1$ is a label of the upper double point, and~$p_2$ is the positive label of the lower double point, then $c$ is the permutation
\[
(\!(\cdots\,a_1\,\,\,a_2\,\cdots\,a_{n-3} \,\,\,a_1+2n\,\cdots)\!)(p_1\,\,\,{-p_1})_{2n}(p_2 \,\,\,{-p_2+2n})_{2n}
\]
\end{lemma}

\begin{proof}
One can easily check the statement when ${c=s_0 s_1 \cdots s_{n-1}}$, with double points $\pm1$ and $\pm(n-1)$ and outer points $2,\ldots,n-2$.
Any two Coxeter elements of $\Stildedes$ are conjugate by source-sink moves, so we can complete the proof by showing that source-sink moves preserve the statement.
We omit the simple details.
\end{proof}

\subsection{Projecting to the Coxeter plane}\label{proj sec}
The symmetric annulus with two double points is related to the projection of a certain orbit in~$V$ to the Coxeter plane in~$V^*$.

The simple roots $\alpha_0$, $\alpha_1$, $\alpha_{n-2}$, and $\alpha_{n-1}$ are \newword{leaves} (i.e.\ the leaves of the Coxeter diagram) and the other simple roots are \newword{non-leaves}.
Given a Coxeter element~$c$, we write $s_i\to s_j$ to mean that $s_i$ and $s_j$ form an edge in the Coxeter diagram and~$s_i$ precedes $s_j$ in~$c$.
We restate \eqref{omega def} for $W$ of type~$\afftype{D}_{n-1}$: $\omega_c(\alpha_i,\alpha_j)=1$ if $s_i\to s_j$ or $\omega_c(\alpha_i,\alpha_j)=-1$ if $s_j \to s_i$ or $\omega_c(\alpha_i,\alpha_j)=0$ otherwise.
The following lemma is immediate because $\delta=\alpha_0+\alpha_1+2\bigl(\sum_{i=2}^{n-3}\alpha_i\bigr)+\alpha_{n-2}+\alpha_{n-1}$.

\begin{lemma} \label{big omega}
If $c$ is a Coxeter element of $\Stildedes$, then $\omega_c(\delta,\,\cdot\,) = \sum_{j=0}^{n-1} k_j \rho_j$, where 
\begin{align*}
k_j =\omega_c(\delta, \, \alpha_j\,)=&\,\,\#\text{ leaves } \alpha_i \text{ such that } s_i \to s_j \hspace{.1in} \\
& + 2\,(\#\text{ non-leaves } \alpha_i\text{ such that } s_i \to s_j) \hspace{.1in}\\
& - \,\# \text{ leaves } \alpha_i \text{ such that } s_j \to s_i \hspace{.1in}\\
& - 2\,(\#\text{ non-leaves } \alpha_i \text{ such that } s_j \to s_i).
\end{align*}
\end{lemma}


The following lemma is immediate from Lemma~\ref{big omega} and the expressions for $\e_1,\ldots,\e_{n-1}$ in terms of simple roots, namely $\e_1=\frac12\alpha_0-\frac12\alpha_1$ and, for each ${i\in\set{2,\ldots, n-1}}$, $\e_i=\frac12\alpha_0+\frac12\alpha_1+\alpha_2+\cdots+\alpha_{i-1}$.

\begin{lemma} \label{d horiz proj}
Let $c$ be a Coxeter element of $\Stildedes$, represented as a partition of $\{\pm 1, \ldots, \pm (n-1)\}$ into inner, outer, and double points. 
For ${j\not\equiv0\mod n}$,
\[\omega_c(\delta, \e_j)=\begin{cases}
2&\text{if }j=i\mod{2n}\text{ for some inner point }i\\
-2&\text{if }j=i\mod{2n}\text{ for some outer point }i,\text{ or}\\
0&\text{if }j=i\mod{2n}\text{ for some double point }i.
\end{cases}\]
\end{lemma}

We can now describe the projection of $\set{\e_i:i\in\integers,i\not\equiv0\mod n}$ to the Coxeter plane in $V^*$.
Recall that this plane is spanned by $\omega_c(\delta,\,\cdot\,)$ and $\omega_c(\gamma_c,\,\cdot\,)$.

\begin{theorem}\label{d orb proj}
Let $c$ be a Coxeter element of $\Stildedes$, represented as a partition of $\set{\pm1,\pm2, \ldots,\pm(n-1)}$ into inner, outer, and double points.
\begin{enumerate}[\quad\rm\bf1.]
\item \label{3 vert d}
Projection to the Coxeter plane in $V^*$ takes $\set{\e_i:i\in\integers,i\not\equiv0\mod n}$ into three distinct lines parallel to $\omega_c(\gamma_c,\,\cdot\,)$.
\item \label{pos neg zero vert d}
One line has positive $\omega_c(\delta,\,\cdot\,)$-coordinate and contains the projections of vectors $\e_i$ such that $i\mod {2n}$ is inner, one line has negative $\omega_c(\delta,\,\cdot\,)$-coordinate and contains the projections of vectors $\e_i$ such that $i\mod {2n}$ is outer, and one line has zero $\omega_c(\delta,\,\cdot\,)$-coordinate and contains the projections of vectors~$\e_i$ such that $i\mod {2n}$ is a double point.
\item \label{bigger vert d}
If $i<i'$ and $i$ and $i'$ are either both outer or both inner, then the projection of $\e_{i'}$ has strictly larger $\omega_c(\gamma_c,\,\cdot\,)$-coordinate than the projection of $\e_i$.
\item \label{evenly}
On the lines containing inner points and outer points, the difference between $\omega_c(\gamma_c,\,\cdot\,)$-coordinates of adjacent points is $2$.
\item \label{doubles share points upper}
If $i$ is an upper double point, then $\e_i$ and $\e_{-i}$ project to the same point.
\item \label{doubles share points lower}
If $i$ is a lower double point, then $\e_i$ and $\e_{-i+2n}$ project to the same point.
\item \label{const dist d}
The difference in the $\omega_c(\gamma_c,\,\cdot\,)$-coordinates of the projection of $\e_i$ and the projection of $\e_{i+2n}$ is $\omega_c(\gamma_c,\delta)=2n-6$, independent of $i$.
\end{enumerate}
\end{theorem}

\begin{proof}
Assertions~\ref{3 vert d} and~\ref{pos neg zero vert d} of the theorem follow immediately from Lemma~\ref{d horiz proj}.

We note that $\omega_c(\gamma_c,c^{-1}\e_i)$ equals $\omega_c(c\gamma_c,\e_i)$, which equals $\omega_c(\gamma_c+\delta,\e_i)$.
Thus $\omega_c(\gamma_c,c^{-1}\e_i)-\omega_c(\gamma_c,\e_i)=\omega_c(\delta,\e_i)$.
Lemma~\ref{d annulus label} says that if $i$ is outer, then $c^{-1}\e_i$ is the next smallest integer that is outer, and Lemma~\ref{d horiz proj} says that $\omega_c(\delta,\e_i)=-2$, so the projection of $\e_i$ has $\omega_c(\gamma_c,\,\cdot\,)$-coordinate $2$ larger than the projection of $\e_{i'}$ for $i'$ the next smallest outer integer.
Similarly, if $i\mod {2n}$ is inner, then $c^{-1}\e_i$ is the next largest integer that is inner and $\omega_c(\delta,\e_i)=2$, so the projection of $\e_i$ has $\omega_c(\gamma_c,\,\cdot\,)$-coordinate $2$ smaller than the projection of $\e_{i'}$ for $i'$ the next largest inner integer.
We have proved Assertions~\ref{bigger vert d} and~\ref{evenly}.

If $i$ is an upper double point, then $c^{-1}\e_i = \e_{-i}$. 
Thus $\omega_c(\gamma_c,\e_{-i})-\omega_c(\gamma_c,\e_i)$ is 
\[\omega_c(\gamma_c,c^{-1}\e_i)-\omega_c(\gamma_c,\e_i)=\omega_c(c\gamma_c,\e_i)-\omega_c(\gamma_c,\e_i)=\omega_c(\delta,\e_i) = 0.\] 
If $i>0$ is a lower double point, $c^{-1}\e_i=\e_{-i+2n}$, so $\omega_c(\gamma_c,\e_{-i+2n})-\omega_c(\gamma_c,\e_i)$ is
\[\omega_c(\gamma_c,c^{-1}\cdot\e_i)-\omega_c(\gamma_c,\e_i) = \omega_c(c\gamma_c,\e_i)-\omega_c(\gamma_c,\e_i)=\omega_c(\delta,\e_i) = 0.\] 
We have proved assertions~\ref{doubles share points upper} and~\ref{doubles share points lower}.

Since $\e_{i+2n}=\e_i+\delta$, the $\omega_c(\gamma_c,\,\cdot\,)$-coordinates of the projection of $\e_i$ and the projection of $\e_{i+2n}$ differ by $\omega_c(\gamma_c,\delta)$.
There are $n-3$ outer points in $\set{\pm1,\ldots,\pm(n-1)}$, so Assertion~\ref{evenly} implies that $\omega_c(\gamma_c,\delta)=2n-6$.
This is Assertion~\ref{const dist d}.
\end{proof}

We see from Theorem~\ref{d orb proj}.\ref{const dist d} that the projection of $\set{\e_i:i\in\integers,i\not\equiv0\mod n}$ to the Coxeter plane in $V^*$ lives on an infinite strip and has translational symmetry.
We mod out by the translational symmetry by identifying points in the Coxeter plane with the same $\omega_c(\delta,\,\cdot\,)$-coordinate and whose $\omega_c(\gamma_c,\,\cdot\,)$-coordinates differ by multiples of $\omega_c(\gamma_c,\delta)$.
We thus recover (up to shifting the points symmetrically along the boundary without moving points through each other) the symmetric annulus $D$ with outer points, inner points, and double points that describe~$c$.

\section{Structure of permutations in $[1,c]_{T\cup L}$}

\subsection{Rank function}\label{rank sec}
We describe the rank function in $[1,c]_{T\cup L}$ in terms of cycle structure.
We think of a permutation $\pi\in\Stildes$ as acting on $\integers\setminus\set{\ldots,-n,0,n,\ldots}$ and thus ignore the cycles $(0)_{2n}$ and $(n)_{2n}$.
Given a cycle in $\pi$, the \newword{class} of the cycle is the cycle together with all cycles that can be obtained by negating and/or adding multiples of $2n$.
The cycle decomposition of $\pi$ is a disjoint union of classes.

The classes of finite cycles of $\pi$ will be classified as \newword{nonsymmetric}, \newword{symmetric}, or \newword{tiny}.
Symmetric and tiny classes are furthermore distinguished as \newword{upper} or \newword{lower}.
A class of \newword{nonsymmetric finite cycles} is $(\!(a_1\,\,\,a_2\,\cdots\,a_m)\!)_{2n}$ with $m\ge1$, such that $a_p\not\equiv\pm a_q\mod{2n}$ for distinct $p,q\in\set{1,\ldots,m}$.

The class of \newword{upper tiny cycles} is $(i\,\,{-i})_{2n}$ where $i$ is the positive number on the upper double point.
This class consists of all cycles of the form $(b\,\,\,{-b+4qn})$ such that $q$ is an integer and $b=i+2qn$.
The class of \newword{lower tiny cycles} is $(j\,\,\,{-j+2n})_{2n}$, where $j$ is the positive number on the lower double point.
This class consists of all cycles of the form $(b\,\,\,\,{-b+(4q+2)n})$ where $q$ is an integer and $b=j+2qn$.
The cycles $(i\,\,\,{-i})$ and $(j\,\,\,{-j+2n})$ are called \newword{principal tiny cycles}.

A class of \newword{symmetric cycles} is a class of \emph{non-tiny} cycles of the form 
\begin{equation}\label{sym cyc}
(a_1\,\,\,a_2\,\cdots\,a_m\,\,\,-a_1+2kn\,\,-a_2+2kn\,\cdots\,-a_m+2kn)_{2n}
\end{equation}
for some integer $k$.
We omit the easy proof of the following observations.
\begin{lemma}\label{sym obs}
The integer $k$ appearing in the notation of \eqref{sym cyc} has constant parity within a class of symmetric cycles.
Writing a cycle in the notation of \eqref{sym cyc} and reading off $k$ defines a bijection between the cycles in the class and the even integers or the odd integers.
Negating a cycle corresponds to negating $k$.
\end{lemma}

Given a class of symmetric cycles written as in \eqref{sym cyc}, we call the representative with $k=0$ or $1$ a \newword{principal symmetric cycle}.

A class of symmetric cycles is a class of \newword{upper symmetric cycles} if $k$ is even or a class of \newword{lower symmetric cycles} if $k$ is odd.
Just as a symmetric cycle or a tiny cycle is either upper or lower, a double point or a mod-$2n$ translate of a double point is also either upper or lower.
We say that two of these objects \newword{match} if they are both upper or both lower.
If a double point matches a symmetric cycle $(a_1\,\cdots\,a_m\,\,\,-a_1+2kn\,\cdots\,-a_m+2kn)$, we say that the double point matches $k$.

We emphasize that by convention (to avoid constantly saying ``non-tiny symmetric cycle''), the tiny cycles are \emph{not} included under the heading of ``symmetric cycle''.
However, even though a symmetric cycle is not tiny, it \emph{may}, in the notation of \eqref{sym cyc}, have $m=1$ and/or \emph{may} consist entirely of double points modulo $2n$.

Every class of infinite cycles is finite, has an even number of cycles, and has a representatve $(\cdots\,a_1\,\,\,a_m\,\,\,a_1+2kn\,\cdots)$ with $k>0$, written with the numbers $a_1,\ldots,a_m$ all distinct modulo $2n$.
The class contains $2k$ cycles, the cycles \linebreak $(\cdots\,a_1+2\ell n\,\,\,a_m+2\ell n\,\,\,a_1+2(k+\ell)n\,\cdots)$ for $\ell=0,\ldots,k-1$ and their negatives.
A class of infinite cycles is \newword{flat} if it contains exactly two cycles (so that $k=1$).

%

We define a statistic $\vr$ on $\Stildejes$ that will serve as the rank function on $[1,c]_{T\cup L}$:
\begin{align*}
\vr(\pi)=(n-1)&-(\text{the number of classes of nonsymmetric finite cycles of }\pi)\\ 
&-\frac12(\text{the number of classes of symmetric cycles of }\pi)\\
&+\frac12(\text{the number of classes of tiny cycles in }\pi).
\end{align*}

\begin{prop}\label{interval rank}
If $\pi\in[1,c]_{T\cup L}$, then $\ell_{T\cup L}(\pi)=\vr(\pi)$.
\end{prop}

The bulk of the proof of Proposition~\ref{interval rank} is accomplished by the following lemma.
We will also use details from the proof of the lemma later in the paper.

\begin{lemma}\label{only one}
If $\pi\in\Stildejes$ and $\tau\in T\cup L$, then $\vr(\tau\pi)\ge\vr(\pi)-1$.  
\end{lemma}
\begin{proof}
The proof of this lemma is long and tedious but completely elementary.
We check every case of how $\tau$ can relate to the cycle structure of $\pi$.

The first four numbered cases below are where $\tau$ is a reflection, and the fifth case is where $\tau$ is a loop.
When $\tau$ is a reflection, it is $\tau=(\!(a\,\,\,b)\!)_{2n}$ with $a\not\equiv b\mod{2n}$, and the four cases are the four possibilities for what kind of cycle $a$ is in, namely nonsymmetric finite, symmetric, tiny, or infinite.
There are also subcases giving the possibilities for~$b$.
Up to the symmetry of swapping $a$ and $b$, we may as well assume that $a$ appears weakly earlier than $b$ on this list of four types of cycles.

\medskip

\noindent\textbf{Case 1.}
The element $a$ is in a nonsymmetric finite cycle $(a_1\,\cdots\,a_m)$, with $a_1=a$.

\smallskip

\noindent\textit{Case 1a.}
The element $b$ is in the same cycle as $a$, say $b=a_i$ for $2\le i\le m$.
Then $\tau\pi$ has nonsymmetric finite cycles $(a_1\,\cdots\,a_{i-1})(a_i\,\cdots\,a_m)$ and no other cycles of $\pi$ are changed, so $\vr(\tau\pi)=\vr(\pi)-1$.  

\smallskip

\noindent\textit{Case 1b.}
The element $b$ is in a cycle that is a mod-$2n$ translate of the cycle containing $a$, say $b=a_i+2kn$ for $2\le i\le m$ and $k\neq0$.
Then $\tau\pi$ has infinite cycles $(\cdots\,a_1\,\cdots\,a_{i-1}\,\,\,a_1-2kn\,\cdots)$ and $(\cdots\,a_i\,\cdots\,a_m\,\,\,a_i+2kn\,\cdots)$.
A class of nonsymmetric finite cycles has been destroyed by $\tau$, and two classes of infinite cycles have been created.
Thus $\vr(\tau\pi)=\vr(\pi)+1$.

\smallskip

\noindent\textit{Case 1c.}
The element $b$ is in the negative of the cycle containing $a$ or in a cycle that is a mod-$2n$ translate of the negative of the cycle containing $a$, say $b=-a_i+2kn$ for $2\le i\le m$.
Then $\tau\pi$ has cycles $(a_1\,\cdots\,a_{i-1}\,\,\,-a_1+2kn\,\cdots\,-a_{i-1}+2kn)$ and $(a_i\,\cdots\,a_m\,\,\,-a_i+2kn\,\cdots\,-a_m+2kn)$.
These cycles could be both symmetric or one symmetric and one tiny.
In any case, they are in different classes.
Multiplying by $\tau$ creates these two classes, destroys a class of nonsymmetric finite cycles, and changes no other cycles, so $\vr(\tau\pi)=\vr(\pi)$ or $\vr(\tau\pi)=\vr(\pi)+1$.

\smallskip

\noindent\textit{Case 1d.}
The element $b$ is in a nonsymmetric finite cycle not in the class of $(a_1\,\cdots\,a_m)$, say $b=b_1$ in the cycle $(b_1\,\cdots\,b_p)$.
Then $\tau\pi$ has a nonsymmetric finite cycle $(a_1\,\cdots\,a_m\,\,\,b_1\,\cdots\,b_p)$, so $\vr(\tau\pi)=\vr(\pi)+1$.

\smallskip

\noindent\textit{Case 1e.}
There is a symmetric cycle $(b_1\,\cdots\,b_p\,\,\,-b_1+2kn\,\cdots\,-b_p+2kn)$ with $b_1=b$.
Then $\tau\pi$ has a symmetric cycle 
\[(a_1\,\cdots\,a_m\,\,\,b_1\,\cdots\,b_p\,\,\,-a_1+2kn\,\cdots\,-a_m+2kn\,\,\,-b_1+2kn\,\cdots\,-b_p+2kn).\]
A class of nonsymmetric finite cycles was destroyed and the number of symmetric cycles is unchanged, so $\vr(\tau\pi)=\vr(\pi)+1$.

\smallskip

\noindent\textit{Case 1f.}
There is a tiny cycle containing $b$.
Arguing as in Case~1e with $p=1$, we see that a class of nonsymmetric finite cycles and a class of tiny cycles has been destroyed, while a symmetric cycle has been created.
Thus $\vr(\tau\pi)=\vr(\pi)$.

\smallskip

\noindent\textit{Case 1g.}
There is an infinite cycle $(\cdots\,b_1\,\cdots\,b_p\,\,\,b_1+2kn\,\cdots)$ with $b_1=b$.
Then $\tau\pi$ has an infinite cycle $(\cdots\,a_1\,\cdots\,a_m\,\,\,b_1\,\cdots\,b_p\,\,\,a_1+2kn\,\cdots)$, so $\vr(\tau\pi)=\vr(\pi)+1$.

\medskip

\noindent\textbf{Case 2.}
The element $a$ is in a symmetric cycle of $\pi$.
Up to adding the same multiple of $2n$ to both $a$ and $b$, we can take $a$ to be in a principal symmetric cycle $(a_1\,\cdots\,a_m\,\,\,-a_1+2kn\,\cdots\,-a_m+2kn)$ with $a_1=a$ and $k=0$ or~$1$.

\smallskip

\noindent\textit{Case 2a.}
The element $b$ is in the same cycle as $a$.
Up to swapping $a$ and $b$, we can assume that $b=a_i$ for $2\le i\le m$.
Multiplying by $\tau$ creates a class of nonsymmetric finite cycles containing $(a_1\,\cdots\,a_{i-1})$ and $(-a_1+2kn\,\cdots\,-a_{i-1}+2kn)$ and a cycle $(a_i\,\cdots\,a_m\,\,\,-a_i+2kn\,\cdots\,{-a_m+2kn})$, which could be symmetric or tiny.
(It is tiny if and only if $i=m$ and $a_i=a_m$ is one of the entries of the principal tiny cycle matching $(a_1\,\cdots\,a_m\,\,\,-a_1+2kn\,\cdots\,-a_m+2kn)$, in which case, the tiny cycle produced is that principal tiny cycle.)
If the cycle is not tiny, then $\vr(\tau\pi)=\vr(\pi)-1$ and if if cycle is tiny, then $\vr(\tau\pi)=\vr(\pi)$.

\smallskip

\noindent\textit{Case 2b.}
The element $b$ is in a cycle in the same class as the cycle containing~$a$.
Up to swapping $a$ and $b$, we can assume that $b=a_i+2qn$ for $2\le i\le m$ and $q\neq0$.
Then $\tau\pi$ has infinite cycles $(\!(\cdots\,a_1\,\cdots\,a_{i-1}\,\,\,a_1-2qn\,\cdots)\!)$ and cycles 
\[(\!(a_i\,\cdots\,a_m\,\,\,-a_i+2(k-q)n\,\cdots\,-a_m+2(k-q)n)\!)\] that are either symmetric or tiny.
Thus $\vr(\tau\pi)=\vr(\pi)$ or $\vr(\tau\pi)=\vr(\pi)+1$.

\smallskip

\noindent\textit{Case 2c.}
The element $b$ is in a symmetric cycle not in the class of the cycle containing~$a$, say $(b_1\,\cdots\,b_p\,\,\,-b_1+2qn\,\cdots\,-b_p+2qn)$ with $b_1=b$.
The action of~$\tau$ destroys two classes of symmetric finite cycles and creates a new cycle containing the sequence $a_1\,\cdots\,a_m\,\,\,-b_1+2kn\,\cdots\,-b_p+2kn\,\,\,\,a_1+2(k-q)n$.
If $k=q$, the new cycle is nonsymmetric finite and $\vr(\tau\pi)=\vr(\pi)$.
If $k\neq q$, the new cycle is infinite and $\vr(\tau\pi)=\vr(\pi)+1$.

\smallskip

\noindent\textit{Case 2d.}
The element $b$ is in a class of upper tiny cycles.
The cycle is $(b\,\,\,{-b+4qn})$, where $i$ is one of the numbers on the upper double point, $b=i+2qn$, and $q\in\integers$.
The action of $\tau$ destroys a class of tiny cycles and a class of symmetric cycles and creates a new cycle containing the sequence $a_1\,\cdots\,a_m\,\,\,{-b+2kn}\,\,\,\,{a_1+2(k-2q)n}$.
Recall that $k\in\set{0,1}$.
If $k=2q$, then $k=q=0$ and the new cycle is nonsymmetric finite, so $\vr(\tau\pi)=\vr(\pi)-1$.
This is if and only if $b$ is in the principal tiny cycle that matches $(a_1\,\cdots\,a_m\,\,\,{-a_1+2kn}\,\cdots\,{-a_m+2kn})$, specifically both are upper.
If $k\neq 2q$, then the new cycle is infinite and $\vr(\tau\pi)=\vr(\pi)$.

\smallskip

\noindent\textit{Case 2e.}
The element $b$ is in a class of lower tiny cycles.
The principal lower cycle is $(c\,\,-c+2n)$, where $j$ is the positive number labeling the lower double point and~$c$ is either $j$ or $-j+2n$.
The cycle containing $b$ is $(b\,\,\,\,{-b+(4q+2)n})$ for $q\in\integers$ and $b=c+2qn$.
The action of $\tau$ destroys a class of tiny cycles and a class of symmetric cycles and creates a new cycle containing $a_1\,\cdots\,a_m\,\,\,-b+2kn\,\,\,a_1+2(k-2q-1)n$.
Again, $k\in\set{0,1}$.
If $k=2q+1$, then $k=1$ and $q=0$ and the new cycle is nonsymmetric finite, so $\vr(\tau\pi)=\vr(\pi)-1$.
This is if and only if $b$ is in the principal tiny cycle that matches $(a_1\,\cdots\,a_m\,\,\,-a_1+2kn\,\cdots\,-a_m+2kn)$, specifically with both lower.
If $k\neq2q+1$, then the new cycle is infinite and $\vr(\tau\pi)=\vr(\pi)$.

\smallskip

\noindent\textit{Case 2f.}
There is an infinite cycle $(\cdots\,b_1\,\cdots\,b_p\,\,\,b_1+2qn\,\cdots)$ with $b_1=b$.
Then $\tau\pi$ has a symmetric cycle containing $a_1\,\cdots\,a_m\,\,\,-b_1+2kn\,\cdots\,-b_p+2kn\,\,\,-a_1+2(k-q)n$.
Thus a class of infinite cycles is destroyed by the action of $\tau$ but the number of classes of symmetric cycles remains the same, so $\vr(\tau\pi)=\vr(\pi)$.


\medskip

\noindent\textbf{Case 3.}
The element $a$ is in a tiny cycle of $\pi$.
Up to adding the same multiple of~$2n$ to $a$ and $b$, we can take $a$ to be in a principal tiny cycle $(a\,\,\,-a+2kn)$ for $k=0$ or~$1$.
If $k=0$, then $a$ is the positive or negative number on the upper double point.
If $k=1$, then~$a$ is $j$ or $-j+2n$, where $j$ is the positive number on the lower double point.
Also, $b$ is not in a tiny cycle of the same class, because $a\not\equiv\pm b\mod{2n}$.

\smallskip

\noindent\textit{Case 3a.}
The element $b$ is in a tiny cycle in the other class.
If $k=0$ ($a$ upper and $b$ lower), then as in Case~2e, the cycle containing $b$ is $(b\,\,\,\,{-b+(4q+2)n})$.
If $k=1$ ($a$ lower and $b$ upper), then as in Case~2d, the cycle containing $b$ is $(b\,\,\,\,{-b+4qn})$.
Concisely, the cycle containing $b$ is $(b\,\,\,-b+(4q+2-2k)n)$.
The action of $\tau$ destroys both classes of tiny cycles and creates an infinite cycle $(\cdots\,a\,\,\,-b+2kn\,\,\,a-2(2q+1-2k)n\,\cdots)$, which is flat if and only if $q=k$ or $q=k-1$.
Thus $\vr(\tau\pi)=\vr(\pi)-1$.

\smallskip

\noindent\textit{Case 3b.}
There is an infinite cycle $(\cdots\,b_1\,\cdots\,b_p\,\,\,b_1+2qn\,\cdots)$ with $b_1=b$.
The action of $\tau$ destroys the tiny cycle and infinite cycle and creates a symmetric cycle containing $a\,\,\,-b_1+2kn\,\cdots\,-b_p+2kn\,\,\,-a+2(k-q)n$, so ${\vr(\tau\pi)=\vr(\pi)-1}$.

\medskip

\noindent\textbf{Case 4.}
There is an infinite cycle $(\cdots\,a_1\,\cdots\,a_m\,\,\,a_1+2kn\,\cdots)$ with $a_1=a$.

\smallskip

\noindent\textit{Case 4a.}
The element $b$ is in the same infinite cycle.
Thus $b=a_i+2qkn$ for some $i=2,\ldots,m$ and integer $q$.
Then $\tau\pi$ has a cycle containing the sequence $a_1\,\cdots\,a_{i-1}\,\,\,{a_1-2qkn}$ (nonsymmetric finite if and only if $q=0$ and otherwise infinite) and a cycle containing the sequence $a_i\,\cdots\,a_m\,\,\,a_i+2(q+1)kn$ (nonsymmetric finite if and only if $q=-1$ and otherwise infinite).
If $q\in\set{0,-1}$, then the action of $\tau$ creates a class of nonsymmetric finite cycles and replaces a class of infinite cycles by another class of infinite cycles, so $\vr(\tau\pi)=\vr(\pi)-1$.
(The new infinite cycles are flat if and only if the original infinite cycles are flat.)
If $q\not\in\set{-1,0}$ then the action of $\tau$ replaces the infinite cycle by two infinite cycles, so $\vr(\tau\pi)=\vr(\pi)$.

\smallskip

\noindent\textit{Case 4b.}
The element $b$ is in a mod-$2n$ translate of the cycle containing~$a$.
Then $|k|>1$ and $b=a_i+2qn$ for some $i=2,\ldots,m$  and $q\not\equiv0\mod k$.
The action of~$\tau$ destroys the infinite cycle and creates infinite cycles ${(\cdots\,a_1\,\cdots\,a_{i-1}\,\,\,a_1-qn\,\cdots)}$ and $(\cdots\,a_i\,\cdots\,a_m\,\,\,a_i+(k+q)n\,\cdots)$, so ${\vr(\tau\pi)=\vr(\pi)}$.

\smallskip

\noindent\textit{Case 4c.}
The element $b$ is in the negative of the cycle containing $a$ or in a mod-$2n$ translate of the negative of the cycle containing $a$.
Then $b=-a_i+2qn$ for $i=2,\ldots,m$ and $q\in\integers$, and $\tau\pi$ has a cycle $(a_1\,\cdots\,a_{i-1}\,\,\,-a_1+2qn\,\cdots\,-a_{i-1}+2qn)$ and a cycle $(a_i\,\cdots\,a_m\,\,\,-a_i+2(q+k)n\,\cdots\,-a_m+2(q+k)n)$,
each of which can be symmetric or tiny.
If neither is tiny, then $\vr(\tau\pi)=\vr(\pi)-1$.
If one or both is tiny, then $\vr(\tau\pi)=\vr(\pi)$ or $\vr(\tau\pi)=\vr(\pi)+1$.

\smallskip

\noindent\textit{Case 4d.}
The element $b$ is in an infinite cycle $(\cdots\,b_1\,\cdots\,b_p\,\,\,b_1+2qn\,\cdots)$, with~${b_1=b}$, not in the class of $(\cdots\,a_1\,\cdots\,a_m\,\,\,a_1+2kn\,\cdots)$.
Then $\tau\pi$ has a cycle containing the sequence $a_1\,\cdots\,a_m\,\,\,b_1+2kn\,\cdots\,b_p+2kn\,\,\,a_1+2(k+q)n$.
If $q=-k$, this is nonsymmetric finite and $\vr(\tau\pi)=\vr(\pi)-1$.
If $q\neq-k$, this is infinite and $\vr(\tau\pi)=\vr(\pi)$.

\medskip

\noindent\textbf{Case 5.}
The element $\tau$ is a loop $\ell_a=(\!(\cdots\,a\,\,\,a+2n\,\cdots)\!)$ for $a\in\set{\pm1,\ldots,\pm(n-1)}$.

\smallskip

\noindent\textit{Case 5a.}
The element $a$ is in a nonsymmetric finite cycle $(a_1\,\cdots\,a_m)$ with $a_1=a$.
Then $\tau\pi$ has an 
infinite cycle $(\cdots\,a_1\,\cdots\,a_m\,\,\,a_1+2n\,\cdots)$, so $\vr(\tau\pi)=\vr(\pi)+1$.

\smallskip

\noindent\textit{Case 5b.}
The element $a$ is in a cycle $(a_1\,\cdots\,a_m\,\,\,{-a_1+2kn}\,\cdots\,{-a_m+2kn})$ that is symmetric and has $a_1=a$.
In $\tau\pi$, this cycle is replaced by a symmetric cycle $(a_1\,\cdots\,a_m\,\,\,{-a_1+2(k-1)n}\,\cdots\,{-a_m+2(k-1)n})$, so $\vr(\tau\pi)=\vr(\pi)$.

\smallskip

\noindent\textit{Case 5c.}
The element $a$ is in a tiny cycle of $\pi$.   
Since $a\in\set{\pm1,\ldots,\pm(n-1)}$, there are three possibilities for what the tiny cycle is:
an upper tiny cycle $(a\,\,\,-a)$, a lower tiny cycle $(a\,\,\,-a+2n)$ with $a>0$, or a lower tiny cycle $(a\,\,\,-a-2n)$ with $a<0$.
In the first case, $\tau\pi$ has a cycle $(a\,\,\,-a-2n)$, in the second, it has a cycle $(a\,\,\,-a)$, and in the third, it has $(a\,\,\,-a-4n)$.
In every case, the action of $\tau$ destroys a class of tiny cycles and creates a class of symmetric cycles, so $\vr(\tau\pi)=\vr(\pi)-1$.

\medskip

\noindent\textit{Case 5d.}
There is an infinite cycle $(\cdots\,a_1\,\cdots\,a_m\,\,\,a_1+2kn\,\cdots)$ with $a_1=a$.
Then $\tau\pi$ has a cycle containing $a_1\,\cdots\,a_m\,\,\,a_1+2(k+1)n$.
If $k=-1$, then the action of $\tau$ changes a flat infinite cycle into a class of nonsymmetric finite cycles, so $\vr(\tau\pi)=\vr(\pi)-1$.  
Otherwise, $\vr(\tau\pi)=\vr(\pi)$. 
\end{proof}

\begin{proof}[Proof of Proposition~\ref{interval rank}]
Lemma~\ref{d annulus label} says that $c$ has a class of infinite cycles and two classes of tiny cycles.
Therefore, $\vr(c)=n$.
Since also $\vr(1)=0$, Lemma~\ref{only one} implies that a reduced word for $c$ in the alphabet $T\cup L$ has at least $n$ letters.
But the defining word for $c$ as a product of simple reflections is, in particular, a word for $c$ as a product of $n$ elements of $T\cup L$.
We see that a reduced word for $c$ in the alphabet $T\cup L$ has exactly $n$ letters.
If $\pi\in[1,c]_{T\cup L}$, then there exists a reduced word $\tau_1\cdots\tau_n$ for $c$ in the alphabet $T\cup L$ and an index $k\in\set{0,\ldots,n}$ such that $\pi=\tau_{k+1}\cdots\tau_n$.
Lemma~\ref{only one} also implies that $\vr(\pi)=n-k=\ell_{T\cup L}(\pi)$.
\end{proof}

\subsection{Detailed cycle structures}\label{struct sec}
As a tool in proving the validity of the eventual combinatorial model for $[1,c]_{T\cup L}$, we list, and give shorthand names to, all of the cycle structures that a permutation $\pi\in[1,c]_{T\cup L}$ can have.
We also list and name some cycle structures that we will show that $\pi$ cannot have.
List~\ref{structures} shows the structures that can occur, while List~\ref{bad structures} shows structures that can't occur.

We next list and name all of the cases in the proof of Lemma~\ref{only one} where $\vr$ decreases by~$1$. 
This appears as List~\ref{moves} where we also indicate the case number from Lemma~\ref{only one}.
In this list, $\pi\in\Stildejes$ and $\tau\in T\cup L$.
If $\tau\in T$, then $a$ and $b$ are (convenient mod-$2n$ representatives of) integers such that $\tau=(\!(a\,\,\,b)\!)_{2n}$.
If $\tau\in L$, then $a$ is an integer in $\set{\pm1,\ldots,\pm(n-1)}$ such that $\tau=(\!(\cdots\,a\,\,\,a+2n\,\cdots)\!)$.




\newcommand{\Inf}{\mathbf{Inf}}
\newcommand{\NonflatInf}{\mathbf{NonflatInf}}
\renewcommand{\Tiny}{\mathbf{Tiny}}
\newcommand{\Sym}{\mathbf{Sym}}
\newcommand{\NonSym}{\mathbf{NonSym}}

\begin{listing}
\caption{Cycle structures in $[1,c]_{T\cup L}$.}
\label{structures}
\begin{minipage}{350pt}
$\Inf^1\Tiny^2\NonSym^k$.
One class of \emph{flat} infinite cycles, both classes of tiny cycles and $k$ classes of nonsymmetric finite cycles.

\bigskip

$\Inf^2\NonSym^k$.
Two classes of \emph{flat} infinite cycles and $k$ classes of nonsymmetric finite cycles.

\bigskip

$\Inf^1\NonSym^k$.
One class of \emph{flat} infinite cycles and $k$ classes of nonsymmetric finite cycles.

\bigskip

$\Tiny^1\Sym^1\NonSym^k$.
One class of symmetric cycles, a \emph{matching} class of tiny cycles, and $k$ classes of nonsymmetric finite cycles.

\bigskip

$\Tiny^2\Sym^2\NonSym^k$.
Both classes of tiny cycles, two classes of symmetric cycles \emph{not matching each other} (and thus each matching one of the tiny cycles), and $k$ classes of nonsymmetric finite cycles.

\bigskip

$\Inf^1\Tiny^1\Sym^1\NonSym^k$.
One class of \emph{flat} infinite cycles, one class of symmetric cycles, a \emph{matching} class of tiny cycles, and $k$ classes of nonsymmetric finite cycles.

\bigskip

$\Tiny^2\NonSym^k$.
Both classes of tiny cycles and $k$ classes of nonsymmetric finite cycles.

\bigskip

$\NonSym^k$.
Only nonsymmetric finite cycles, in $k$ classes.

\bigskip

\end{minipage}
\end{listing}

\begin{listing}
\caption{Some cycle structures that can't occur in $[1,c]_{T\cup L}$.}
\label{bad structures}
\begin{minipage}{350pt}
$\Inf^1\NonflatInf^1\NonSym^k$.
One class of flat infinite cycles, one class of \emph{nonflat} infinite cycles, and $k$ classes of nonsymmetric finite cycles.

\bigskip

$\Sym^2\NonSym^k$.
Two classes of symmetric cycles, \emph{not matched}, and $k$ classes of nonsymmetric finite cycles.

\bigskip

$\Inf^1\Sym^2\NonSym^k$.
One class of \emph{flat} infinite cycles, two classes of symmetric cycles, \emph{not matched}, and $k$ classes of nonsymmetric finite cycles.

\bigskip

$\NonflatInf^1\Sym^2\NonSym^k$.
One class of \emph{nonflat} infinite cycles, two classes of symmetric cycles, \emph{not matched}, and $k$ classes of nonsymmetric finite cycles.

\bigskip

$\Tiny^1\Sym^3\NonSym^k$.
One class of tiny cycles, three classes of symmetric cycles, and $k$ classes of nonsymmetric finite cycles.

\bigskip

$\Inf^1\Sym^3\NonSym^k$.
One class of infinite cycles, three classes of symmetric cycles, and $k$ classes of nonsymmetric finite cycles.

\bigskip

$\NonflatInf^1\NonSym^k$.
One class of \emph{nonflat} infinite cycles and $k$ classes of nonsymmetric finite cycles.

\bigskip

\end{minipage}
\end{listing}

\newcommand{\SplitNonSym}{\mathit{SplitNonSym}}
\newcommand{\SplitSym}{\mathit{SplitSym}}
\newcommand{\CombineSymTiny}{\mathit{CombineSymTiny}}
\newcommand{\CombineTiny}{\mathit{CombineTiny}}
\newcommand{\CombineTinyInf}{\mathit{CombineTinyInf}}
\newcommand{\SplitInf}{\mathit{SplitInf}}
\newcommand{\CombineInfAndNeg}{\mathit{CombineInfAndNeg}}
\newcommand{\CombineInfInf}{\mathit{CombineInfInf}}
\newcommand{\EnlargeTiny}{\mathit{EnlargeTiny}}
\newcommand{\InfToNonSym}{\mathit{InfToNonSym}}

\begin{listing}
\caption{Moves taking $\pi$ to $\tau\pi$ with $\vr(\tau\pi)=\vr(\pi)-1$.}
\label{moves}
\begin{minipage}{350pt}
$\SplitNonSym$ (Case 1a).
There is a nonsymmetric finite cycle of $\pi$ containing $a$ and $b$.
In $\tau\pi$, this cycle is split into two nonsymmetric finite cycles.

\medskip

$\SplitSym$ (Case 2a).
There is a symmetric cycle (without loss of generality principal) of $\pi$ containing $a$ and $b$, naming $a$ and $b$ so that the distance forward in the cycle from $a$ to $b$ is less than the distance from $b$ to $a$, and ruling out the case where both $b$ is one of the entries of the principal tiny cycle matching the symmetric cycle and also $b$ precedes $-a$ or $-a+2n$ in the symmetric cycles.
In $\tau\pi$, this cycle is split into a symmetric cycle and a pair of nonsymmetric finite cycles in the same class.

\medskip

$\CombineSymTiny$ (Case 2de).
The element $a$ is in a principal symmetric cycle of~$\pi$ and $b$ is in the matching principal tiny cycle.
The elements of this symmetric and tiny cycle in $\pi$ form a pair of nonsymmetric finite cycles in $\tau\pi$ in the same class.

\medskip

$\CombineTiny$ (Case 3a).
The element $a$ is in a tiny cycle of $\pi$, while $b$ is in a tiny cycle in the other class.
In $\tau\pi$, there is a class of infinite cycles (negatives of each other, flat or nonflat), each alternating between a mod-$2n$ translate of an upper double point and a mod-$2n$ translate of a lower tiny point.

\medskip

$\CombineTinyInf$ (Case 3b).
The element $a$ is in a tiny cycle of $\pi$ and $b$ is in an infinite cycle.
In $\tau\pi$, in place of these cycles, there is a class of symmetric cycles.
If $|q|=1$ in Case 3b (e.g.\ if $\pi$ has no nonflat infinite cycles), then the symmetric cycle that is created does not match the tiny cycle that is destroyed.

\medskip

$\SplitInf$ (Case 4a).
The elements $a$ and $b$ are in the same infinite cycle, with~$b$ between $a$ and the next element of the cycle that is congruent to $a$ modulo $2n$ (before or after $a$ in the cycle).
In $\tau\pi$, this infinite cycle is replaced by a nonsymmetric finite cycle and an infinite cycle, which is flat if and only if the cycle in~$\pi$ is flat.

\medskip

$\CombineInfAndNeg$ (Case 4c).
The element $a$ is $a_1$ in an infinite cycle of the form $(\cdots\,a_1\,\cdots\,a_m\,\,\,a_1+2kn\,\cdots)$ and $b$ is the entry $-a_i+2qn$ in the negative of that cycle or in a mod-$2n$ translate of the negative of that cycle, ruling out the case where $i=2$ and $(a_1\,\,\,-a_1+2qn)$ is tiny and the case where $i=m$ and $(a_m\,\,\,-a_m+2(q+k)m)$ is tiny.
In $\tau\pi$, this class of infinite cycles is replaced by two classes of symmetric cycles that are matched if and only if $k$ is even. 

\medskip

$\CombineInfInf$ (Case 4d).
There is an infinite cycle $(\cdots\,a_1\,\cdots\,a_m\,\,\,a_1+2kn\,\cdots)$ with $k\in\integers$ having $a_1=a$ and in infinite cycle $(\cdots\,b_1\,\cdots\,b_p\,\,\,b_1-2kn\,\cdots)$ in a different class but with the same $k$, having $b_1=b$.
In $\tau\pi$, these two classes of infinite cycles are replaced by a single class of nonsymmetric finite cycles.

\medskip

$\EnlargeTiny$ (Case 5c).
The element $\tau$ is a loop $\ell_a=(\!(\cdots\,a\,\,\,a+2n\,\cdots)\!)$ with $a\in\set{\pm1,\ldots,\pm(n-1)}$ and $a$ also is in a tiny cycle of $\pi$.   
Then in $\tau\pi$ this class of tiny cycles is replaced by a class of ($2$-element) symmetric cycles that don't match the tiny cycle in $\pi$ that contains $a$.

\medskip

$\InfToNonSym$ (Case 5d).
The element $\tau$ is a loop $\ell_a=(\!(\cdots\,a\,\,\,a+2n\,\cdots)\!)$ for $a\in\set{\pm1,\ldots,\pm(n-1)}$ and there is an infinite cycle $(\cdots\,a_1\,\cdots\,a_m\,\,\,a_1-2n\,\cdots)$ in $\pi$ with $a_1=a$.
In $\tau\pi$, this class of infinite cycles is replaced by a class of nonsymmetric finite cycles.
\end{minipage}
\end{listing}

\clearpage

\begin{prop}\label{cycle types}
If $\pi\in[1,c]_{T\cup L}$, then $\pi$ has one of the cycle structures on List~\ref{structures}. 
\end{prop}
\begin{proof}
By Proposition~\ref{interval rank}, every element of $[1,c]_{T\cup L}$ is obtained from $c$ by multiplying on the left by a sequence of elements $\tau\in T\cup L$ with each multiplication decreasing $\vr$ by $1$.
The cycle structure of $c$ is $\Inf^1\Tiny^2\NonSym^0$ and the cycle structure of $1$ is $\NonSym^{n-1}$.
We prove the proposition by checking two assertions:
If $\pi$ has cycle structure on List~\ref{structures} and $\pi'$ is obtained by one of the moves on List~\ref{moves}, then $\pi'$ has cycle structure on List~\ref{structures} or List~\ref{bad structures}.
If $\pi$ has cycle structure on List~\ref{bad structures}, then no sequence of moves on List~\ref{moves} leads to the cycle structure $\NonSym^{n-1}$.

To begin, for every cycle structure on List~\ref{structures}, we list the possible moves and the resulting cycle structure. 

\medskip

\noindent
$\Inf^1\Tiny^2\NonSym^k$.
\[\begin{array}{|rcl|}\hline
\Inf^1\Tiny^2\NonSym^k &\xrightarrow{\SplitNonSym} & \Inf^1\Tiny^2\NonSym^{k+1}\\\hline
\Inf^1\Tiny^2\NonSym^k &\xrightarrow{\CombineTiny} & \Inf^2\NonSym^k\\\hline
\Inf^1\Tiny^2\NonSym^k &\xrightarrow{\CombineTiny} & \Inf^1\NonflatInf^1\NonSym^k\\\hline
\Inf^1\Tiny^2\NonSym^k &\xrightarrow{\CombineTinyInf} & \Tiny^1\Sym^1\NonSym^k\\\hline
\Inf^1\Tiny^2\NonSym^k &\xrightarrow{\SplitInf} & \Inf^1\Tiny^2\NonSym^{k+1}\\\hline
\Inf^1\Tiny^2\NonSym^k &\xrightarrow{\CombineInfAndNeg} &\Tiny^2\Sym^2\NonSym^k\\\hline
\Inf^1\Tiny^2\NonSym^k &\xrightarrow{\EnlargeTiny} &\Inf^1\Tiny^1\Sym^1\NonSym^k\\\hline
\Inf^1\Tiny^2\NonSym^k &\xrightarrow{\InfToNonSym} &\Tiny^2\NonSym^{k+1}\\\hline
\end{array}\]

\bigskip

\noindent
$\Inf^2\NonSym^k$
\[\begin{array}{|rcl|}\hline
\Inf^2\NonSym^k &\xrightarrow{\SplitNonSym} & \Inf^2\NonSym^{k+1}\\\hline
\Inf^2\NonSym^k &\xrightarrow{\SplitInf} & \Inf^2\NonSym^{k+1}\\\hline
\Inf^2\NonSym^k &\xrightarrow{\CombineInfAndNeg} & \Inf^1\Sym^2\NonSym^k\\\hline
\Inf^2\NonSym^k &\xrightarrow{\CombineInfInf} & \NonSym^{k+1}\\\hline
\Inf^2\NonSym^k &\xrightarrow{\InfToNonSym} & \Inf^1\NonSym^{k+1}\\\hline
\end{array}\]

\bigskip

\noindent
$\Inf^1\NonSym^k$
\[\begin{array}{|rcl|}\hline
\Inf^1\NonSym^k &\xrightarrow{\SplitNonSym} & \Inf^1\NonSym^{k+1}\\\hline
\Inf^1\NonSym^k &\xrightarrow{\SplitInf} & \Inf^1\NonSym^{k+1}\\\hline
\Inf^1\NonSym^k &\xrightarrow{\CombineInfAndNeg} & \Sym^2\NonSym^k\\\hline
\Inf^1\NonSym^k &\xrightarrow{\InfToNonSym} & \NonSym^{k+1}\\\hline
\end{array}\]

\bigskip

\noindent
$\Tiny^1\Sym^1\NonSym^k$.
\[\begin{array}{|rcl|}\hline
\Tiny^1\Sym^1\NonSym^k &\xrightarrow{\SplitNonSym} & \Tiny^1\Sym^1\NonSym^{k+1}\\\hline
\Tiny^1\Sym^1\NonSym^k &\xrightarrow{\SplitSym} & \Tiny^1\Sym^1\NonSym^{k+1}\\\hline
\Tiny^1\Sym^1\NonSym^k &\xrightarrow{\CombineSymTiny} & \NonSym^{k+1}\\\hline
\Tiny^1\Sym^1\NonSym^k &\xrightarrow{\EnlargeTiny} & \Sym^2\NonSym^k\\\hline
\end{array}\]

\bigskip

\noindent
$\Tiny^2\Sym^2\NonSym^k$.
\[\begin{array}{|rcl|}\hline
\Tiny^2\Sym^2\NonSym^k &\xrightarrow{\SplitNonSym} & \Tiny^2\Sym^2\NonSym^{k+1}\\\hline
\Tiny^2\Sym^2\NonSym^k &\xrightarrow{\SplitSym} & \Tiny^2\Sym^2\NonSym^{k+1}\\\hline
\Tiny^2\Sym^2\NonSym^k &\xrightarrow{\CombineSymTiny} & \Tiny^1\Sym^1\NonSym^{k+1}\\\hline
\Tiny^2\Sym^2\NonSym^k &\xrightarrow{\CombineTiny} & \Inf^1\Sym^2\NonSym^k\\\hline
\Tiny^2\Sym^2\NonSym^k &\xrightarrow{\CombineTiny} & \NonflatInf^1\Sym^2\NonSym^k\\\hline
\Tiny^2\Sym^2\NonSym^k &\xrightarrow{\EnlargeTiny} & \Tiny^1\Sym^3\NonSym^k\\\hline
\end{array}\]

\bigskip

\noindent
$\Inf^1\Tiny^1\Sym^1\NonSym^k$.
\[\begin{array}{|rcl|}\hline
\Inf^1\Tiny^1\Sym^1\NonSym^k &\xrightarrow{\SplitNonSym} & \Inf^1\Tiny^1\Sym^1\NonSym^{k+1}\\\hline
\Inf^1\Tiny^1\Sym^1\NonSym^k &\xrightarrow{\SplitSym} & \Inf^1\Tiny^1\Sym^1\NonSym^{k+1}\\\hline
\Inf^1\Tiny^1\Sym^1\NonSym^k &\xrightarrow{\CombineSymTiny} & \Inf^1\NonSym^{k+1}\\\hline
\Inf^1\Tiny^1\Sym^1\NonSym^k &\xrightarrow{\CombineTinyInf} & \Sym^2\NonSym^k\\\hline
\Inf^1\Tiny^1\Sym^1\NonSym^k &\xrightarrow{\SplitInf} & \Inf^1\Tiny^1\Sym^1\NonSym^{k+1}\\\hline
\Inf^1\Tiny^1\Sym^1\NonSym^k &\xrightarrow{\CombineInfAndNeg} & \Tiny^1\Sym^3\NonSym^k\\\hline
\Inf^1\Tiny^1\Sym^1\NonSym^k &\xrightarrow{\EnlargeTiny} & \Inf^1\Sym^3\NonSym^k\\\hline
\Inf^1\Tiny^1\Sym^1\NonSym^k &\xrightarrow{\InfToNonSym} & \Tiny^1\Sym^1\NonSym^{k+1}\\\hline
\end{array}\]

\bigskip

\noindent
$\Tiny^2\NonSym^k$.
\[\begin{array}{|rcl|}\hline
\Tiny^2\NonSym^k &\xrightarrow{\SplitNonSym} & \Tiny^2\NonSym^{k+1}\\\hline
\Tiny^2\NonSym^k &\xrightarrow{\CombineTiny} & \Inf^1\NonSym^k\\\hline
\Tiny^2\NonSym^k &\xrightarrow{\CombineTiny} & \NonflatInf^1\NonSym^k\\\hline
\Tiny^2\NonSym^k &\xrightarrow{\EnlargeTiny} & \Tiny^1\Sym^1\NonSym^k\\\hline
\end{array}\]

\bigskip

\noindent
$\NonSym^k$.
\[\begin{array}{|rcl|}\hline
\NonSym^k &\xrightarrow{\SplitNonSym} & \NonSym^{k+1}\\\hline
\end{array}\]

\bigskip

We see that, for every cycle structure on List~\ref{structures}, the result of applying a move from List~\ref{moves} is a cycle structure on List~\ref{structures} or List~\ref{bad structures}.

Next, for every cycle structure on List~\ref{bad structures}, we argue that no sequence of moves can lead to the cycle structure $\NonSym^{n-1}$.
First of all, there is only one move on List~\ref{moves} that destroys a class of symmetric cycles, namely $\CombineSymTiny$, but this move also destroys a tiny cycle.
Since also no move creates a tiny cycle, we see if $\pi$ has more classes of symmetric cycles than classes of tiny cycles, no sequence of moves from List~\ref{moves} can destroy all of the symmetric cycles of $\pi$.

All but two cycle structures on List~\ref{bad structures} have more classes of symmetric cycles than tiny cycles.
The remaining two are the structures $\Inf^1\NonflatInf^1\NonSym^k$ and $\NonflatInf^1\NonSym^k$.
We check that there is no way to destroy the class of nonflat infinite cycles without creating a symmetric cycle (and thus having more symmetric cycles than tiny cycles). 
The only moves that can be done at all are $\SplitInf$, which can't get rid of a nonflat cycle, $\CombineInfAndNeg$, which creates two classes of symmetric cycles and no tiny cycles, and $\InfToNonSym$, which only applies to $\Inf^1\NonflatInf^1\NonSym^k$ and changes it to $\NonflatInf^1\NonSym^{k+1}$.
(The move $\CombineInfInf$ can't be applied to $\Inf^1\NonflatInf^1\NonSym^k$ precisely because one class of infinite cycles is flat and one is not.)
%
\end{proof}

We also need a version of Proposition~\ref{cycle types} for the smaller interval $[1,c]_T$.

\begin{prop}\label{cycle types T}
If $\pi\in[1,c]_T$, then the cycle structure of $\pi$ is one of the following from List~\ref{structures}.
\begin{itemize}
\item $\Inf^1\Tiny^2\NonSym^k$, 
\item $\Inf^2\NonSym^k$, 
\item $\Tiny^1\Sym^1\NonSym^k$, 
\item $\Tiny^2\Sym^2\NonSym^k$, or 
\item $\NonSym^k$ 
\end{itemize}
\end{prop}

%
%
%
%
%
%
%
%
%
%

\begin{proof}
Since $[1,c]_T$ is a subset of $[1,c]_{T\cup L}$, Proposition~\ref{cycle types} says that every permutation in $[1,c]_T$ has a cycle structure described in List~\ref{structures}.
We look back at the proof of Proposition~\ref{cycle types}, disallowing moves $\EnlargeTiny$ and $\InfToNonSym$ because they correspond to multiplication by loops.
We see that only the five cycle structures listed in the proposition can be reached from $c$ by moves described in List~\ref{moves}, disallowing $\EnlargeTiny$ and $\InfToNonSym$.
\end{proof}

\section{Symmetric NC partitions of an annulus with two double points}\label{nc D sec}
The combinatorial model for $[1,c]_{T\cup L}$ in type~$\afftype{D}_{n-1}$ takes place in $D$, the annulus with two double points defined in Section~\ref{sym ann doub sec}.  
This model is a special case of a definition of \newword{symmetric noncrossing partitions of a marked surface with double points} given in \cite[Section~3]{surfnc}.

\subsection{Noncrossing partitions}\label{nc D defs sec}
Recall from Section~\ref{sym ann doub sec} that the choice of a Coxeter element is equivalent to a placement of the numbers $\set{\pm1,\ldots,\pm(n-1)}$ on the boundary and double points of $D$.
The points with these labels are \newword{numbered points}, and more specifically \newword{inner points}, \newword{outer points}, or \newword{double points}.
(``Numbered points'' here are ``marked points'' in~\cite{surfnc}, because the numbering of points is not important there.)
Recall that the symmetry $\phi$ maps each numbered point $i$ to $-i$, including mapping each double point to the opposite-signed double point at the same location.

A \newword{boundary segment} of $D$ is a curve contained in the boundary of $D$, connecting a numbered point to a numbered point, but intersecting no numbered points other than its endpoints.
The endpoints of a boundary segment coincide if and only if there is only one marked point on the component of the boundary containing that segment.
(This is the case $n=4$.)

We will consider certain subsets of $D$ up to a notion of \newword{symmetric ambient isotopy}, meaning that they are related by a homeomorphism from $D$ to itself that fixes the boundary pointwise, fixes each double point, commutes with $\phi$ and is homotopic to the identity by a homotopy that fixes the boundary and fixes each double point at every step and commutes with $\phi$ at every step.


An \newword{arc} in $D$ is a non-oriented curve $\alpha$ in $D$, having endpoints at numbered points and satisfying these requirements:
\begin{itemize}
\item
$\alpha$ does not intersect itself or $\phi(\alpha)$ except possibly at endpoints.
\item
$\alpha$ does not intersect the double points or the boundary of $D$, except at its endpoints.
\item
$\alpha$ does not bound a monogon in $D$ (even if its endpoints are two different double points at the same location in $D$).
\item 
$\alpha$ does not combine with a boundary segment to bound a digon in $D$. 
\end{itemize}
A curve $\alpha$ is an arc if and only if $\phi(\alpha)$ is an arc, and we call $\alpha,\phi(\alpha)$ a \newword{symmetric pair of arcs}.
Symmetric pairs of arcs are considered up to symmetric ambient isotopy and up to swapping an arc $\alpha$ with~$\phi(\alpha)$.

Noncrossing partitions of $D$ are defined below as collections of embedded blocks.
As we define embedded blocks, we will refer forward to Figure~\ref{nc ex fig} for examples.
An \newword{embedded block} in $D$ is a closed subset $E$ with either $E\cap\phi(E)=\emptyset$ or $E=\phi(E)$ and fitting one of the descriptions below.
(We emphasize that the condition $E\cap\phi(E)=\emptyset$ allows the possibility that $E$ contains a double point and $\phi(E)$ contains the opposite double point at the same location.)
If $E\cap\phi(E)=\emptyset$, then $(E,\phi(E))$ is a \newword{symmetric pair of blocks} and must be one of the following:
\begin{itemize}
\item
a \newword{symmetric pair of trivial blocks}, meaning a pair of numbered points related by $\phi$ and therefore numbered $\pm i$ for some $i$ (Figure~\ref{nc ex fig}, top-left $\pm4$ and top-middle $\pm1$);
\item
a \newword{symmetric pair of curve blocks}, meaning a symmetric pair of arcs or symmetric pair of boundary segments, each having two distinct endpoints (Figure~\ref{nc ex fig}, top-middle and -right, bottom-right);
\item
a symmetric pair of \newword{disk blocks}, closed disks each of whose boundaries is a union of arcs and/or boundary segments of $D$ (Figure~\ref{nc ex fig}, second row-left, the two blocks containing $6$ and $-6$);
\item
a symmetric pair of \newword{non-dangling annular blocks}, each a closed annulus with each of its  boundary components a union of arcs and/or boundary segments of $D$ (Figure~\ref{nc ex fig}, top-left); or
\item
a symmetric pair of \newword{dangling annular blocks}, each a closed annulus with one of its boundary components a closed curve in the interior of D not containing the double points and one of its boundary components a union of arcs and/or boundary segments of $D$ (Figure~\ref{nc ex fig}, third row-middle and -right and fourth row-middle).
\end{itemize}


\noindent
If $E=\phi(E)$, then $E$ is a \newword{symmetric block} and must be one of the following:
\begin{itemize}
\item
a \newword{symmetric disk block}, a closed disk in $D$ whose whose boundary is a union of arcs and/or boundary segments of $D$, including as a special case a \newword{stitched disk block}, which has two boundary points at opposite double points in the same location, thus appearing to have two boundary points identified (Figure~\ref{nc ex fig} top-middle, second row-left and -middle, bottom-middle);
\item
a \newword{symmetric non-dangling annular block}, an annulus with each of its  boundary components a union of arcs and/or boundary segments of $D$ (Figure~\ref{nc ex fig} third row-left); or
\item
a \newword{symmetric dangling annular block}, an annulus with each of its boundary components a closed curve in the interior of D not containing the double points (Figure~\ref{nc ex fig} bottom-left).
\end{itemize}

The first two cases (symmetric pairs of trivial blocks or curve blocks) are also called pairs of \newword{degenerate disk blocks}. 

\begin{figure}
\scalebox{0.47}{\includegraphics{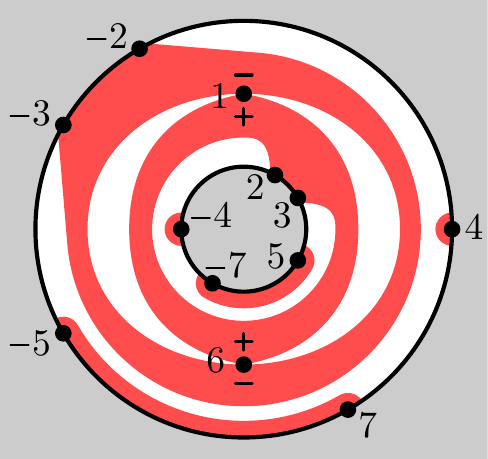}}
\quad 
\scalebox{0.47}{\includegraphics{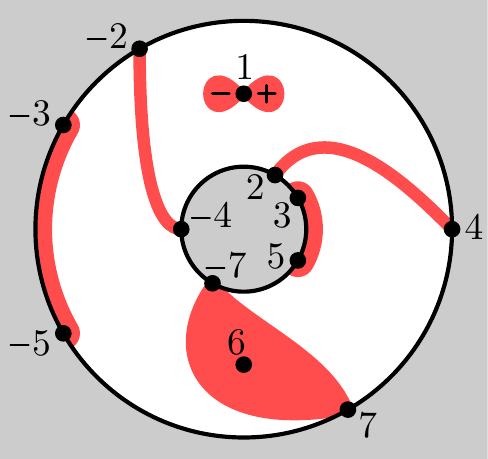}}
\quad 
\scalebox{0.47}{\includegraphics{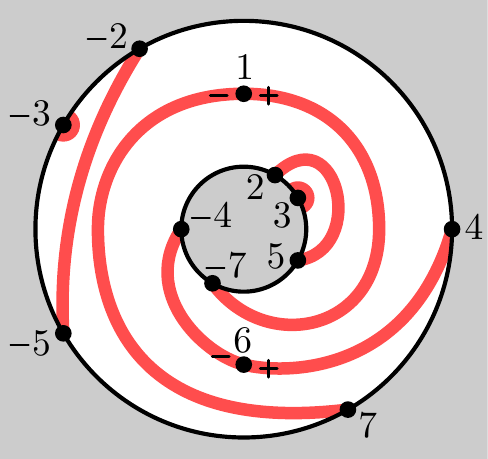}}\\[12pt]
\scalebox{0.47}{\includegraphics{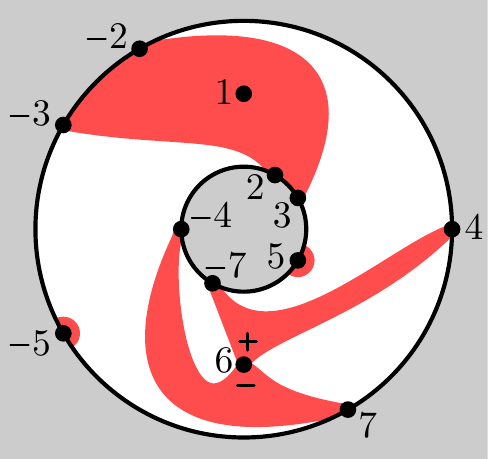}}
\quad 
\scalebox{0.47}{\includegraphics{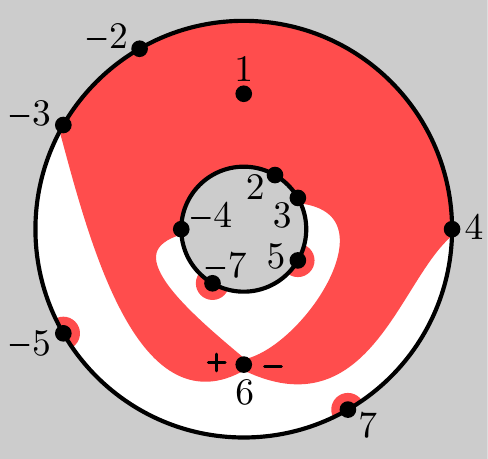}}
\quad 
\scalebox{0.47}{\includegraphics{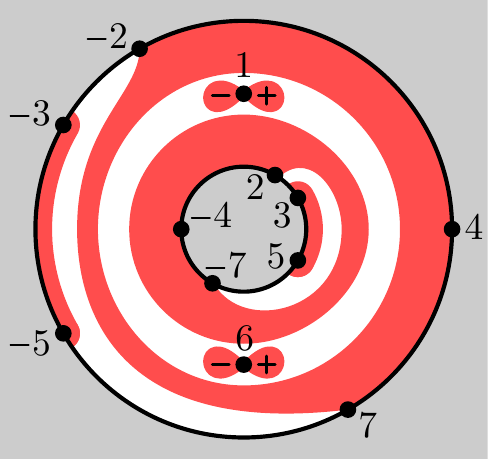}}\\[12pt]
\scalebox{0.47}{\includegraphics{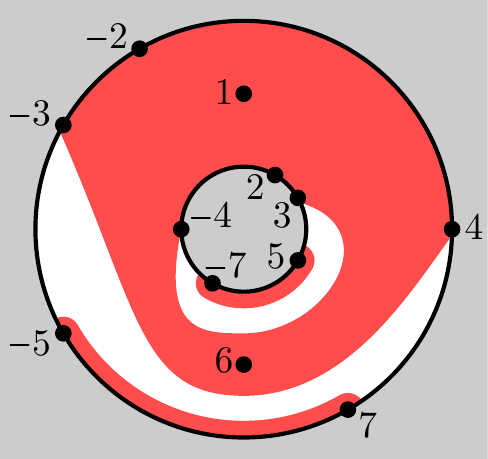}}
\quad 
\scalebox{0.47}{\includegraphics{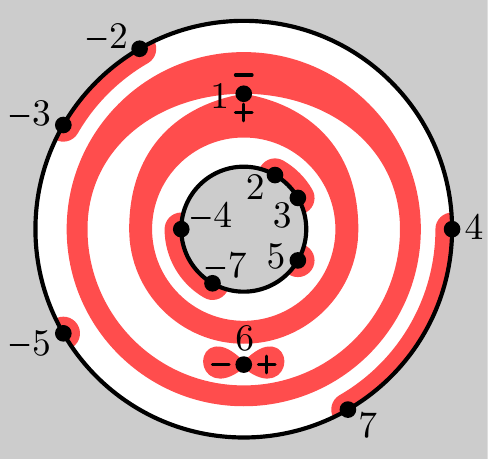}}
\quad 
\scalebox{0.47}{\includegraphics{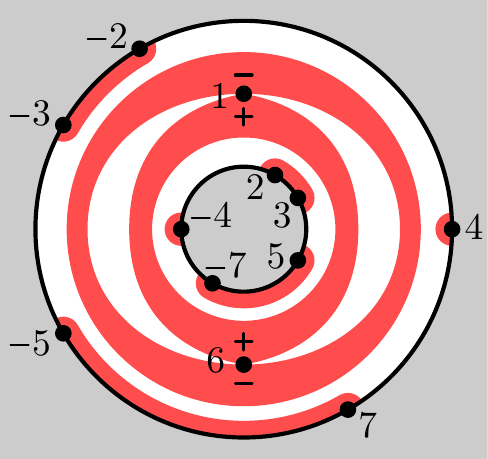}}\\[12pt]
\scalebox{0.47}{\includegraphics{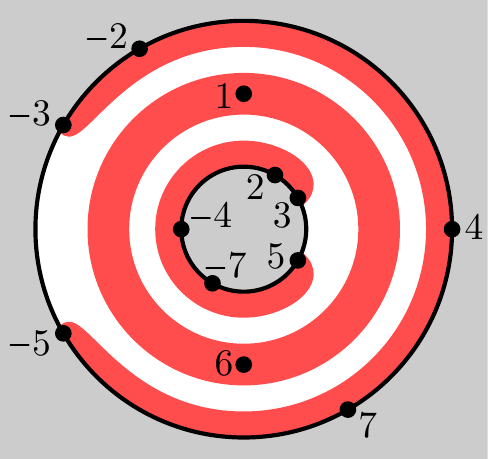}}
\quad 
\scalebox{0.47}{\includegraphics{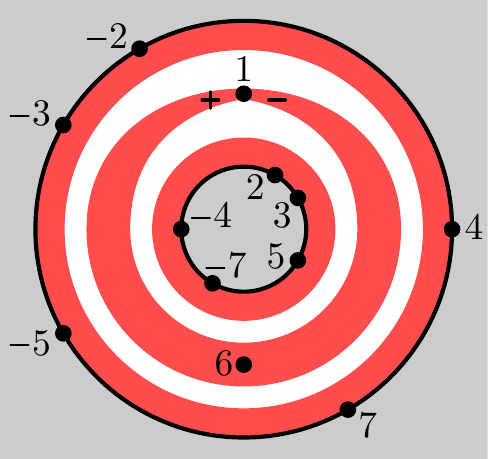}}
\quad 
\scalebox{0.47}{\includegraphics{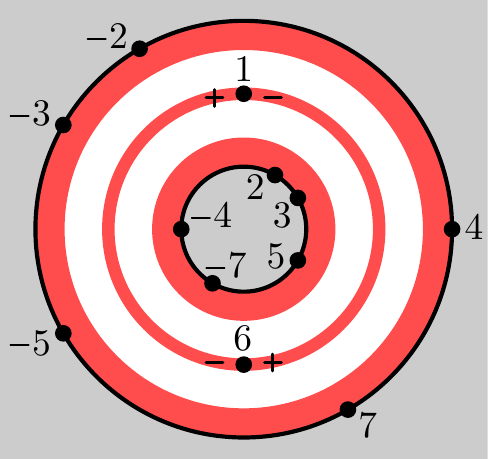}}
\caption{Symmetric noncrossing partitions of an annulus with two double points}
\label{nc ex fig}
\end{figure}

Double points may lie on the boundary of a block or in the interior of a block.
If a double point is in the interior of a block, then the other double point in the same location is also.
Embedded blocks with double points on their boundaries are distinguished by which double point (of a pair of double points at the same location) is on their boundary.
In pictures, the double point in a block is indicated by $+$ or $-$.
Stitched disks are distinguished by which double point is at which point along the boundary, again indicated in pictures by $+$ or $-$ in the appropriate place.
A symmetric disk block contains one pair of double points in its interior.
A symmetric annular block contains both pairs of double points in its interior.

A \newword{(symmetric) noncrossing partition} of $D$ is a collection $\P$ of disjoint embedded blocks such that every numbered point is in some block of $\P$, such that the action of $\phi$ permutes the blocks of $\P$, and having at most two annular blocks.
We emphasize that the requirement of disjoint embedded blocks allows the possibility that a block $E$ of $\P$ contains a double point and $\phi(E)\neq E$, so that $\phi(E)$ is also a block of $\P$ and contains the opposite double point at the same location.
The requirement that $\P$ has at most two annular blocks rules out the possibility that $\P$ has two dangling annular blocks that share (up to isotopy) a component of their boundary.
Noncrossing partitions are considered up to symmetric isotopy.
An \newword{embedding} of $\P$ is a specific symmetric-isotopy representative of~$\P$.

\begin{example}\label{D exs}
Figure~\ref{nc ex fig} shows several examples of symmetric noncrossing partitions of an annulus with two double points.
This is the case where $n=8$ and $c=s_3s_6s_2s_0s_1s_5s_7s_4$.
(Compare Figure~\ref{sym ann fig}.)
To make the pictures more legible, degenerate disk blocks are shown with some thickness.
\end{example}


We partially order the symmetric noncrossing partitions of an annulus with two double points as follows:
$\P \leq \Q$ if and only if there exists an embedding of $\P$ and an embedding of $\Q$ with each block of $\P$ contained in some block of $\Q$.
We write $\tNCDc$ for the set of noncrossing symmetric partitions of an annulus with two double points with this partial order and $\tNCDcircc$ for the subposet induced by the set of noncrossing partitions with no dangling annular blocks.  

The following theorem originally appeared in~\cite{BThesis}.
We quote it here as \cite[Theorem~4.5]{surfnc}, which is a special case of \cite[Theorem~3.18]{surfnc}.

\begin{theorem}\label{D tilde main}
The poset $\tNCDc$ of symmetric noncrossing partitions of an annulus with $n-3$ marked points on each boundary and two pairs of double points is graded, with rank function given by $n-1$ minus the number of symmetric pairs of distinct non-annular blocks plus the number of symmetric annular blocks.
\end{theorem}

\begin{remark}\label{error}
In \cite{BThesis} and an early version of~\cite{surfnc}, it was erroneously claimed that $\tNCDc$ is a lattice.
This is false, and a counterexample is given in \cite[Example~3.17]{surfnc}.
In Section~\ref{lat sec}, we complete $\tNCDc$ to a lattice, with guidance from~\cite{McSul}.
\end{remark}

The proof of Theorem~\ref{D tilde main} in \cite{surfnc} begins with the \newword{curve set} $\curve(\P)$ of a noncrossing partition $\P$, the set of all arcs and boundary segments contained in blocks of $\P$.
We mention some facts that can be found in~\cite{surfnc}:
A noncrossing partition is determined by its curve set.
The partial order on $\tNCDc$ is containment of curve sets.
The set $\curve(\P)$ can be viewed as a set of noncrossing partitions; $\P$ is the join of that set.

Cover relations among noncrossing partitions are described explicitly in \cite[Section~3]{surfnc} in the generality of symmetric marked surfaces with double points.
We summarize briefly here.
Examples specific to $\tNCDc$ are in \cite[Figure~11]{surfnc} and \cite[Figure~12]{surfnc}.
Those figures also contain examples in a symmetric annulus with only one double point (see Section~\ref{aff type b}) that are still useful for understanding $\tNCDc$.

A \newword{simple connector} for $\P\in\tNCDc$ is an arc or boundary segment $\kappa\not\in\curve(\P)$ such that there exist blocks $E$ and $E'$ of $\P$ and an isotopy representative of $\kappa$ that starts inside $E$, leaves $E$, \emph{intersects no other block of $\P$}, then enters $E'$ and stays there until it ends there.
We allow $E=E'$, but if so, since $\kappa$ is not in $\curve(\P)$, it does not have a representative that is entirely contained in $E$.
A \newword{simple symmetric pair of connectors} for $\P$ is a symmetric pair $\kappa,\phi(\kappa)$ of simple connectors for $\P$, but ruling out one possibility:
We disallow the case where~$\kappa$ and $\phi(\kappa)$ combine with blocks of $\P$ to bound a disk in $D$ containing a pair of double points that form a symmetric pair of trivial blocks of $\P$. 
Examples of pairs~$\kappa,\phi(\kappa)$ ruled out by this definition are in the bottom row of \cite[Figure~9]{surfnc}.

If $\kappa,\phi(\kappa)$ is a simple symmetric pair of connectors for $\P$, the \newword{augmentation of~$\P$ along $\kappa,\phi(\kappa)$}, written $\P\cup\kappa\cup\phi(\kappa)$ is the smallest symmetric noncrossing partition greater than $\P$ and having $\kappa$ and $\phi(\kappa)$ in its curve set.
Supposing that $\kappa$ connects blocks $E$ and $E'$ of $\P$, the basic construction of $\P\cup\kappa\cup\phi(\kappa)$ is to replace $E$ and $E'$ with the union of $E$ and $E'$ with $\kappa$ and $\phi(\kappa)$, possibly ``thickening'' the union to make it a disk or annulus.
More details are given in \cite[Section~3]{surfnc}.

The following proposition is a special case of \cite[Proposition~3.27]{surfnc}.

\begin{prop}\label{conn edge cov d}
If $\P,\Q\in\tNCDc$, then $\P\covered \Q$ if and only if there exists a simple symmetric pair of connectors $\kappa,\phi(\kappa)$ for $\P$ such that $\Q=\P\cup\kappa\cup\phi(\kappa)$.
\end{prop}

\subsection{Isomorphism}
We saw in Section~\ref{proj sec} how the projection of a $\Stildedes$-orbit to the Coxeter plane leads to the construction of the symmetric annulus $D$ with two double points.
We now show that symmetric noncrossing partitions of $D$ model the intervals $[1,c]_T$ and $[1,c]_{T\cup L}$.
Specifically, we define a map $\perm^D:\tNCDc\to\Stildejes$ and show that $\perm^D$ is an isomorphism from $\tNCDc$ to $[1,c]_{T \cup L}$, and that its restriction is an isomorphism from $\tNCDcircc$ to $[1,c]_T$.   

We define $\perm^D$ by reading the cycle notation of a permutation from the embedded blocks of a noncrossing partition, as explained in more detail below.
We obtain a class of nonsymmetric finite cycles from each symmetric pair of disk blocks and a class of symmetric cycles from each symmetric disk block (stitched or not).
We obtain a class of tiny cycles from each pair of double points that is in the interior of a block.
We obtain a class of infinite cycles from each symmetric non-dangling annular block and from each symmetric pair of dangling annular blocks.
We obtain two classes of infinite cycles from each symmetric pair of non-dangling annular blocks.
For the purposes of the map, we consider a degenerate block to be a small disk with one or two numbered points on its boundary.

The \newword{dateline} is a radial segment from the inner boundary to the outer boundary, passing through the lower double point and shown as a vertical gray segment in the right picture of Figure~\ref{sym ann fig}.
Crossing the dateline in the clockwise direction (right to left in pictures) is considered a \newword{positive} crossing, while crossing the dateline in the counterclockwise direction (left to right in pictures) is a \newword{negative} crossing.
Since the lower double point is on the dateline, there is a convention for crossing the dateline as we move to or from that double point:
The \emph{positive} double point is infinitesimally \emph{counterclockwise} (right) of the dateline, while the \emph{negative} double point is infinitesimally \emph{clockwise} (left) of the dateline.
Thus, moving clockwise (i.e.\ from the right) to the positive lower double point or moving counterclockwise (i.e.\ to the right) from the positive lower double point \emph{does not} cross the dateline, but moving counterclockwise (i.e.\ from the left) to the positive lower double point or moving clockwise (i.e.\ to the left) from the positive lower double point \emph{does} cross the dateline.
The negative lower double point behaves in the opposite way.

We now explain how to read cycles from blocks.
Suppose $\P\in\tNCDc$.
Given a block $E$ of $\P$ and a component of the boundary of~$E$, we obtain a cycle by reading along the boundary with the interior of $E$ on the right and recording the numbered points, but keeping track of the total number~$w$ of times we cross the date line (positive crossings minus negative crossings) and adding $2wn$ to the numbered points as we record them.
We return to the starting point with $w=0$ if and only if $E$ is a disk.
In this case, if we record numbered points $a_1,\ldots,a_k$ and then return to $a_1$, the permutation $\perm^D(\P)$ has cycles $(a_1\,\cdots\,a_k)_{2n}$.
If $E$ is not a symmetric disk, then we obtain the negatives of these cycles from the boundary of $\phi(E)$, so that $\perm^D(\P)$ has cycles $(\!(a_1\,\cdots\,a_k)\!)_{2n}$.
If we return to the starting point with $w\neq0$, then $w=\pm1$ and $E$ is an annular block, with $w=+1$ if and only if we are reading the outer boundary of $E$.
In this case, we read a cycle of the form $(\cdots\,a_1\,\cdots\,a_k\,\,\,a_1\pm n\,\cdots)$, where $a_1\pm n$ is $a_1+wn$.
We also read the negative of this cycle from $\phi(E)$ (whether $\phi(E)=E$ or not), so $\perm^D(\P)$ has cycles $(\!(\cdots\,a_1\,\cdots\,a_k\,\,\,a_1\pm n\,\cdots)\!)$.
When $E$ is a dangling annular block, we only read a cycle from the component of its boundary that contains numbered points.

If the upper double point is numbered $\pm i$ and is in the interior of some block of~$\P$, then we record tiny cycles $(i\,\,-i)_{2n}$.
If the lower double point is numbered $\pm j$ for $j>1$ and is in the interior of some block of~$\P$, then we record tiny cycles $(j\,\,-j+2n)_{2n}$.


\begin{theorem}\label{isom d}
The map $\perm^D:\tNCDc\to\Stildejes$ is an isomorphism from $\tNCDc$ to the interval $[1,c]_{T\cup L}$ in $\Stildejes$.   
It restricts to an isomorphism from $\tNCDcircc$ to the interval $[1,c]_T$ in $\Stildedes$.
\end{theorem}

\begin{example}
We apply the map $\perm^D$ to the noncrossing partitions pictured in Figure~\ref{nc ex fig}, labeling the first row as $\P_1$ through $\P_3$, the second row as $\P_4$ through $\P_6$, etc.
{\allowdisplaybreaks
\begin{align*}
\perm^D(\P_1)&=(\!(\cdots\,1\,\,\,6\,\,\,17\,\cdots)\!)_{16}(\!(\cdots\,3\,\,\,2\,\,\,-\!\!13\,\cdots)\!)_{16}(\!(4)\!)_{16}(\!(5\,\,\,9)\!)_{16}\\
\perm^D(\P_2)&=(\!(1)\!)_{16}(\!(2\,\,\,4)\!)_{16}(\!(3\,\,\,5)\!)_{16}(6\,\,\,10)_{16}(7\,\,\,9)_{16}\\
\perm^D(\P_3)&=(\!(1\,\,\,9)\!)_{16}(\!(2\,\,\,5)\!)_{16}(\!(3)\!)_{16}(\!(4\,\,\,6)\!)_{16}\\
\perm^D(\P_4)&=(1\,\,-\!1)_{16}(2\,\,-\!3\,\,-\!2\,\,\,3)_{16}(\!(4\,\,\,6\,\,\,9)\!)_{16}(\!(5)\!)_{16}\\
\perm^D(\P_5)&=(1\,\,-\!1)_{16}(2\,\,-\!4\,\,-\!10\,\,-\!3\,\,-\!2\,\,\,4\,\,\,10\,\,\,3)_{16}(\!(5)\!)_{16}(\!(7)\!)_{16}\\
\perm^D(\P_6)&=(\!(1)\!)_{16}(\!(\cdots\,2\,\,-\!4\,\,-\!\!7\,\,-\!\!14\,\cdots)\!)_{16}(\!(3\,\,\,5)\!)_{16}(\!(6)\!)_{16}\\
\perm^D(\P_7)&=(1\,\,-\!1)_{16}(\!(\cdots\,2\,\,-\!4\,\,-\!\!13\,\,-\!\!14\,\cdots)\!)_{16}(\!(5\,\,\,9)\!)_{16}(6\,\,\,10)_{16}\\
\perm^D(\P_8)&=(\!(\cdots\,1\,\,\,17\,\cdots)\!)_{16}(\!(2\,\,\,3)\!)_{16}(\!(4\,\,\,7)\!)_{16}(\!(5)\!)_{16}(\!(6)\!)_{16}\\
\perm^D(\P_9)&=(\!(\cdots\,1\,\,\,6\,\,\,17\,\cdots)\!)_{16}(\!(2\,\,\,3)\!)_{16}(\!(4)\!)_{16}(\!(5\,\,\,9)\!)_{16}\\
\perm^D(\P_{10})&=(1\,\,-\!1)_{16}(\!(2\,\,-\!4\,\,-\!\!7\,\,-\!\!11\,\,\,3)\!)_{16}(6\,\,\,10)_{16}\\
\perm^D(\P_{11})&=(1\,\,-\!17)_{16}(\!(\cdots\,2\,\,-\!4\,\,-\!\!7\,\,-\!\!11\,\,-\!13\,\,-\!14\,\cdots)\!)_{16}(6\,\,\,10)_{16}\\
\perm^D(\P_{12})&=(\!(1\,\,-\!6)\!)_{16}(\!(\cdots\,2\,\,-\!4\,\,-\!\!7\,\,-\!\!11\,\,-\!13\,\,-\!14\,\cdots)\!)_{16}
\end{align*}
}
\end{example}

To prove the first assertion of Theorem~\ref{isom d}, will prove the following three propositions in the next three sections (Sections~\ref{one to one d sec}, \ref{perm inv cov sec}, and~\ref{cov ref loop sec}).
The second assertion of Theorem~\ref{isom d} is proved in Section~\ref{circ sec}.

\begin{prop}\label{one to one d}
The map $\perm^D$ is one-to-one.
\end{prop}

\begin{prop}\label{perm inv cov}
Suppose $\sigma\covered\pi$ in $[1,c]_{T\cup L}$ and $\pi=\perm^D(\Q)$ for some ${\Q\in\tNCDc}$.
Then there exists $\P\in\tNCDc$ such that $\sigma=\perm^D(\P)$ and $\P\covered\Q$.
\end{prop}

\begin{prop}\label{cov ref loop}
Suppose $\P,\Q\in\tNCDc$ have $\P\covered\Q$.
Then there exists a reflection or loop $\tau\in T\cup L$ such that $\perm^D(\Q)=\tau\cdot\perm^D(\P)$.
\end{prop}


\begin{proof}[Proof of the first assertion of Theorem~\ref{isom d}, assuming Props.~\ref{one to one d}, \ref{perm inv cov},~\ref{cov ref loop}]
A maximal chain in $\tNCDc$ has the form $\P_0\covered\cdots\covered\P_n$, where $\P_0$ consists of trivial blocks and $\P_n$ has a single block (the entire annulus $D$).
Since $\perm(\P_0)$ is the identity $1$ and $\perm(\P_n)=c$, using Proposition~\ref{cov ref loop} to find $\tau_i\in T\cup L$ for each $\P_{i-1}\covered\P_i$, we write a word $\tau_1\cdots\tau_n$ for $c$.
This is a reduced word in the alphabet $T\cup L$ because $\ell_{T\cup L}=n$ by Proposition~\ref{interval rank}.
Since every $\P\in\tNCDc$ is on some maximal chain, it follows that $\perm^D$ maps $\tNCDc$ into $[1,c]_{T\cup L}$.
Furthermore, if $\P\covered\Q$, then this cover relation is on some maximal chain, so $\perm^D(\P)\covered_{T\cup L}\perm^D(\Q)$.
That is, $\perm^D$ is order-preserving.

Proposition~\ref{one to one d} says that $\perm^D$ is one-to-one.
Proposition~\ref{perm inv cov} and an easy induction (with base case $c$) shows that every element of $[1,c]_{T\cup L}$ is in the image of $\perm^D$.
We conclude that $\perm^D$ is a bijection from $\tNCDc$ to $[1,c]_{T\cup L}$.
Proposition~\ref{perm inv cov} also implies that the inverse map to $\perm^D$ is order-preserving.
Thus $\perm^D$ is an isomorphism.
\end{proof}

\subsection{Proof of Proposition~\ref{one to one d}}\label{one to one d sec}
We describe how to recover $\P$ from $\perm^D(\P)$.
Writing $\pi$ for $\perm^D(\P)$, we observe that $\pi$ has at most four infinite cycles.

Suppose $\pi$ has four infinite cycles.
Then $\P$ has a symmetric pair of non-dangling annular blocks, one with double and inner points and one with double and outer points.
(See, for example, the top-left picture of Figure~\ref{nc ex fig}.)
Two of the infinite cycles involve double points, one involves inner points and one involves outer points.
The infinite cycle with outer points determines the outer points in the outer annular block (the numbers strictly between $-n$ and $n$ that appear in the cycle).
The \emph{decreasing} infinite cycle involving double points determines the double points on the outer annular block in the same way.
Similarly, the infinite cycle involving inner points and the increasing infinite cycle involving outer points determine the inner annular block.

Suppose $\pi$ has two infinite cycles.
If the infinite cycles involve double points, then $\P$ has a symmetric pair of dangling annular blocks on those double points.
(For example, Figure~\ref{nc ex fig}, third row-right.)
The outer annular block is determined by the decreasing infinite cycle on double points and the inner annular block is determined by the increasing cycle.
Suppose instead that one infinite cycle involves inner points while the other involves outer points.
If $\pi$ has both classes of tiny cycles, then there is a symmetric non-dangling annular block containing these inner points and outer points (Figure~\ref{nc ex fig}, third row-left).
If $\pi$ does not have both classes of tiny cycles, then~$\P$ has a symmetric pair of dangling annular blocks, one on inner points and one on outer points (Figure~\ref{nc ex fig}, bottom-middle).
These annular blocks are determined from the infinite cycles as in earlier cases.

Suppose $\pi$ has no infinite cycles.
If $\pi$ has both classes of tiny cycles, but has no families of symmetric cycles, then $\P$ has a symmetric dangling annular block (Figure~\ref{nc ex fig}, bottom-left).
Otherwise, $\P$ has no annular blocks.

In every case, we have determined all of the annular blocks of $\P$ from $\perm^D(\P)$.
The disk blocks (including degenerate disks) are also determined by $\perm^D(\P)$, because its (non-tiny) finite cycles describe the boundaries of blocks.
\qed

\subsection{Proof of Proposition~\ref{perm inv cov}}\label{perm inv cov sec}
Proposition~\ref{cycle types} says that $\pi$ and $\sigma$ have cycle structures described in List~\ref{structures}.
Furthermore, $\sigma=\tau\cdot\pi$ where $\tau$ is as in List~\ref{moves} and we break into cases according to that list.
The cases are not numbered, but rather are labeled by names of moves and parenthetically by case numbers in Lemma~\ref{only one}.
The integers $a$ and $b$ (or just~$a$) are as in List~\ref{moves}.  
In each case, we describe how $\P$ is constructed from $\Q$.

\medskip

\noindent
\textbf{Case $\SplitNonSym$ (1a).}
The nonsymmetric finite cycle of $\pi$ containing $a$ and $b$ corresponds to a pair of nonsymmetric disks in~$\Q$.
This is split into two pairs of nonsymmetric disks (splitting each individual disk) to form~$\P$.

\medskip

\noindent
\textbf{Case $\SplitSym$ (2a).}
The symmetric cycle of $\pi$ containing $a$ and $b$ corresponds to a symmetric disk in $\Q$.
This is split into a smaller symmetric disk and a pair of nonsymmetric disks to form~$\P$.

\medskip

\noindent
\textbf{Case $\CombineSymTiny$ (2de).}
The symmetric cycle corresponds to a symmetric disk in $\Q$.
The matching tiny cycle comes from the double points $\pm d$ in the interior of the symmetric disk.
To form $\P$, the symmetric disk is split into a symmetric pair of nonsymmetric disks, one containing~$d$ and one containing~$-d$.

\medskip

\noindent
\textbf{Case $\CombineTiny$ (3a).}
Since $\pi$ has cycle structure described in List~\ref{structures}, that cycle structure is $\Inf^1\Tiny^2\NonSym^k$, $\Tiny^2\Sym^2\NonSym^k$, or $\Tiny^2\NonSym^k$.
But since $\sigma$ also has cycle structure described in List~\ref{structures}, we rule out the possibility that $\pi$ has cycle structure $\Tiny^2\Sym^2\NonSym^k$.
Thus $\pi$ has cycle structure $\Inf^1\Tiny^2\NonSym^k$ or $\Tiny^2\NonSym^k$.
Correspondingly, $\Q$ has a symmetric non-dangling or dangling annular block.
Since the infinite cycle in~$\sigma$ is flat, it has a class $(\!(\cdots\,a\,\,\,{-b+2kn}\,\,\,{a-2(2q+1-2k)n}\,\cdots)\!)$ of infinite cycles with $q=k$ or $q=k-1$, as explained in Case~3a of Lemma~\ref{only one}.
Here $a$ is in a principal tiny cycle matching $k$ and $b$ is in a tiny cycle $(b\,\,\,{-b+(4q+2-2k)n})$ in the other class.
Up to swapping $a$ and $b$, we can take $a$ to be upper, or in other words, $k=0$.
Up to negating both $a$ and $b$, we can take $a$ to be the positive upper double point~$i$.
Thus the class of infinite cycles is $(\!(\cdots\,i\,\,\,-b\,\,\,i-2(2q+1)n\,\cdots)\!)$.
There are four possibilities, because $q=0$ or $-1$ and because $b$ is associated to the positive or negative lower double point.
Writing $j$ for the positive lower double point, the possibilities are shown below, with the resulting classes of infinite cycles.
\[\small\begin{array}{c|c|c|c|c|c}
q	&(b\,\,\,{-b+(4q+2)n})		&b		&(\!(\cdots\,i\,\,\,-b\,\,\,i-2(2q+1)n\,\cdots)\!)	&\text{upper}&\text{lower}\\\hline
0	&(b\,\,\,{-b+2n})			&j		&(\!(\cdots\,i\,\,\,{-j}\,\,\,i-2n\,\cdots)\!)		&+		&-\\		
0	&(b\,\,\,{-b+2n})			&-j+2n	&(\!(\cdots\,i\,\,\,j-2n\,\,\,i-2n\,\cdots)\!)	&+		&+\\	
-1	&(b\,\,\,{-b-2n})			&-j		&(\!(\cdots\,i\,\,\,j\,\,\,i+2n\,\cdots)\!)		&-		&-\\	
-1	&(b\,\,\,{-b-2n})			&j-2n		&(\!(\cdots\,i\,\,\,{-j+2n}\,\,\,i+2n\,\cdots)\!)	&-		&+	
\end{array}\]
To form $\P$ in these four cases, we break the symmetric annular block in $\Q$ into a pair of nonsymmetric annular blocks each containing two double points in all four possible ways.
The last two columns of the table indicate the signs on upper and lower double points that are contained in the \emph{outer} annulus in $\P$.

\medskip

\noindent
\textbf{Case $\CombineTinyInf$ (3b).}
Since $\pi$ and $\sigma$ have cycle structures on List~\ref{structures} and are related by a $\CombineTinyInf$ move, $\pi$ has structure $\Inf^1\Tiny^2\NonSym^k$ and $\sigma$ has $\Tiny^1\Sym^1\NonSym^k$.
Thus $\Q$ has a non-dangling symmetric annular block.

Now $a$ is in a principal tiny cycle $(a\,\,\,-a+2kn)$ with $k=0$ or~$1$.
For~$i$ the positive upper double point and $j$ the positive lower double point, if $k=0$, then $a=\pm i$ and if $k=1$, then $a$ is $j$ or $-j+2n$.
Also, $b$ is $b_1$ in $(\cdots\,b_1\,\cdots\,b_p\,\,\,b_1+2qn\,\cdots)$, with $q=1$ or $-1$.
By passing from $(a\,\,\,b)$ to $({-a+2kn}\,\,\,{-b+2kn})$ if necessary, we can assume that $q=1$ (preserving the fact that $a$ is in a principal tiny cycle).
Then $\sigma$ has a symmetric cycle
$(a\,\,\,{-b_1+2kn}\,\cdots\,{-b_p+2kn}\,\,\,{-a+2(k-1)n}\,\,\,{b_1-2n}\,\cdots\,{b_p-2n})$.
By checking all four cases for $a$, we see that this symmetric cycle would be read by $\perm^D$ from a stitched disk:

If $k=0$, then the cycle is $(\pm i\,\,\,{-b_1}\,\cdots\,{-b_p}\,\,\,{\mp i-2n}\,\,\,{b_1-2n}\,\cdots\,{b_p-2n})$,
which is read from a stitched disk starting at $\pm i$, reading the inner points $-b_1,\ldots,-b_p$ and coming to $\mp i$ (with one positive crossing of the date line), then reading the outer points $b_1-2n,\ldots,b_p-2n$ before returning to $\pm i$ (undoing the positive crossing).

If $k=1$ and $a=j$, it is 
$(j\,\,\,{-b_1+2n}\,\cdots\,{-b_p+2n}\,\,\,{-j}\,\,\,{b_1-2n}\,\cdots\,{b_p-2n})$,
read from a stitched disk starting at $j$, reading inner points $-b_1+2n,\ldots,-b_p+2n$ and coming to $-j$ (with no net crossing of the date line), then reading the outer points $b_1-2n,\ldots,b_p-2n$ before returning to $j$ (againt no net crossing of the date line).

If $k=1$ and $a=-j+2n$, the cycle is 
\[(-j+2n\,\,\,{-b_1+2n}\,\cdots\,{-b_p+2n}\,\,\,{j-2n}\,\,\,{b_1-2n}\,\cdots\,{b_p-2n}),\]
which is read from a stitched disk starting at $-j+2n$, reading the inner points $-b_1+2n,\ldots,-b_p+2n$ and coming to $j$ (crossing the date line, in net, twice counterclockwise, once when leaving $-j+2n$ to the right and once when coming to $-j+2n$ from the left), and then reading the outer points $b_1-2n,\ldots,b_p-2n$ before returning to $-j+2n$ (having undone the two counterclockwise crossings).

In all of these cases, the stitched disk is contained in the symmetric non-dangling annular block of $\Q$ corresponding to the class $(\!(\cdots\,b_1\,\cdots\,b_p\,\,\,b_1+2qn\,\cdots)\!)$.
To form~$\P$, we replace this annular block with the appropriate stitched disk.

\medskip

\noindent
\textbf{Case $\SplitInf$ (4a).}
The infinite cycle of $\pi$ containing $a$ and $b$ corresponds to a boundary component of an annular block in $\Q$, either symmetric and non-dangling or part of a symmetric pair of annular blocks.
A nonsymmetric disk is split from the annular block in $\Q$ (and symmetrically, another nonsymmetric disk is split off of the opposite boundary or the other annulus) to form~$\P$.

\medskip

\noindent
\textbf{Case $\CombineInfAndNeg$ (4c).}
Since $\pi$ and $\sigma$ have structures on List~\ref{structures} and are related by a $\CombineInfAndNeg$ move, $\pi$ has structure $\Inf^1\Tiny^2\NonSym^k$ and~$\sigma$ has structure $\Tiny^2\Sym^2\NonSym^k$.
Thus $\Q$ has a symmetric non-dangling annular block.
Up to rewriting $(\!(a\,\,\,b)\!)_{2n}$ as $(\!(-a\,\,\,-b)\!)_{2n}$ (and because the infinite cycle is read from an annular block and thus is flat), we can assume that $a$ is $a_1$ in an infinite cycle $(\cdots\,a_1\,\cdots\,a_m\,\,\,a_1+2n\,\cdots)$ and $b=-a_i+2qn$ for some $i=2,\ldots,m$ and some $q\in\integers$.
In $\sigma$, there are cycles $(a_1\,\cdots\,a_{i-1}\,\,\,-a_1+2qn\,\cdots\,-a_{i-1}+2qn)$
and $(a_i\,\cdots\,a_m\,\,\,-a_i+2(q+1)n\,\cdots\,-a_m+2(q+1)n)$,
that are symmetric (because the case where they are tiny is disallowed but also because none of the entries $a_1,\ldots,a_m$ correspond to double points).
To form $\P$, the symmetric annular block in $\Q$ is split into two disjoint symmetric blocks, one upper and one lower.

\medskip

\noindent
\textbf{Case $\CombineInfInf$ (4d).}
The structure of $\pi$ is $\Inf^2\NonSym^k$ and the structure of $\sigma$ is $\NonSym^{k+1}$.
Since $\pi$ has no tiny cycles but two classes of flat infinite cycles, $\Q$ has two nonsymmetric non-dangling annular blocks.
Up to rewriting $(\!(a\,\,\,b)\!)_{2n}$ as $(\!(-a\,\,\,-b)\!)_{2n}$, we can take $a=a_1$ in an infinite cycle $(\cdots\,a_1\,\cdots\,a_m\,\,\,a_1+2n\,\cdots)$.
Therefore $b=b_1$ in an infinite cycle $(\cdots\,b_1\,\cdots\,b_p\,\,\,b_1-2n\,\cdots)$.
Since the $a$ and $b$ are in infinite cycles of different classes, they are the two cycles read from the two boundary components of the same annular block in $\Q$
In $\sigma$, these two classes of infinite cycles are replaced by a single class of nonsymmetric finite cycles $(\!(a_1\,\cdots\,a_m\,\,\,b_1+2n\,\cdots\,b_p+2n)\!)$.
This corresponds to a pair of nonsymmetric disk blocks, contained in the original pair of annular blocks.
We form $\P$ by replacing the annular blocks with the disks.

\medskip

\noindent
\textbf{Case $\EnlargeTiny$ (5c).}
Either $\pi$ has structure $\Inf^1\Tiny^2\NonSym^k$ and $\sigma$ has structure $\Inf^1\Tiny^1\Sym^1\NonSym^k$ or $\pi$ has structure $\Tiny^2\NonSym^k$ and $\sigma$ has structure $\Tiny^1\Sym^1\NonSym^k$.
In either case, $\tau=\ell_a=(\!(\cdots\,a\,\,\,a+2n\,\cdots)\!)$ for $a\in\set{\pm1,\ldots,\pm(n-1)}$.
Taking $i$ to be the positive upper double point and~$j$ to be the positive lower double point, there are four possibilities for cycles in $\sigma$ involving double points, shown here with the corresponding stitched disks. 

\smallskip

\begin{longtable}{|c|c|c|c|}\hline&&\\[-10pt]
tiny cycle in $\pi$	&$a$		&cycles in $\sigma$ involving double points&\\\hline
\raisebox{13pt}{$(a\,\,\,{-a})$}		&\raisebox{13pt}{$i$}		&\raisebox{13pt}{$(i\,\,\,-i-2n)_{2n}\cdot(j\,\,\,{-j+2n})_{2n}$}&\scalebox{0.9}{\includegraphics{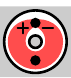}}\\\hline
\raisebox{13pt}{$(a\,\,\,{-a})$}		&\raisebox{13pt}{$-i$}	&\raisebox{13pt}{$(-i\,\,\,i-2n)_{2n}\cdot(j\,\,\,{-j+2n})_{2n}$}&\scalebox{0.9}{\includegraphics{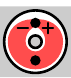}}\\\hline
\raisebox{13pt}{$(a\,\,\,{-a+2n})$}	&\raisebox{13pt}{$j$}		&\raisebox{13pt}{$(j\,\,\,-j)_{2n}\cdot(i\,\,\,-i)_{2n}$}&\scalebox{0.9}{\includegraphics{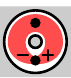}}\\\hline
\raisebox{13pt}{$(a\,\,\,{-a-2n})$	}	&\raisebox{13pt}{$-j$	}	&\raisebox{13pt}{$(-j\,\,\,j-4n)_{2n}\cdot(i\,\,\,-i)_{2n}$}&\scalebox{0.9}{\includegraphics{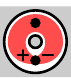}}\\\hline
\end{longtable}

\smallskip

If $\pi$ has structure $\Inf^1\Tiny^2\NonSym^k$ and $\sigma$ has $\Inf^1\Tiny^1\Sym^1\NonSym^k$, then $\Q$ has a non-dangling symmetric annular block.
To form $\P$, this annular block is replaced by a symmetric pair of dangling annular blocks, one on the outer points and one on the inner points, and one of the stitched disks shown above.
If $\pi$ has structure $\Tiny^2\NonSym^k$ and $\sigma$ has structure $\Tiny^1\Sym^1\NonSym^k$, then $\Q$ has a symmetric dangling annular block.
To form $\P$, this annular block is replaced by one of the stitched disks shown above.

\medskip

\noindent
\textbf{Case $\InfToNonSym$ (5d).}
This case can happen whenever $\pi$ has one or more classes of infinite cycles.
Again, $\tau$=$\ell_a=(\!(\cdots\,a\,\,\,a+2n\,\cdots)\!)$.
If $a$ is in a symmetric non-dangling annular block in $\Q$, then $\P$ is formed by cutting that annular block into a symmetric dangling annular block and a symmetric pair of dangling annular blocks (one on inner points and one on outer points) and then cutting the inner and outer annular blocks into disks.
If $a$ is in a nonsymmetric non-dangling annular block, then $\P$ is formed by cutting that block into two dangling annular blocks and cutting the piece containing $a$ into a disk (and then making the symmetric cuts).
If $a$ is in a dangling annular block, then $\P$ is formed by cutting that annular block into a disk.
\qed

\subsection{Proof of Proposition~\ref{cov ref loop}}\label{cov ref loop sec}
By Proposition~\ref{conn edge cov d}, there exists a simple symmetric pair of connectors $\kappa,\phi(\kappa)$ for $\P$ with $\Q=\P\cup\kappa\cup\phi(\kappa)$.
The curve $\kappa$ starts in a block $E$ of $\P$ and ends in a block $E'$ of $\P$, possibly with $E'=E$.
There are two possibilities for how $\kappa$ crosses the boundary of a block.
It passes either between two numbered points (possibly coinciding) on the boundary of the block or through the ``empty boundary component'' of a dangling annular block (the component of the boundary of the block that contains no numbered points). 

We rule out the possibility that $\kappa$ leaves $E$ through an empty boundary component and enters $E'$ through an empty boundary component.
Suppose to the contrary.
Then $E\neq E'$, because if $E=E'$, either $\kappa$ is isotopic to a curve that never leaves $E$ or $\kappa$ intersects $\phi(\kappa)$.
Since $\P$ has at most $2$ annular blocks, the empty boundary components of $E$ and $E'$ can't coincide up to isotopy.
Thus $E$ and $E'$ are a symmetric pair of dangling annular blocks, one containing outer points and one containing inner points.
Now $\kappa$ and $\phi(\kappa)$ bound a disk in containing a pair of double points that form a symmetric pair of trivial blocks of $\P$ (in fact, two such disks), contradicting the definition of a simple symmetric pair of connectors for~$\P$.

By this contradiction, we assume, without loss of generality, that $\kappa$ leaves~$E$ between two numbered points in the boundary of $E$.
Let $p$ and $q$ be numbers equivalent modulo $2n$ to those two numbered points such that $\perm^D(\P)$ has a cycle $\gamma$ with a subsequence $p,q$.
Thus $p$ is on the left as $\kappa$ leaves $E$ and $q$ on the right.
It is possible that $p$ and $q$ correspond to the same numbered point.
In that case, $q=p$ and $\gamma$ is a $1$-cycle if $E$ is a degenerate disk block (a point), and otherwise $q=p\pm2n$, $\gamma$ is a loop, and $E$ is an annular block one of whose boundary components contains only one numbered point, equivalent modulo $2n$ to $p$.
We argue in two cases, based on how $\kappa$ enters $E'$.

\medskip
\noindent
\textbf{Case A.}  The curve $\kappa$ enters $E'$ through an empty boundary component.
Then~$E$ is a disk containing only outer points, only inner points, or only double points.

If $E$ contains only outer points, then the subsequence of $\gamma$ from $q$ to $p$ is read in clockwise order along the outer boundary of $D$.
Let $\tau$ be the loop $\ell_q$.
Then~$\tau\gamma$ is an infinite cycle that follows the same points from $q$ to $p$ and then sends $p$ to $q+2n$.
This new infinite cycle is in $\perm^D(\Q)$, read off from the outer boundary of the non-dangling annular block $E''$ in $\Q$ containing $E$ and~$E'$. 
(If $E'$ is a symmetric dangling annular block, then this non-dangling annular block in $\Q$ is symmetric.
If $E',\phi(E')$ is a symmetric pair of dangling annular blocks, then this non-dangling annular block in $\Q$ is part of a symmetric pair.)
Thus $\perm(\Q)=\tau\cdot\perm(\P)$.
If~$E$ contains only inner points, then we argue similarly with $\tau=\ell_q^{-1}$.

If $E$ contains only double points, then $E$ is a degenerate disk or a stitched disk and $E'$ is a dangling annulus with either only inner points or only outer points.
Suppose $E$ is a degenerate disk.  
If $E'$ contains only inner points, let $\tau=\ell_p$, so that $\tau\cdot\perm^D(\P)$ has a new cycle with $q$ then $p$ then $q+2n$ (or just $p$ then $p+2n$ if $p=q$.
This new infinite cycle is in $\perm^D(\Q)$, read from the outer boundary of the nonsymmetric non-dangling annular block in $\Q$ containing $E$ and $E'$.
If $E'$ contains only outer points, then we set $\tau=\ell_q^{-1}$ argue similarly.

If $E$ is a stitched disk, then the symmetric pair of dangling annuli in $\P$ (one with outer points and one with inner points) becomes one symmetric annulus in~$\Q$, with the same outer points and inner points.
Thus the cycles in $\perm^D(\P)$ associated with the stitched disk (one class of 2-element symmetric cycles and one class of tiny cycles) are replaced with the two classes of tiny cycles.
There are four possibilities for~$E$, shown in the left column below.
The right column shows the cycles associated with the stitched disk.
(Here, and throughout the proof, $i$ is the positive upper double point and $j$ is the positive lower double point.)
The middle column shows the loop~$\tau$ such that $\tau\cdot\perm(\P)$ has the two classes of tiny cycles.
\begin{longtable}{|c|c|c|}\hline&&\\[-10pt]
\scalebox{0.9}{\includegraphics{stitched1.pdf}}&\raisebox{13pt}{$(\!(\cdots i\,\,\,\,{i-2n}\cdots)\!)_{2n}$}&\raisebox{13pt}{$(i\,\,\,\,{-i-2n})_{2n}\cdot(j\,\,\,{-j+2n})_{2n}$}\\\hline
\scalebox{0.9}{\includegraphics{stitched2.pdf}}&\raisebox{13pt}{$(\!(\cdots i\,\,\,\,{i+2n}\cdots)\!)_{2n}$}&\raisebox{13pt}{$(i\,\,\,\,{-i+2n})_{2n}\cdot(j\,\,\,{-j+2n})_{2n}$}\\\hline
\scalebox{0.9}{\includegraphics{stitched3.pdf}}&\raisebox{13pt}{$(\!(\cdots j\,\,\,\,{j+2n}\cdots)\!)_{2n}$}&\raisebox{13pt}{$(j\,\,\,\,{-j+4n})_{2n}\cdot(i\,\,\,-i)_{2n}$}\\\hline
\scalebox{0.9}{\includegraphics{stitched4.pdf}}&\raisebox{13pt}{$(\!(\cdots j\,\,\,\,{j-2n}\cdots)\!)_{2n}$}&\raisebox{13pt}{$(j\,\,\,\,{-j})_{2n}\cdot(i\,\,\,-i)_{2n}$}\\\hline
\end{longtable}

\medskip
\noindent
\textbf{Case B.}
The curve $\kappa$ enters $E'$ between numbered points.
We define $p'$, $q'$, and $\gamma'$ for $E'$ and $\kappa$ just as we did for $E$ and $\kappa$.
Thus the numbered point corresponding to $p'$ is on the left and the numbered point corresponding to $q'$ is on the right as $\kappa$ \emph{leaves $E'$ heading for $E$}. 
So far, we have only specified $p'$ and $q'$ up to adding the same multiple of $2n$ to both, but we will become more definite below.

If $E$ is a nonsymmetric disk and $E=E'$, then $E$ does not contain both outer points and inner points.
(If $E$ contains both outer and inner points, then $\kappa$ intersects $\phi(E)$.)
Thus without loss of generality, $E$ either contains only outer points, or only outer points and double points, or only double points.
If $E$ contains only outer points, then $p'=p$ and $q'=q$ and $\Q$ has a new dangling annular block containing the outer points in $E$.
Setting $\tau=\ell_q$, we see that $\tau\cdot\perm^D(\P)$ has the class of infinite cycles corresponding to this dangling annular block, so $\perm^D(\Q)=\tau\cdot\perm^D(\P)$.
If $E$ contains outer and double points, then up to reversing the direction of $\kappa$ (and thus swapping $(p,q)$ with $(p',q')$) we can take $p$ and $q'$ to correspond to outer points and $p'$ and $q$ to correspond to double points.
Possibly $p'=q$.
In $\Q$, there is a non-dangling annular block containing the same numbered points as $E$.
Taking $\tau$ to be $(\!(q\,\,\,{q'+2n})\!)_{2n}$, we have $\perm^D(\Q)=\tau\cdot\perm^D(\P)$.
If~$E$ contains only double points, then $\Q$ has a dangling annular block containing only the double points in $E$.
Without loss of generality, we can assume that the dangling annular block containing $E$ is outside of the dangling annular block containing $\phi(E)$.
If $E$ is a trivial block on a double point $p=q=p'=q'$, then taking $\tau=\ell_p$, we see that $\perm^D(\Q)=\tau\cdot\perm^D(\P)$.
If $E$ is an arc connecting an upper double point to a lower double point, then $p'=p$ and $q'=q$.
Setting $\tau=\ell_q^{-1}$, we again have $\perm^D(\Q)=\tau\cdot\perm^D(\P)$.

If at least one of $E$ and $E'$ is a nonsymmetric disk but $E\neq E'$, then the cycles $\gamma$ and $\gamma'$ in $\perm^D(\P)$ are combined in $\perm^D(\Q)$.
We choose $p'$ and $q'$ so that the combined cycle contains a subsequence $p\,\,\,q'$, and equivalently contains a subsequence $p'\,\,\,q$.
The effect of this choice (if $E'$ is a disk) is that, as we read cycles in the definition of $\perm^D(\Q)$, we reach the a numbered point where we record $p$, follow the boundary of $E$ with the interior of $E$ on the right, turn left to follow close to $\kappa$ to the boundary of $E'$, and turn left to follow the boundary of $E'$ with the interior of $E'$ on the right to a numbered point where we record~$q'$.  
(If instead $E$ is a disk, the same is true, swapping primes for non-primes throughout.)

If neither of $E$ and $E'$ is a nonsymmetric disk, then (since~$\kappa$ leaves~$E$ between numbered points and enters $E'$ between numbered points), there are three possibilities:
the block $E$ is a nonsymmetric annular block containing double points and $E'=\phi(E)$;
the block $E$ is a stitched disk and $E'=\phi(E)=E$; or
the blocks $E$ and~$E'$ are distinct non-stitched symmetric disks.

If $E$ is a nonsymmetric annular block containing double points and $E'=\phi(E)$, then $E$ contains two double points.
(If $E$ contains only one, then there is a pair of trivial blocks in $\P$, and $\kappa$ and $\phi(\kappa)$ combine with $E$ and $E'$ to bound a disk containing these trivial blocks.)
Thus $\perm^D(\P)$ has a class of infinite cycles involving all the double points and $\perm^D(\Q)$ instead has both classes of tiny cycles.
The change from $\perm^D(\P)$ to $\perm^D(\Q)$ only involves the double points, so we can ignore the inner and outer points of annular blocks of $\P$ and $\Q$ and work as if $E$ is a dangling annular block.
The four possibilities for~$E$ and $E'=\phi(E)$ are shown below.
The right column shows the associated class of infinite cycles.
The middle column shows $\tau\in T$ such that $\tau\cdot\perm(\P)$ has the two classes of tiny cycles. 
\begin{longtable}{|c|c|c|}\hline&&\\[-10pt]
\scalebox{0.85}{\includegraphics{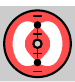}}
&\raisebox{13pt}{$(\!(i\,\,\,\,{-j})\!)_{2n}$}
&\raisebox{13pt}{$(\!(\cdots i\,\,\,\,j\,\,\,\,{i+2n}\cdots)\!)_{2n}$}
\\\hline
\scalebox{0.85}{\includegraphics{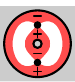}}
&\raisebox{13pt}{$(\!(i\,\,\,\,{j-2n})\!)_{2n}$}
&\raisebox{13pt}{$(\!(\cdots i\,\,\,\,{-j+2n}\,\,\,\,{i+2n}\cdots)\!)_{2n}$}
\\\hline
\scalebox{0.85}{\includegraphics{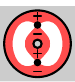}}
&\raisebox{13pt}{$(\!(i\,\,\,\,j)\!)_{2n}$}
&\raisebox{13pt}{$(\!(\cdots i\,\,\,\,{-j}\,\,\,\,{i-2n}\cdots)\!)_{2n}$}
\\\hline
\scalebox{0.85}{\includegraphics{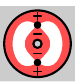}}
&\raisebox{13pt}{$(\!(i\,\,\,\,{-j+2n})\!)_{2n}$}
&\raisebox{13pt}{$(\!(\cdots i\,\,\,\,{j-2n}\,\,\,\,{i-2n}\cdots)\!)_{2n}$}\\\hline
\end{longtable}

If $E$ is a stitched disk block and $E'=\phi(E)=E$, then we have already handled the case where $E$ contains only double points.
(In that case, $\P$ and $\Q$ are just like the case where $E$ is a stitched disk with only double points and $E'$ is a dangling annular block.)
Thus we assume that $E$ has outer points and inner points.
Up to swapping $\kappa$ and $\phi(\kappa)$, we can assume that $p$ is inner and~$q$ is double.
By adding the same multiple of $2n$ to $p$ and $q$, we can choose $q\in\set{\pm i,-j,j-2n}$.
Then~$\gamma$ (the cycle in $\perm^D(\P)$ containing a subsequence $p,q$) is a symmetric cycle
\[(p\,\,\,q\,\,\,a_1\,\cdots\,a_m\,\,\,{-p+2kn}\,\,\,{-q+2kn}\,\,\,{-a_1+2kn}\,\cdots\,{-a_m+2kn})\]
such that $a_1,\ldots,a_m$ is an increasing sequence corresponding to outer points, with $k$ even if and only if $q$ corresponds to a lower double point and $k$ odd if and only if $q$ corresponds to an upper double point.
But in fact, checking each of the four cases for $q$, we see that $k\in\set{0,1}$, so that $\gamma$ is a principal cycle.
In $\perm^D(\Q)$, the class of $\gamma$ becomes $(\!(\cdots\,a_1\,\cdots\,a_m\,\,\,{-p+2kn}\,\,\,{a_1+2n}\,\cdots\,)\!)$, a class of infinite cycles consisting of outer points, and the double points corresponding to $q$ are in a tiny cycle.
Setting $\tau=(\!({-q+2kn}\,\,\,{a_1+2n})\!)_{2n}$, the desired class of infinite cycles is in $\tau\cdot\perm(\P)$.
Checking all four cases for $q$ (keeping in mind that $k=1$ when $q=\pm 1$ and $k=0$ when $q=-j+2n$ or $q=-j$), we see that $\tau\cdot\perm(\P)$ also has the tiny cycles for the double point $q$.
Specifically, if $q=\pm i$, then $k=1$ and $\tau\cdot\perm^D$ has $(i\,\,\,{-i})$ and if $q=j-2n$ or $q=-j$, then $k=0$ and $\tau\cdot\perm^D$ has $(j\,\,\,{-j+2n})$.

If $E$ and $E'$ are distinct non-stitched symmetric disks, then in $\Q$, these blocks are combined into a non-dangling symmetric annular block.
Without loss of generality,~$E$ is upper and $E'$ is lower.
We may as well take $\gamma$ to be a principal symmetric cycle.
Since $\kappa$ leaves $E$ between $p$ and $q$ (with $p$ on the left) and enters $E'$ between $p'$ and $q'$ (with $p'$ on the right), up to swapping $\kappa$ with $\phi(\kappa')$, we can assume that $p$ and $q'$ are both outer while $p'$ and $q$ are both inner.
By adding the correct multiple of $2n$ to $p'$ and $q'$, when following the path from $p$ to $q'$ along the boundary of~$E$, left along $\kappa$, and left along the boundary of $E'$ and recording integers from the numbered points is in the definition of $\perm^D$, we read $p$ and then $q'$.
We can write 
\begin{align*}
\gamma&=(p\,\,\,q\,\,\,a_1\,\cdots\,a_m\,\,\,{-p}\,\,\,{-q}\,\,\,{-a_1}\,\cdots\,{-a_m})\\
\gamma'&=(p'\,\,\,q'\,\,\,a'_1\,\cdots\,a'_{m'}\,\,\,\,{-p'+2kn}\,\,\,\,{-q'+2kn}\,\,\,\,{-a'_1+2kn}\,\cdots\,{-a_{m'}+2kn}),
\end{align*}
with $k$ odd.
Since $p$ and $q'$ are both outer, then since $-p$ is inner, all of the points $q,a_1,\ldots,a_m$ are inner as well, and similarly, all the points $a'_1,\ldots,a'_{m'},{-p'+2n}$ are outer.
Setting $\tau=(\!(q\,\,\,q')\!)$, we see that $\tau\cdot\perm^D(\P)$ has the infinite cycle
\[(\cdots{-q}\,\,\,{-a_1}\,\cdots\,{-a_m}\,\,\,p\,\,\,q'\,\,\,a'_1\,\cdots\,a'_{m'}\,\,\,\,{-p'+2kn}\,\,\,\,{-q+2kn}\cdots)\]
But the sequences ${-q}\,\,\,{-a_1}\,\cdots\,{-a_m}\,\,\,p\,\,\,q$ and $q'\,\,\,a'_1\,\cdots\,a'_{m'}\,\,\,\,{-p'+2kn}\,\,\,\,{-q'+2kn}$ come from reading outer points in clockwise order according to the definition of $\perm^D$. Thus this is the cycle that $\perm^D$ reads along the outer boundary of the non-dangling symmetric annular block in $\Q$.
(In particular, $k=1$.)
\qed

\subsection{Noncrossing partitions with no dangling annular blocks}\label{circ sec}
We now prove the second assertion of Theorem~\ref{isom d}.
The proof consists of first relating covers in $\tNCDcircc$ to covers in $\tNCDc$ and then following the same argument as for  the first assertion, re-using most of the technical details.

\begin{prop}\label{cov circ d}
Suppose $\P,\Q\in\tNCDcircc$.
Then $\Q$ covers $\P$ in $\tNCDcircc$ if and only if $\Q$ covers $\P$ in $\tNCDc$.
If $\P\covered\Q$ and $\kappa,\phi(\kappa)$ is a simple symmetric pair of connectors for $\P$ such that $\Q=\P\cup\kappa\cup\phi(\kappa)$ and if $\kappa$ connects a block $E$ of $\P$ to itself, then $E$ contains double points and also contains inner points and/or outer points.
\end{prop}

\begin{proof}
If $\P\covered\Q$ in $\tNCDc$, then $\P\covered\Q$ in $\tNCDcircc$, because $\tNCDcircc$ is an induced subposet.

Conversely, suppose $\P\covered\Q$ in $\tNCDcircc$.
Suppose for the sake of contradiction that there exists $\R\in\tNCDc$ with $\P<\R\covered\Q$. 
Then 
since $\R\not\in\tNCDcircc$, there is a dangling annular block $E$ in $\R$ contained in a non-dangling annular block $F$~of~$\Q$.
Since $\R\covered\Q$, there are four possibilities for $E$ and $F$, which we treat as four separate cases.
In each case, we will find a contradiction to the supposition that $\P\covered\Q$ in $\tNCDcircc$ by exhibiting a simple symmetric pair of connectors $\kappa,\phi(\kappa)$ for $\P$ such that  $\P\cup\kappa\cup\phi(\kappa)$ has no annular block and $\kappa,\phi(\kappa)\in\curve(\Q)$.
Since $\Q$ has an annular block, we can conclude that $\P<(\P\cup\kappa\cup\phi(\kappa))<\Q$ in $\tNCDcircc$ in each case.

\medskip

\noindent\textbf{Case A.}
$E$ is a nonsymmetric dangling annular block containing double points and $F$ is a nonsymmetric non-dangling annular block containing double points and, without loss of generality, outer points.
Then $\R$ has a nonsymmetric disk block~$E'$ containing only outer points and there is a simple symmetric pair of connectors $\kappa,\phi(\kappa)$ for $\R$ such that $\kappa$ connects $E$ to $E'$.
(For example, $\R$ might be the noncrossing partition pictured in the third row-right picture of Figure~\ref{nc ex fig}.
Then~$E$ is the embedded block containing $-1$ and $-6$ and $E'$ might be the embedded block containing $-2$ and $-3$.)
In $\P$, the double point that is an endpoint of $\kappa$ is contained in a nonsymmetric disk block (containing no inner or outer points) and the outer points in $E'$ are contained in some number of nonsymmetric disk blocks.
We can choose $\kappa$ so that $\kappa,\phi(\kappa)$ is also a simple symmetric pair of connectors for $\P$.
Since~$\kappa$ connects two distinct nonsymmetric disk blocks, $\P\cup\kappa\cup\phi(\kappa)$ has no annular block.

\medskip

\noindent\textbf{Case B.}
$E$ is a nonsymmetric dangling annular block containing (without loss of generality) outer points and $F$ is a nonsymmetric non-dangling annular block containing double points and outer points.
This case is like Case A, with the roles of double points and outer points reversed:
$\R$ has a nonsymmetric disk block $E'$ containing only double points connected to $E$ by $\kappa$ in a simple symmetric pair of connectors $\kappa,\phi(\kappa)$ for $\R$.
(For example, $\R$ might be the bottom-right picture of Figure~\ref{nc ex fig}.)
In $\P$, the double points are in nonsymmetric disk blocks containing no inner or outer points and the outer points in $E'$ are contained in nonsymmetric disk blocks.
We can choose $\kappa$ so that $\kappa,\phi(\kappa)$ is also a simple symmetric pair of connectors for $\P$.
Again, $\P\cup\kappa\cup\phi(\kappa)$ has no annular block.

\medskip

\noindent\textbf{Case C.}
$E$ is a nonsymmetric dangling annular block containing (without loss of generality) outer points and $F$ is a symmetric non-dangling annular block.
This case only occurs when $\R$ has a stitched disk $E'$ containing only double points.
(For example, $\R$ might be the bottom-middle picture of Figure~\ref{nc ex fig}.)
There is a simple symmetric pair of connectors $\kappa,\phi(\kappa)$ for $\R$ connecting $E$ to~$E'$.
In $\P$, the outer points in $E$ are contained in nonsymmetric disks.
The stitched disk $E'$ of $\R$ may be a block of $\P$ or the double points may be contained in nonsymmetric blocks of~$\P$ (containing only double points).
In either case, we can choose $\kappa$ so that $\kappa,\phi(\kappa)$ is also a simple symmetric pair of connectors for $\P$.
The noncrossing partition $\P\cup\kappa\cup\phi(\kappa)$ has no annular block.

\medskip

\noindent\textbf{Case D.}
$E$ is a symmetric dangling annular block and $F$ is a symmetric non-dangling annular block.
In this case, there is a nonsymmetric disk block $E'$ containing the outer points in $F$ and a simple symmetric pair of connectors $\kappa,\phi(\kappa)$ for $\R$ such that $\kappa$ connects $E$ to $E'$.
(For example, $\R$ might be the bottom-left picture of Figure~\ref{nc ex fig}.)
In $\P$, either the double points are contained in a stitched disk (containing no inner or outer points) or in nonsymmetric disk blocks.
As before, we can choose $\kappa$ so that $\kappa,\phi(\kappa)$ is also a simple symmetric pair of connectors for~$\P$, and we see that $\P\cup\kappa\cup\phi(\kappa)$ has no annular block.

In every case, we have contradicted the supposition that $\P\covered\Q$ in $\tNCDcircc$.
We conclude that $\Q$ covers $\P$ in $\tNCDcircc$ if and only if $\Q$ covers $\P$ in $\tNCDc$.

The second assertion of the proposition is immediate:  It states the condition on~$\kappa$ that prevents the augmentation from creating a dangling annular block.
\end{proof}

As a consequence of Proposition~\ref{cov circ d}, the rank function in $\tNCDcircc$ is the restriction of the rank function in $\tNCDc$.
It is also true that the rank function $\ell_T$ in $[1,c]_T$ is the restriction to $[1,c]_T$ of the rank function $\ell_{T\cup L}$ in $[1,c]_{T\cup L}$.

The bulk of the proof of the first assertion of Theorem~\ref{isom d} consisted of proving two propositions (Propositions~\ref{perm inv cov} and~\ref{cov ref loop}).
We now prove the analogous propositions for $\tNCDcircc$ and $[1,c]_T$, re-using most of the work.

\begin{prop}\label{perm inv cov circ}
Suppose $\sigma\covered\pi$ in $[1,c]_T$ and $\pi=\perm^D(\Q)$ for some $\Q\in\tNCDcircc$.
Then there exists $\P\in\tNCDcircc$ such that $\sigma=\perm^D(\P)$ and $\P\covered\Q$.
\end{prop}
\begin{proof}
Since $\sigma\covered\pi$ in $[1,c]_T$, also  $\sigma\covered\pi$ in $[1,c]_{T\cup L}$.
By Proposition~\ref{perm inv cov}, there exists $\P\in\tNCDc$ such that $\sigma=\perm^D(\P)$ and $\P\covered\Q$.
We know $\Q\in\tNCDcircc$, and we look at the proof of Proposition~\ref{perm inv cov} to see that $\P\in\tNCDcircc$:
Leaving out Cases $\EnlargeTiny$ and $\InfToNonSym$, which correspond to multiplying by a loop, we see that the construction of $\P$ from $\Q$ does not create a dangling annular block.
\end{proof}

\begin{prop}\label{cov ref circ}
Suppose $\P,\Q\in\tNCDcircc$ have $\P\covered\Q$.
Then there exists a reflection $\tau\in T$ such that $\perm^D(\Q)=\tau\cdot\perm^D(\P)$.
\end{prop}
\begin{proof}
In light of Proposition~\ref{cov circ d}, Proposition~\ref{cov ref loop} give us a reflection or loop $\tau\in T\cup L$ such that $\perm^D(\Q)=\tau\cdot\perm^D(\P)$.
In the proof of Proposition~\ref{cov ref loop}, in every case where $\tau$ is a loop, either $\Q$ or $\P$ has a dangling annular block.
Thus $\tau$ is in fact a reflection.
\end{proof}

\begin{proof}[Proof of the second assertion of Theorem~\ref{isom d}]
Given a maximal chain $\P_0\covered\cdots\covered\P_n$ in $\tNCDcircc$, use Proposition~\ref{cov ref circ} to write a word $\tau_1\cdots\tau_n$ for $c$ in the alphabet $T$.
This is reduced because $\ell_T(c)=n$.
Since every $\P\in\tNCDcircc$ is on a maximal chain, $\perm^D$ maps $\tNCDcircc$ into $[1,c]_T$.
Furthermore, if $\P\covered\Q$, this cover relation is on some maximal chain, so $\perm^D(\P)\covered_{T\cup L}\perm^D(\Q)$.
Thus $\perm^D$ is order-preserving.

Proposition~\ref{one to one d} says that $\perm^D$ is one-to-one.
Proposition~\ref{perm inv cov circ} and an easy induction shows that $\perm^D$ maps $\tNCDcircc$ onto $[1,c]_T$.
Thus $\perm^D$ restricts to a bijection from $\tNCDcircc$ to $[1,c]_T$.
Proposition~\ref{perm inv cov circ} also implies that the inverse of the restriction is order-preserving, so $\perm^D$ is an isomorphism.
\end{proof}

\section{McCammond and Sulway's lattice in type $\afftype{D}$}\label{barred sec}
Much of the interest in the interval $[1,c]_T$ is motivated by the study of Artin groups.
When $[1,c]_T$ is a lattice (for example when $W$ is finite), it can be used to prove strong results about the Artin group associated to~$W$.
When $W$ is affine, $[1,c]_T$ is often not a lattice, as, for example, we have seen in affine type~$\afftype{D}$.
McCammond and Sulway \cite{McSul} proved many of the same strong results about Artin groups associated to affine Coxeter groups by extending $W$ to a larger group where the analog of $[1,c]_T$ is a lattice.
Briefly, the key idea is the following:
Some elements of $W$ act as translations in the affine reflection representation of~$W$.
We factor the translations that appear in $[1,c]_T$ into translations that are not in $W$ and consider the larger group generated by $W$ and these ``factored translations''.

In this section, we review the notion of factored translations from \cite{McSul}, as slightly generalized in~\cite{affncA}.
We will see that the interval $[1,c]_{T\cup L}$ from Section~\ref{gen sec}, which was shown in Theorem~\ref{isom d} to be isomorphic to $\tNCDc$, is obtained by factoring the translations in $[1,c]_T$, but not factoring them as completely as required in the construction of~\cite{McSul}.
We describe a way to further factor the loops at double points to recover, up to isomorphism, the lattice of McCammond and Sulway in type~$\afftype{D}$.
We close the section with a discussion of how to describe the lattice with planar diagrams and some additional algebraic information.

\subsection{Factored translations in general affine type}
Factored translations make sense for general Coxeter groups of affine type, and we begin with a short summary in this full generality.
More details are in \cite{McFailure,McSul}, and a brief account with conventions and aims similar to the present paper is found in \mbox{\cite[Section~5.1]{affncA}}.

Recall from Section~\ref{gen sec} the reflection representation of a Coxeter group~$W$ on~$V$.
When~$W$ is affine, the affine reflection representation of~$W$ is the restriction of the dual representation to the affine hyperplane $E=\set{x\in V^*:\br{x,\delta}=1}$.

For each Coxeter element $c$, the \newword{Coxeter axis} is the unique line in $E$ fixed as a set by~$c$.
The Coxeter axis is the intersection of $E$ with the Coxeter plane in~$V^*$, and its direction is $\omega_c(\delta,\,\cdot\,)$.
A reflection in $W$ is \newword{horizontal} if its reflecting hyperplane in~$E$ is parallel to the Coxeter axis.
A root $\beta+k\delta\in\Phi$ with $\beta\in\Phi_\fin$ defines a horizontal reflection if and only if $\omega_c(\delta,\beta)=0$.
The set of roots $\beta\in\Phi_\fin$ such that the roots $\beta+k\delta$ are horizontal is a finite root system called the \newword{horizontal root system}.
If the horizontal root system has irreducible components $\Psi_1,\ldots,\Psi_k$, let $U_i$ denote the span of $\set{K(\,\cdot\,,\beta):\beta\in\Psi_i}$.
Also, let $U_0$ be the line spanned by $\omega_c(\delta,\,\cdot\,)$.
The linear subspace $E_0$ parallel to $E$ is an orthogonal direct sum ${U_0\oplus\cdots\oplus U_k}$.

The group of translations in $W$ is generated by products of two reflections in $W$ with adjacent parallel reflecting hyperplanes in $E$.
If $w$ is a translation in $[1,c]_T$ with translation vector $\lambda$, write $\lambda=\lambda_0+\lambda_1+\cdots+\lambda_k$ with $\lambda_i\in U_i$ for $i=0,\ldots,k$.
Fix real numbers $q_1,\ldots,q_k$ with $\sum_{i=1}^kq_k=1$ and define vectors $\lambda_i+q_i\lambda_0$ for $i=1,\ldots,k$.
These vectors sum to~$\lambda$.
The \newword{factored translations} associated to~$w$ are the $k$ translations defined by these translations vectors.
Let $F$ be the set of all factored translations associated to translations $w\in[1,c]_T$.

Each factored translation on $E$ defines a linear transformation on $V^*$, so the group generated by $W$ and $F$ (or equivalently by $T$ and $F$) has a representation on $V^*$ and thus dually, a representation on $V$.

The translations in $[1,c]_T$ have reflection length $2$.
We assign each factored translation a length so that the factors of each translation have lengths summing to~$2$.
We write $[1,c]_{T\cup F}$ for the interval between $1$ and $c$ in the supergroup (the prefix order on the supergroup with respect to the alphabet $T\cup F$ with this length function).
As before, we think of this interval as a labeled poset.
In principle, one might want a generating set that is closed under passing to inverses.
Thus one would consider $[1,c]_{T\cup F^\pm}$, where $F^\pm$ is the set of elements of $F$ and their inverses.
But even if we allow the larger alphabet, only letters from $T\cup F$ would appear in reduced words for $c$.
(See Proposition~\ref{if a fact}, below.)

The \newword{interval group} associated to $[1,c]_{T\cup F}$ is given by an abstract presentation, generated by those elements of $T\cup F$ that appear as labels in $[1,c]_{T\cup F}$, with relations equating label sequences on different unrefinable chains with the same endpoints.

The constants $q_i$ do not appear in \cite{McSul}.
In the language of this paper, \cite{McSul} uses $q_i=\frac1k$ for all~$i$.
Our use of constants $q_i$ in \cite{affncA} and here is guided by factorizations of translations arising from the planar combinatorial models, which use choices of the~$q_i$ different from $q_i=\frac1k$.
Thus the following theorem, which is \cite[Theorem~5.1]{affncA}, is important in relating our work to~\cite{McSul}.

\begin{thm}\label{same interval group}
Given $q_1+\cdots+q_k=1$ and $q'_1+\cdots+q'_k=1$, let~$F$ and~$F'$ be the corresponding sets of factored translations.
Then the intervals $[1,c]_{T\cup F}$ and $[1,c]_{T\cup F'}$ are isomorphic as labeled posets.
The interval group constructed from $[1,c]_{T\cup F}$ is isomorphic to the interval group constructed from $[1,c]_{T\cup F'}$.
\end{thm}

McCammond and Sulway \cite[Theorem~8.9]{McSul} show that $[1,c]_{T\cup F}$ is a lattice when $q_i=\frac1k$ for all $i$.
Combining with Theorem~\ref{same interval group}, we have the following version of their theorem.

\begin{thm}\label{lattice}
For an arbitrary choice of real numbers $q_1,\ldots,q_k$ summing to~$1$, let $F$ be the corresponding set of factored translations.
Then $[1,c]_{T\cup F}$ is a lattice.
\end{thm}

The following proposition is a concatenation of results of~\cite{McSul} whose arguments work for general choices of the $q_i$.
The proposition appears as \cite[Proposition~5.4]{affncA}.

\begin{prop}\label{if a fact}
For an arbitrary choice of real numbers $q_1,\ldots,q_k$ summing to~$1$, suppose an element of $F^\pm$ appears as one of the letters in a reduced word for $c$ in the alphabet $T\cup F^\pm$.
\begin{enumerate}[\quad\rm\bf1.]
\item \label{fact horiz}
Each reflection in the word is horizontal.
\item \label{ref in int}
Each reflection in the word is contained in $[1,c]_T$.
\item \label{n-2 ref}
There are exactly $n-2$ reflections in the word.
\item \label{k trans}
There are exactly $k$ translations in the word.
\item \label{all factors}
The translations in the word are the $k$ factors of some translation in $[1,c]_T$.
(In particular, they are elements of $F$, whereas \textit{a priori} they are in $F^\pm$.)
\item \label{reorder}
The reduced word can be reordered, by swapping letters that commute in the group, so that, for each~$i$, all reflections associated to $\Psi_i$ and all factored translations associated to $U_i$ are adjacent to each other.
\item \label{fact c}
For each $i$, let $c_i$ be the product of the subword consisting only of reflections associated to $\Psi_i$ and all factored translations associated to $U_i$.
Then $c_i$ depends only on $c$ and the constants $q_i$, not on the choice of reduced word. 
\end{enumerate}
\end{prop}

\subsection{Factored translations in affine type D}\label{McSul D sec} 
The type-$\afftype{D}$ Coxeter group $\Stildedes$ has only one conjugacy class of Coxeter element. 
Thus for various choices of Coxeter elements, the intervals $[1,c]_T$ are isomorphic (by conjugating each element of $[1,c]_T$ by the same fixed element of $W$).
In particular, the isomorphism sends translations to translations.
For this reason, it is enough to work with one fixed Coxeter element
\[c=s_0\cdots s_{n-1}=(\!(\cdots\,2\,\,\,3\,\cdots\,{n-2}\,\,\,\,2+2n\,\cdots)\!)\,(1\,\,\,{-1})_{2n}\,({n-1}\,\,\,{n+1})_{2n},\] 
corresponding to the choice of $\pm1$ and $\pm(n-1)$ to be double points, $2,\ldots,n-2$ to be outer points, and $-2,\ldots,-n+2$ to be inner points.
Statements for general~$c$ can be recovered using source-sink moves as in Lemma~\ref{d annulus label}, as discussed in Section~\ref{change cox sec}.

The following fact is immediate from Lemma~\ref{d horiz proj}.
We omit the proof.

\begin{prop}\label{D horiz}
Choose the Coxeter element $c=s_0\cdots s_{n-1}$ of $\Stildedes$.
Then
\begin{enumerate}[\quad\rm\bf1.]
\item \label{D horiz root}
The horizontal root system is 
\[\set{\pm(\e_j-\e_i):\,2\le i<j\le n-2}\cup\set{\pm(\e_{n-1}-\e_1)}\cup\set{\pm(\e_{n-1}+\e_1)},\]
and this is a decomposition into irreducible components.
\item \label{D horiz ref}
The horizontal reflections in $[1,c]_T$ are:
\begin{itemize}
\item $(\!(i\,\,\,\,j)\!)_{2n}$ and $(\!(i\,\,\,\,{j-2n})\!)_{2n}$ such that ${2\le i<j\le n-2}$, 
\item $(\!(1\,\,\,\,{n-1})\!)_{2n}$, $(\!(1\,\,\,\,{n+1})\!)_{2n}$, $(\!(1\,\,\,\,{-n-1})\!)_{2n}$, and $(\!(1\,\,\,\,{-n+1})\!)_{2n}$.
\end{itemize}
\item\label{D horiz NC} 
The horizontal reflections in $[1,c]_T$ are the elements $\perm^D(\P)$ for $\P\in\tNCDc$ whose only nontrivial blocks are a pair of nonsymmetric disks, each containing exactly two numbered points, both outer, both inner, or both double.
\end{enumerate}
\end{prop}


Proposition~\ref{D horiz} allows us to describe the orthogonal decomposition of the hyperplane ${E_0=\set{x\in V^*:\br{x,\delta}=0}}$.
The numbering of $U_1,\ldots,U_k$ is arbitrary, but the following lemma makes an implicit decision on how to number.
The lemma is an easy consequence of Lemma~\ref{big omega} and Proposition~\ref{D horiz}, and we omit the details.

\begin{lemma}\label{orth decomp}
Choose the Coxeter element $c=s_0\cdots s_{n-1}$ of $\Stildedes$.
The orthogonal decomposition $E_0=U_0\oplus\cdots\oplus U_k$ has $k=3$ and 
\begin{itemize}
\item $U_0$ is the span of $-\rho_0-\rho_1+\rho_{n-2}+\rho_{n-1}$.
\item $U_0\oplus U_1$ is the set of vectors $\sum_{i=0}^{n-1}c_i\rho_i\in E_0$ with $c_0=c_1$ and $c_{n-2}=c_{n-1}$.
\item $U_2$ is the span of $-\rho_0+\rho_1+\rho_{n-2}-\rho_{n-1}$.
\item $U_3$ is the span of $\rho_0-\rho_1+\rho_{n-2}-\rho_{n-1}$.
\end{itemize}
\end{lemma}

We now characterize the translations in $[1,c]_T$.

\begin{prop}\label{D trans}
Choose the Coxeter element $c=s_0\cdots s_{n-1}$ of $\Stildedes$.
Then the translations in $[1,c]_T$ are the permutations ${(\!(\cdots\,a\,\,\,a+2n\,\cdots)\!)}\cdot(\!(\cdots\,b\,\,\,b-2n\,\cdots)\!)$ for $a\in\set{2,\ldots,n-2}$ and $b\in\set{\pm1,\pm(n-1)}$.
In $E$, the translations vectors are
\[
\rho_a-\rho_{a-1}+
\begin{rcases}\begin{dcases}
-\rho_0&\text{if }a=2,\\
\rho_{n-1}&\text{if }a=n-2\\
\end{dcases}\end{rcases}
+
\begin{rcases}\begin{dcases}
\pm\rho_0\mp\rho_1&\text{if }b=\pm1,\\
\pm\rho_{n-2}\mp\rho_{n-1}&\text{if }b=\pm(n-1)\\
\end{dcases}\end{rcases}.
\]
\end{prop}
\begin{proof}
The subgroup of $\Stildedes$ consisting of translations is generated by elements
\[(\!(a\,\,\,b)\!)\cdot(\!(a\,\,\,b+2n)\!)=(\!(\cdots\,a\,\,\,a+2n\,\cdots)\!)\cdot(\!(\cdots\,b\,\,\,b-2n\,\cdots)\!)\] for $a,b\in\integers\setminus\set{\ldots,-n,0,n,\ldots}$ with $a\neq \pm b\mod{2n}$.
Thus Theorem~\ref{isom d} implies that the translations in $[1,c]_T$ are precisely the images under $\perm^D$ of noncrossing partitions $\P$ with exactly two nontrivial blocks, namely a symmetric pair of annular blocks, each having two numbered points, one inner or outer and one a double point.
The outer of the two has a numbered point $a\in\set{2,\ldots,n-2}$ and a numbered point $b\in\set{\pm1,\pm(n-1)}$.

For such a $\P$, the action of $\perm^D(\P)$ on the simple roots basis of~$V$ is
\[
\alpha_0=\e_1+\e_2\,\,\mapsto\alpha_0
+\{\delta\text{ if }a=2\}
+\{\mp\delta\text{ if }b=\pm1\}
\]
\[\alpha_{n-1}=\e_{n+1}-\e_{n-2}\,\,\mapsto\alpha_{n-1}
+\{-\delta\text{ if }a=n-2\}
+\{\pm\delta\text{ if }b=\pm(n-1)\}
\]
and for $i=1\ldots,n-2$,\\[-12pt]
\[\alpha_i=\e_{i+1}-\e_i\,\,\mapsto\alpha_i
+\begin{rcases}\begin{dcases}
\delta&\text{if }a=i+1\\
-\delta&\text{if }a=i
\end{dcases}\end{rcases}
+\begin{rcases}\begin{dcases}
\mp\delta&\text{if }b=\pm(i+1)\\
\pm\delta&\text{if }b=\pm i
\end{dcases}\end{rcases}.
\]
Let $M$ be the matrix of $\perm^D(\P)$ acting on column vectors of simple-root coordinates.
The action of $\perm^D(\P)$ on a vector $x=\sum_{i=0}^{n-1}c_i\rho_i$ is the inverse of the action of $M$ on the row vector $[c_0\,\cdots\,c_{n-1}]$.
(We have used the fact that roots and co-roots coincide.)
If $x\in E=\set{x\in V^*:\br{x,\delta}=1}$ we can check that the dual action sends $x$ to $x-\lambda$, where~$\lambda$ is the vector described in the proposition.
\end{proof}

The McCammond-Sulway construction (with general coefficients $q_i$) factors translations in $[1,c]_T$ into factors $\lambda_i+q_i\lambda_0$ for $i=1,2,3$.
However, the translations described in Proposition~\ref{D trans} naturally factor into two factors ${(\!(\cdots\,a\,\,\,a+2n\,\cdots)\!)}$ and $(\!(\cdots\,b\,\,\,b-2n\,\cdots)\!)$, each of which is a loop, as defined in Section~\ref{D big sec}.
Define
\begin{align*}
\lambda_\out&=\rho_a-\rho_{a-1}+
\begin{rcases}\begin{dcases}
-\rho_0&\text{if }a=2,\\
\rho_{n-1}&\text{if }a=n-2\\
\end{dcases}\end{rcases}\\ 
\lambda_\doub&=\begin{rcases}\begin{dcases}
\pm\rho_0\mp\rho_1&\text{if }b=\pm1,\\
\pm\rho_{n-2}\mp\rho_{n-1}&\text{if }b=\pm(n-1)\\
\end{dcases}\end{rcases}.
\end{align*}
Arguing exactly as in Proposition~\ref{D trans}, we have the following proposition.
\begin{prop}\label{D trans loops} 
Choose the Coxeter element $c=s_0\cdots s_{n-1}$ of $\Stildedes$.
Given a translation ${(\!(\cdots\,a\,\,\,a+2n\,\cdots)\!)}\cdot(\!(\cdots\,b\,\,\,b-2n\,\cdots)\!)$ in $[1,c]_T$, factor its translation vector $\lambda$ as $\lambda_\out+\lambda_\doub$.
Then $(\!(\cdots\,a\,\,\,a+2n\,\cdots)\!)$ is the translation with vector $\lambda_\out$ and $(\!(\cdots\,b\,\,\,b-2n\,\cdots)\!)$ is the translation with vector $\lambda_\doub$.
Thus the set of factored translations obtained from factorizations of $\lambda$ as $\lambda_\out+\lambda_\doub$ are the loop elements $\ell_i$ such that $i$ is an outer point or a (positive or negative) double point.
\end{prop}


We now show that the factorization $\lambda=\lambda_\out+\lambda_\doub$ can be further factored to obtain a factorization that comes from the orthogonal decomposition of~$E_0$.

\begin{lemma}\label{McSul 3}
Choose the Coxeter element $c=s_0\cdots s_{n-1}$ of $\Stildedes$.
Given a translation in $[1,c]_T$ with translation vector $\lambda$, write $\lambda=\lambda_0+\lambda_1+\lambda_2+\lambda_3$ with $\lambda_i\in U_i$ for $i=0,1,2,3$.
Then $\lambda_0+\lambda_1=\lambda_\out$ and $\lambda_2+\lambda_3=\lambda_\doub$.
More specifically,
\begin{itemize}
\item $\lambda_2=\pm\frac12(-\rho_0+\rho_1+\rho_{n-2}-\rho_{n-1})$ and 
\item $\lambda_3=\pm\frac12(\rho_0-\rho_1+\rho_{n-2}-\rho_{n-1})$.
\end{itemize}
The signs depend on which of the four loops on double points $\lambda_\doub$ corresponds~to. 
\end{lemma}

\begin{proof}
Lemma~\ref{orth decomp} implies that $\lambda_\out$ is in $U_0\oplus U_1$.
Furthermore, the vectors \linebreak
$\pm\frac12(-\rho_0+\rho_1+\rho_{n-2}-\rho_{n-1})$ and $\pm\frac12(\rho_0-\rho_1+\rho_{n-2}-\rho_{n-1})$ are in $U_2$ and $U_3$ respectively and, with the appropriate signs, add up to $\lambda_\doub$.
%
\end{proof}

As a consequence of Lemma~\ref{McSul 3}, we can set $q_1=1$ and $q_2=q_3=0$ and factor the translation as $\lambda=(\lambda_1+\lambda_0)+\lambda_2+\lambda_3=\lambda_\out+\lambda_2+\lambda_3$.
Thus the factored translations are described in the following proposition.

\begin{prop}\label{D factored}
For $c=s_0\cdots s_{n-1}$ and $q_1=1$ and $q_2=q_3=0$, the factored translations vectors are
\begin{multline*}\set{\rho_a-\rho_{a-1}+
\begin{rcases}\begin{dcases}
-\rho_0&\text{if }a=2,\\
\rho_{n-1}&\text{if }a=n-2\\
\end{dcases}\end{rcases}:a\in\set{2,\ldots,n-2}}\\
\cup\set{\pm\frac12(-\rho_0+\rho_1+\rho_{n-2}-\rho_{n-1},\pm\frac12(\rho_0-\rho_1+\rho_{n-2}-\rho_{n-1})}.
\end{multline*}
\end{prop}

Also by Lemma~\ref{McSul 3}, for $q_1=1$ and $q_2=q_3=0$, the group generated by $T$ and~$F$ is a supergroup of the group $\Stildejes$ of affine jointly even-signed permutations.
In the next section, we describe the supergroup.

\subsection{Affine barred even-signed permutations}\label{barred subsec}
We now describe the group generated by reflections and factored translations for $q_1=1$ and $q_2=q_3=0$.
For the purpose of defining $[1,c]_{T\cup F}$, we define a factored translation associated to $U_1$ to have length $1$ and a factored translation associated to $U_2$ or $U_3$ to have length $\frac12$.

For any $\lambda$, Lemma~\ref{McSul 3} says that $\lambda_2$ and $\lambda_3$ are translations whose vector is half of the translation associated to a loop on an upper double point plus half of the translation associated to a loop on a lower double point.
Again looking back at the proof of Proposition~\ref{D trans}, we see that the corresponding linear maps on $V$ send $\e_1$ to $\e_1\pm\frac12\delta$, send $\e_{n-1}$ to $\e_{n-1}\pm\frac12\delta$, send $\e_{n+1}$ to $\e_{n+1}\pm\frac12\delta$ (taking the opposite sign for $\e_{n+1}$ as for $\e_{n-1}$), and fix $\e_2,\ldots,\e_{n-2}$.
These linear maps also fix $\delta=\e_{n+1}+\e_{n-1}$.

Define upper-barring and lower-barring operations on $V$, inverse to each other, by $\bar{v}=v+\frac\delta2$ and $\ubar{v}=v-\frac\delta2$.
The sets $\set{\e_i:i\not\equiv0\mod n}$ and $\set{\eb_i:i\not\equiv0\mod n}$ are disjoint, but $\set{\eb_i:i\not\equiv0\mod n}$ coincides with $\set{\eub_i:i\not\equiv0\mod n}$ because $\eb_i=\eub_{i+2n}$.
All elements of $W$ and the four linear maps corresponding to the possibilities for $\lambda_2$ and $\lambda_3$ permute $\set{\e_i:i\not\equiv0\mod n}\cup\set{\eb_i:i\not\equiv0\mod n}$ and fix $\delta$.
We will describe them as permutations of indices, with $\bar{\imath}$ standing for the index on $\eb_i$ and $\ubar{i}$ standing for the index on $\eub_i$.

To do so, we extend notions of negation and adding $2n$ to barred indices.
Specifically, $-\bar\imath$ means $\underline{-i}$ because it stands for $-\eb_i=-(\e_i+\frac\delta2)=-\e_i-\frac\delta2$.
Similarly, $-\ubar{i}$ means $\overline{-i}$.
Also, $\bar\imath+2n$ means $\overline{i+2n}$, because it stands for $\eb_i+\delta=\e_i+\frac{3\delta}2=\overline{\e_{i+2n}}$.

A linear map $\pi$ that permutes $\set{\e_i:i\not\equiv0\mod n}\cup\set{\eb_i:i\not\equiv0\mod n}$ and fixes $\delta$ must also commute with the barring operators.
That is (writing $\pi$ as a permutation of indices),  $\pi(\bar\imath)=\overline{\pi(i)}$ and $\pi(\ubar\imath)=\underline{\pi(i)}$.
Also, $\pi(i+2n)=\pi(i)+2n$ and $\pi(\overline{i+2n})=\overline{\pi(i)}+2n$.
Finally, $\pi(-i)=-\pi(i)$ and $\pi(-\bar\imath)=-\pi(\bar\imath)$.

We will use double parentheses $(\!(\cdots)\!)$ with an expanded meaning in this context, namely $(\!(\cdots)\!)$ not only means a cycle and its negative, but also the cycles obtained by upper-barring each entry in the original cycle.
With this notation, the permutations corresponding to the possibilities for $\lambda_2$ and $\lambda_3$ are:
\begin{align*}
\fuu&=(\!(\cdots\,1\,\,\,\,\bar{1}\,\,\,\,1+2n\,\cdots)\!)\cdot(\!(\cdots\,n-1\,\,\,\,\overline{n-1}\,\,\,\,3n-1\,\cdots)\!)\\
\fud&=(\!(\cdots\,1\,\,\,\,\bar{1}\,\,\,\,1+2n\,\cdots)\!)\cdot(\!(\cdots\,n-1\,\,\,\,\underline{n-1}\,\,\,\,-n-1\,\cdots)\!)\\
\fdu&=(\!(\cdots\,1\,\,\,\,\ubar{1}\,\,\,\,1-2n\,\cdots)\!)\cdot(\!(\cdots\,n-1\,\,\,\,\overline{n-1}\,\,\,\,3n-1\,\cdots)\!)\\
\fdd&=(\!(\cdots\,1\,\,\,\,\ubar{1}\,\,\,\,1-2n\,\cdots)\!)\cdot(\!(\cdots\,n-1\,\,\,\,\underline{n-1}\,\,\,\,-n-1\,\cdots)\!).
\end{align*}
Thus the factorizations of loops on double points are 
\[\begin{array}{rclcrcl}
\ell_1&\!\!\!\!=\!\!\!\!&\fuu\fud&\qquad\qquad&
\ell_{n-1}&\!\!\!\!=\!\!\!\!&\fuu\fdu\\
\ell_{-1}&\!\!\!\!=\!\!\!\!&\fdu\fdd&&
\ell_{-(n-1)}&\!\!\!\!=\!\!\!\!&\fud\fdd.
\end{array}\]
The permutations $\fuu,\fud,\fdu,\fdd$ are more easily described as follows:
The first subscript tells whether to upper- or lower-bar every integer $1$ modulo $2n$, and the second subscript tells whether to upper- or lower-bar every integer $n-1$ modulo~$2n$.

Given a permutation $\pi$ of the unbarred and barred integers, 
define a map $\pi'$ on the unbarred integers by 
\[\pi'(i)=\begin{cases}
\pi(i)&\text{ if }\pi(i)\in\integers,\\[3pt]
\overline{\pi(i)}&\text{ if }\pi(i)=\ubar j\text{ for }j\in\integers\text{ and }\exists k\in\set{1,\ldots,n-1}\text{ s.t. }j\equiv -k\cmod{2n}\\[3pt]
\underline{\pi(i)}&\text{ if }\pi(i)=\bar\jmath\text{ for }j\in\integers\text{ and }\exists k\in\set{1,\ldots,n-1}\text{ s.t. }j\equiv k\cmod{2n}.
\end{cases}\]
An \newword{affine barred even-signed permutation} is a permutation $\pi$ of $\integers\cup\set{\bar\imath:i\in\integers}$ that commutes with barring, such that $\set{\pi(1),\ldots,\pi(n-1)}\cap\set{\bar\imath:i\in\integers}$ has an even number of elements and $\pi'$ is an affine jointly even-signed permutation.
The affine barred even-signed permutations form a group under composition, denoted $\Stildebes$.

Any element of $\Stildejes$ can be considered as an element of $\Stildebes$ by requiring that it commute with barring.
Define $L_\out=\set{\ell_a:a=2,\ldots,n-2}$ (the set of loops on outer points) and $F_\doub=\set{\fuu,\fud,\fdu,\fdd}$.
The following proposition is immediate from Proposition~\ref{D factored} and the considerations above.

\begin{prop}\label{D factored perms}
For $c=s_0\cdots s_{n-1}$ and $q_1=1$ and $q_2=q_3=0$, the factored translations, as elements of $\Stildebes$, are $L_\out\cup F_\doub$.
\end{prop}

\begin{prop}\label{bes prop ever}
The group $\Stildebes$ of affine barred even-signed permutations is generated by $T\cup L_\out\cup F_\doub$.
\end{prop}
\begin{proof}
An element $f\in F_\doub$ is in $\Stildebes$ because it sends exactly two elements of $\set{1,\ldots,n-1}$ to barred elements and $f'$ is a product of $0$, $1$, or $2$ loops. 

Given $\pi\in\Stildebes$, Proposition~\ref{big group gen} says that $\pi'\in\Stildejes$ is a product of elements in $T\cup L_\out$.
Conjugating $f\in F_\doub$ by products of reflections of the form $(\!(i\,\,\,j)\!)_{2n}$ with $1\le i<j\le n-1$ yields a permutation that upper- or lower bar two elements in $\set{1,\ldots,n-1}$.
We obtain $\pi$ by applying such permutations to the left of~$\pi'$.
\end{proof}

The set $L_\out\cup F_\doub$ is the set of factored translations for $q_1=1$ and $q_2=q_3=0$.
As before, elements of $F_\doub$ have length $\frac12$ and elements of $T\cup L_\out$ have length~$1$.
The following corollary is an immediate consequence of Theorem~\ref{lattice}.

\begin{cor}\label{D big lattice}
The interval $[1,c]_{T\cup L_\out\cup F_\doub}$ in $\Stildebes$ is a lattice.
\end{cor}

\subsection{A description of the lattice}\label{lat sec}
Write $\McSul^D$ for the lattice $[1,c]_{T\cup L_\out\cup F_\doub}$.
With sufficient determination (or stubbornness), one can make planar diagrams for $\McSul^D$.
However, it seems difficult to find diagrams that admit both a natural description of the map to permutations and a natural description of the partial order.  
Instead, we will use results of \cite{McSul} to describe the lattice as the union of $[1,c]_{T\cup L}$ (the image of $\tNCDc$ under $\perm^D$) and a finite set of additional elements of $\Stildebes$.
We begin by making Proposition~\ref{if a fact}.\ref{fact c} explicit for our special choice of~$c$.

\begin{lemma}\label{c1c2c3}
Choose the Coxeter element $c=s_0\cdots s_{n-1}$ of $\Stildedes$.
The factorization $c=c_1c_2c_3$ described in Proposition~\ref{if a fact}.\ref{fact c} has 
\begin{align*}
c_1&=(\!(\cdots\,2\,\,\,3\,\cdots\,n-2\,\,\,2n+2\,\cdots)\!)\\
c_2&=(\!(1\,\,\,\overline{-n-1})\!)_{2n}\\
c_3&=(\!(1\,\,\,\overline{-n+1})\!)_{2n}
\end{align*}
\end{lemma}
\begin{proof}
One can compute that $c_1c_2c_3=c$.
To complete the proof, we express each $c_i$ as a product of horizontal reflections associated to~$\Psi_i$ and factored translations associated to $U_i$ such that the concatenation of the expressions is a reduced word~for~$c$.

First, $c_1$ corresponds to a noncrossing partition with one symmetric pair of dangling annular blocks and two pairs of trivial blocks at double points, with rank~$n-3$ in $\tNCDc$.
Thus there is an expression of length $n-3$ for $c_1$ in the alphabet $T\cup L$.
But this expression only involves reflections $(\!(a\,\,\,b)\!)$ with $a$ and $b$ outer and loops $\ell_a$ with~$a$ outer (that is, horizontal reflections associated to $\Psi_i$ and factored translations associated to~$U_i$).


Next, we write $c_2=\fud\cdot(\!(1\,\,\,{n-1})\!)_{2n}$ and note that $(\!(1\,\,\,{n-1})\!)_{2n}$ is a horizontal reflection in $[1,c]_T$ associated to $\Psi_2$ and $\fud$ is a factored translation for $U_2$.
We also write $c_3=\fuu\cdot(\!(1\,\,\,{n+1})\!)_{2n}$ and note that $(\!(1\,\,\,{n+1})\!)_{2n}$ is a horizontal reflection in $[1,c]_T$ associated to $\Psi_3$ and $\fuu$ is a factored translation for $U_3$.

Concatenating these expressions for $c_1$, $c_2$, and $c_3$, we obtain a word with $n+1$ letters and length $n$, because $\fud$ and $\fuu$ have length $\frac12$.
%
\end{proof}

Write $T_H$ for the horizontal reflections and write $F$ for $L_\out\cup F_\doub$.
The group generated by $T_H\cup F$ is called the \newword{factorable group} \cite[Definition~6.8]{McSul}.
We will consider the interval $[1,c]_{T_H\cup F}$ in the factorable group.
This interval is the subposet of $[1,c]_{T\cup F}$ consisting of elements that are products of horizontal reflections and translations in $F$.
In light of Proposition~\ref{if a fact}.\ref{reorder}--\ref{fact c}, $[1,c]_{T_H\cup F}$ is the product $[1,c_1]_{T_H\cup F}\times[1,c_2]_{T_H\cup F}\times[1,c_3]_{T_H\cup F}$, which is finite.
(For a much stronger statement, see \cite[Proposition~7.6]{McSul}.)

Every element of $[1,c]_{T\cup F}$ that is not already in $[1,c]_T$ is in $[1,c]_{T_H\cup F}$.
(This follows from Proposition~\ref{if a fact} or \cite[Lemma~7.2]{McSul}.)
But many of these elements are already in $[1,c]_{T\cup L}$:
Proposition~\ref{D horiz} and Lemma~\ref{c1c2c3} imply that $[1,c_1]_{T_H\cup F}$ equals $[1,c_1]_{T\cup L}$, so the elements of $[1,c]_{T\cup F}$ that are not in $[1,c]_{T\cup L}$ are the elements $x\in[1,c]_{T_H\cup F}$ having exactly one element of $F_\doub$ in any reduced expression for~$x$.
Equivalently, these are the elements $v_1v_2v_3\in[1,c]_{T_H\cup F}$ (with each $v_i\in[1,c_i]_{T_H\cup F}$) having exactly one element of $F_\doub$ in any reduced expression for~$v_2v_3$.

Write $\P_0$ for the noncrossing partition in $\tNCDc$ with only trivial blocks.
Write~$\P_1$ for the preimage of $c_1$ in $\tNCDc$.
This is the noncrossing partition with pairs of trivial blocks at the double points and a symmetric pair of dangling annular blocks, one containing all outer points and one containing all inner points.
Then $[1,c_1]_{T_H\cup F}=[1,c_1]_{T\cup L}=\perm^D([\P_0,\P_1])$, and we write $\perm_1$ for this poset.

The interval $[1,c_2]_{T_H\cup F}$ has six elements:  
Besides the bottom and top elements~$1$ and $c_2$, there are four pairwise incomparable elements $\fud$, $(\!(1\,\,\,{n-1})\!)_{2n}$, $\fdu$, and $(\!(1\,\,\,{-n-1})\!)_{2n}$.
Of these six elements, three ($\fud$, $\fdu$, and $c_2$) have an element of $F_\doub$ (necessarily only one) in any reduced expression for $v_2$.
Similarly, $[1,c_3]_{T_H\cup F}$ has six elements, including pairwise incomparable elements 
$\fuu$, $(\!(1\,\,\,{n+1})\!)_{2n}$, $\fdd$, and $(\!(1\,\,\,{-n+1})\!)_{2n}$.
Three of the elements ($\fuu$, $\fdd$, and $c_3$) have an element of $F_\doub$ in any reduced expression for $v_2$.

Write $\C_{23}$ for $[1,c_2]_{T_H\cup F}\times[1,c_3]_{T_H\cup F}$.
There are 18 elements $v_2v_3\in\C_{23}$ having exactly one element of $F_\doub$ in any reduced expression for~$v_2v_3$.
We write $\New$ for the subposet of $\C_{23}$ induced by these elements, so that $[1,c]_{T\cup F}\setminus[1,c]_{T\cup L}=\perm_1\times\New$.

Write $\P_{23}$ for the preimage of $c_2c_3$ in $\tNCDc$, the noncrossing partition whose only nontrivial block is a symmetric dangling annular block.
Define $\perm_{23}$ to be $\perm^D([\P_0,\P_{23}])=[1,c_2c_3]_{T\cup L}$, the complement of $\New$ in $\C_{23}$.
Figure~\ref{perm and new} shows the lattice $\C_{23}$, with elements of $\perm_{23}$ represented by the corresponding noncrossing partitions and elements of $\New$ given as permutations.
\begin{sidewaysfigure}
\scalebox{0.6}{\includegraphics{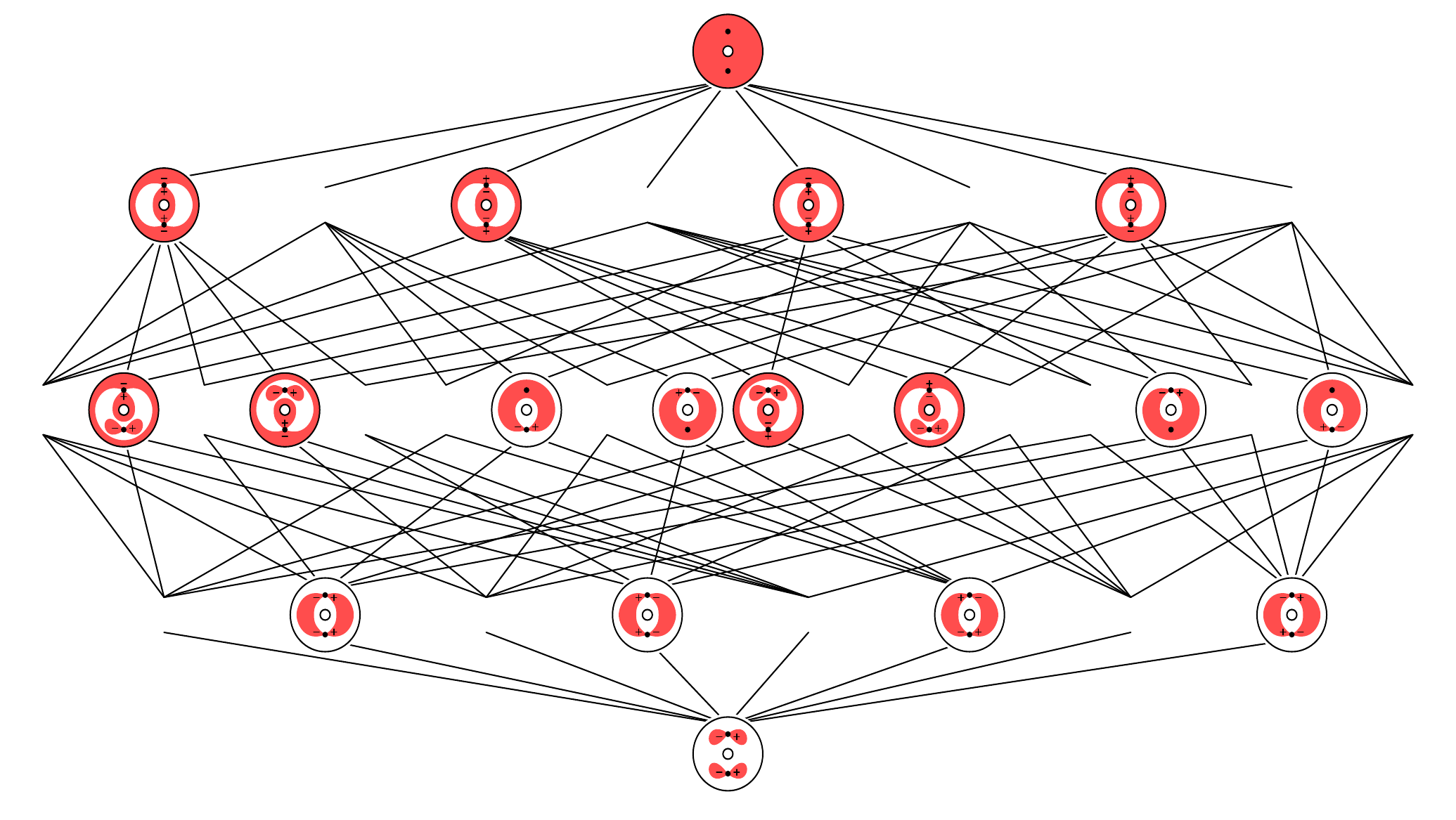}
\begin{picture}(0,0)(510,-40)
\put(-388,99){\Large$\fud$}
\put(-168,99){\Large$\fdu$}
\put(54,99){\Large$\fuu$}
\put(276,99){\Large$\fdd$}
\put(-467,239){\Large$c_2$}
\put(-374,244){\large$(1\,\,\,\,{n-1})$}
\put(-344,232){\large$\cdot\fuu$}
\put(-270,245){\large$(1\,\,\,{-n-1})$}
\put(-235,232){\large$\cdot\fuu$}
\put(-209,245){\large$(1\,\,\,{-n+1})$}
\put(-175,232){\large$\cdot\fud$}
\put(-100,245){\large$(1\,\,\,{-n+1})$}
\put(-66,232){\large$\cdot\fdu$}
\put(66,245){\large$(1\,\,\,\,{n-1})$}
\put(96,232){\large$\cdot\fdd$}
\put(172,245){\large$(1\,\,\,{-n-1})$}
\put(206,232){\large$\cdot\fdd$}
\put(238,245){\large$(1\,\,\,\,{n+1})$}
\put(266,232){\large$\cdot\fud$}
\put(345,245){\large$(1\,\,\,\,{n+1})$}
\put(373,232){\large$\cdot\fdu$}
\put(473,239){\Large$c_3$}
\put(-305,380){\Large$c_2\cdot(1\,\,\,\,{-n+1})$}
\put(-81,380){\Large$c_2\cdot(1\,\,\,\,{n+1})$}
\put(141,380){\Large$c_3\cdot(1\,\,\,\,{n-1})$}
\put(357,380){\Large$c_3\cdot(1\,\,\,\,{-n-1})$}
\end{picture}
}
\caption{The interval $[1,c_2c_3]_{T\cup F}$.  Cycles $(a\,\,\,b)$ in the figure should be read as $(\!(a\,\,\,b)\!)_{2n}$.}
\label{perm and new}
\end{sidewaysfigure}
Figure~\ref{perm only} shows the interval $[\P_0,\P_{23}]$ corresponding to the subposet $\perm_{23}$ of $\C_{23}$.
\begin{figure}
\scalebox{0.8}{\includegraphics{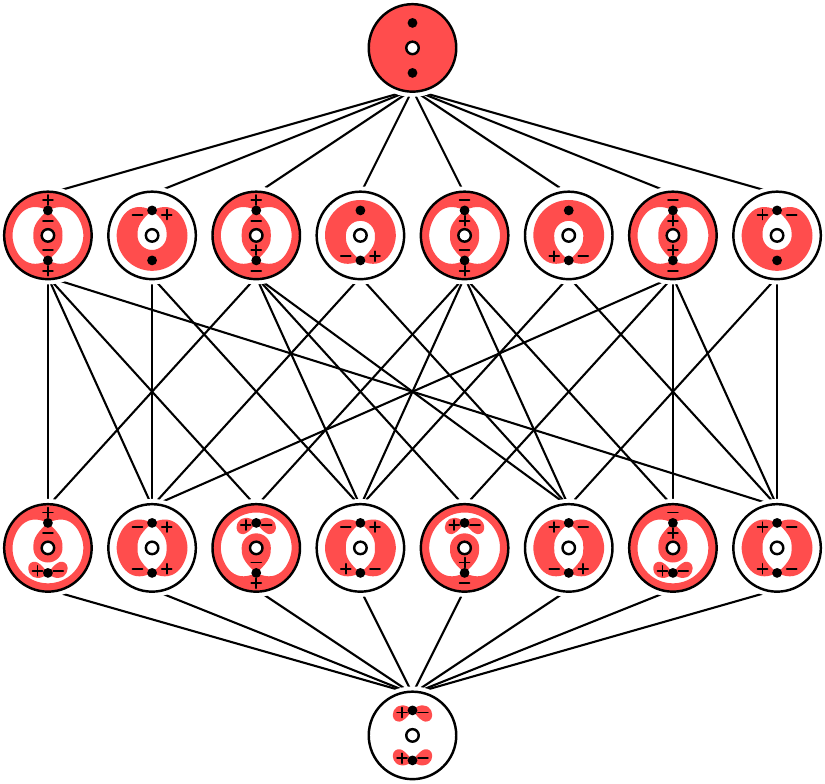}}
\caption{The non-lattice $[1,c_2c_3]_{T\cup L}$}
\label{perm only}
\end{figure}
In both figures, the noncrossing partition pictures omit the singleton blocks on inner and outer points.
(Alternatively, one could view these pictures as the case where $n=3$.
However, the Coxeter-theoretic results relating to these noncrossing partitions are only valid for $n\ge5$.)

Every element of $[1,c]_{T_H\cup F}\setminus[1,c]_T$ is above a translation, and thus (by Proposition~\ref{if a fact}) is only related, in $[1,c]_{T\cup F}$, to other elements of $[1,c]_{T_H\cup F}$.

Each permutation $\pi$ in $\McSul^D\setminus[1,c]_{T\cup L}$ factors as a permutation of ordinary and barred double points and a permutation of ordinary and barred outer and inner points.
Let $\nu$ be the map that takes $\pi$ to the ordered pair of these factors.

These considerations amount to a fairly detailed description of $\McSul^D$, which we summarize in Theorem~\ref{lat desc}, below.

\begin{theorem}\label{lat desc}
The lattice $\McSul^D$ is the union (in the sense of binary relations) of $[1,c]_{T\cup L}$ with $[1,c]_{T_H\cup F}=\perm_1\times\C_{23}$.
The intersection of $[1,c]_{T\cup L}$ with $[1,c]_{T_H\cup F}$ is $\perm_1\times\perm_{23}$.
The map $\nu$ is an isomorphism from the subposet of $\McSul^D$ induced by $\McSul^D\setminus[1,c]_{T\cup L}$ to $\perm_1\times\New$.
\end{theorem}

\begin{remark}\label{model can be used}
Theorem~\ref{lat desc} falls short of a planar model for the lattice $\McSul^D$, but does provide effective tools for dealing concretely with the lattice.
Most elements of the lattice can be represented by their associated noncrossing partition $\tNCDc$.
The remaining elements can be represented by a pair $(\P,\pi)$ consisting of a noncrossing partition $\P\in[\P_0,\P_1]$ and an affine barred even-signed permutation $\pi\in\New$ (one of the 18 permutations appearing in Figure~\ref{perm and new}).
Order relations between elements of $\tNCDc$ are made in the usual way.
Order relations between two pairs are componentwise, in $\tNCDc$ and in the poset shown in Figure~\ref{perm and new}.
An order relation between a pair $(\P,\pi)$ and an element $\Q\in\tNCDc$ can only occur when $\perm^D(\Q)\in\perm_1\times\perm_{23}$.
This is when $\Q$ can be decomposed into two pieces $\Q_1$ and $\Q_{23}$, where $\Q_{23}$ is one of the diagrams shown in Figure~\ref{perm only} and $\Q_1$ is what remains when a thin annulus containing the double points is removed from $\Q$.
Putting pairs of trivial blocks at the double points, we can think of $\Q_1$ as an element of $\tNCDc$.
Order relations between $(\P,\pi)$ and $\Q$ are made by componentwise comparisons of $(\P,\pi)$ and $(\Q_1,\Q_{23})$, comparing $\P$ and $\Q_1$ in $\tNCDc$ and comparing $\pi$ and $\Q_{23}$ according to Figure~\ref{perm and new}.
\end{remark}

\subsection{General Coxeter elements}\label{change cox sec}
Recall that in Sections~\ref{McSul D sec}--\ref{lat sec}, we took a special choice of Coxeter element. 
This choice simplified the notation and the wording of arguments, but was not essential mathematically.
We now briefly state the results of Sections~\ref{barred subsec} and~\ref{lat sec} for an arbitrary Coxeter element $c$.

To begin, we define sets $L_\out$ and $F_\doub$ and $F=L_\out\cup F_\doub$, depending on $c$, such that the special choice of $c$ recovers the earlier definitions.
Given a Coxeter element $c$, represented as a  partition of $\set{\pm1,\ldots,\pm(n-1)}$ into outer, inner, and double points, let $L_\out$ be the set of loops $\ell_i$ for $i$ outer.
Let $F_\doub=\set{\fuu,\fud,\fdu,\fdd}$, where symbols $\fuu$, etc.\ now refer to the double points associated to $c$.
For example
\[
\fuu=(\!(\cdots\,i\,\,\,\,\bar{i}\,\,\,\,i+2n\,\cdots)\!)_{2n}\cdot(\!(\cdots\,j\,\,\,\,\overline{j}\,\,\,\,j+2n\,\cdots)\!)_{2n},
\]
where $i$ is the positive upper double point and $j$ is the positive lower double point, and $\fud$, $\fdu$, and $\fdd$ are defined analogously.
Let $\McSul_c^D$ denote $[1,c]_{T\cup F}$.
We define $\perm_1$, $\perm_{23}$, $\C_{23}$, and $\New$ exactly as before, but using the new arbitrary~$c$.

As in Lemma~\ref{d annulus label}, source-sink moves change the partition of $\set{\pm1,\ldots,\pm(n-1)}$ into outer, inner, and double points.
Every Coxeter element can be obtained from our special choice of Coxeter element by a sequence of source-sink moves.
Every source-sink move is conjugation by an element of $S$, and this conjugation makes the same rearrangements of the cycle notation of elements of $\Stildedes$, $\Stildejes$, and $\Stildebes$ as it makes of the numbering of outer, inner, and double points.
This conjugation is an isomorphism of the relevant partial orders on the groups and on noncrossing partitions.
In particular, the main results of Sections~\ref{barred subsec} and~\ref{lat sec} hold for all~$c$.
Here, we state Proposition~\ref{bes prop ever}, Corollary~\ref{D big lattice}, and Theorem~\ref{lat desc} for arbitrary~$c$.
 
\begin{theorem}\label{gen cox}
Let $c$ be an arbitrary Coxeter element of $\Stildedes$, and define $L_\out$ and $F_\doub$ according to~$c$.
Then
\begin{enumerate}[\quad\rm\bf1.]
\item
The group $\Stildebes$ is generated by $T\cup L_\out\cup F_\doub$.
\item
The interval $[1,c]_{T\cup L_\out\cup F_\doub}$ in $\Stildebes$ is a lattice.
\item
The lattice $\McSul^D_c$ is the union (in the sense of binary relations) of $[1,c]_{T\cup L}$ with $[1,c]_{T_H\cup F}=\perm_1\times\C_{23}$.
The intersection of $[1,c]_{T\cup L}$ with $[1,c]_{T_H\cup F}$ is $\perm_1\times\perm_{23}$.
The map $\nu$ is an isomorphism from the subposet of $\McSul_c^D$ induced by $\McSul_c^D\setminus[1,c]_{T\cup L}$ to $\New\times\perm_{23}$.
\end{enumerate}
\end{theorem}

\section{Affine type $\afftype{B}$}
\label{aff type b}
In this section, we construct a planar model for $[1,c]_T$ in type~$\afftype{B}$ by folding the type-$\afftype{D}$ model.
Details on such folding are in \cite[Section~2.3]{affncA}, with an example (folding $\afftype{A}$ to~$\afftype{C}$) in \cite[Section~4.1]{affncA}.
For the Coxeter-theoretic results in type $\afftype{B}_{n-1}$, we need $n\ge4$, but some results on the planar models work for $n\ge3$.

\subsection{Affine singly even-signed permutations}\label{aff sing sec}
We build a root system of affine type $\afftype B_{n-1}$ in the vector space $V$ from Section~\ref{C sec}.
The simple roots are $\alpha_0=\e_1$, $\alpha_i=\e_{i+1}-\e_i$ for $i=1,\ldots,n-2$ and $\alpha_{n-1}=\e_{n+1}-\e_{n-2}$.
The simple coroots are $\alpha\ck_0=2\alpha_0=2\e_1$ and otherwise $\alpha_i\ck=\alpha_i$.
The corresponding set $S$ of simple reflections contains $s_0=(1 \,\, -1)_{2n}$, $s_i=(\!(i \,\,\,i+1)\!)_{2n}$ for $i = 1,\ldots,n-2$, and $s_{n-1}=(\!(n-2 \,\,\,n+1)\!)_{2n}$.
These simple reflections generate a Coxeter group of type $\afftype B_{n-1}$, realized as the group $\Stildeses$ of \newword{affine singly even-signed permutations}.   
These are the permutations $\pi:\integers\to\integers$ such that
\begin{itemize}
    \item $\pi(i + 2n) = \pi(i) + 2n$,
    \item $\pi(i) = -\pi(-i)$, and
    \item $\set{i\in\integers:i<n,\pi(i)>n}$ has an even number of elements.
\end{itemize}
(Each such $\pi$ fixes all multiples of $n$.)
These are the ``singly'' even-signed affine permutations because they have only one evenness condition, in contrast to the doubly even-signed permutations.
The set of reflections in $\Stildeses$ is 
\[T=\set{(\!(i\,\,\,j)\!)_{2n}:j\not\equiv\pm i\cmod{2n}}\cup\set{(i\,\,\,-i)_{2n}:i\not\equiv0\cmod{2n}}.\]

In type~$\afftype{D}$, we considered the Coxeter group $\Stildedes$ and larger groups~$\Stildejes$ and~$\Stildebes$.
In type~$\afftype{B}$, we consider the Coxeter group $\Stildeses$ and the larger group $\Stildes$, the group of affine signed permutations from Section~\ref{C sec}.
Recall the set $L$ of loops from Section~\ref{D big sec}.

\begin{proposition}\label{big group gen B}
The group $\Stildes$ is generated by $S\cup\set{\ell_1}$ or by~$T\cup L$.
\end{proposition}
\begin{proof}
Since $L\subseteq\Stildes$ and $S$ generates $\Stildeses$ and $S\subseteq T$, we need only check that $\Stildeses\cup\set{\ell_1}$ generates~$\Stildes$.
This follows from the first assertion of Lemma~\ref{loop big neg}.
\end{proof}

For the purposes of folding, we also realize the Coxeter group of type $\afftype{B}_{n-1}$ as a subgroup of a Coxeter group of type~$\afftype{D}_n$.
(The subgroup will be called $W$ to distinguish it from $\Stildeses$.)
Create a totally ordered set $\integers'$ isomorphic to $\integers$ by inserting symbols $(i-1)'$ and $(i+1)'$ with $i-1<(i-1)'<i<(i+1)'<i+1$ for every $i\equiv0\cmod{2n}$.
There is a natural notion of ``modulo $2n$'' on $\integers'$, but there are $2n+2$ classes ``modulo $2n$''.

Consider the group $W'$ of doubly even-signed permutations of $\integers'$, a Coxeter group of type $\afftype{D}_n$, isomorphic to $\Stilde^{\mathrm{des}}_{2n+2}$. 
Numbering the simple reflections with indices in~$\integers'$, the simple reflections are $s'_0=(\!(1' \,\,{-1})\!)_{2n}$, $s'_{1'}=(\!(1' \,\,\,1)\!)_{2n}$, $s'_i=(\!(i \,\,\,{i+1})\!)_{2n}$ for $i=2,\ldots,n-1$, and $s'_{n-1}=(\!(n-1 \,\,\,n+2)\!)_{2n}$,
where subscripts $2n$ are interpreted in the sense of ``modulo $2n$'' of the previous paragraph.
The simple reflections of~$W'$ are $S'=\set{s'_0,s'_{1'},s'_1,\ldots,s'_{n-1}}$.
The background on doubly even-signed permutations in Section~\ref{aff type d} is easily rephrased in terms of these permutations of~$\integers'$.

Let $\chi$ be the automorphism of $W'$ given by $\chi(\pi)=(1'\,\,\,{(-1)'})_{2n}\cdot\pi\cdot(1'\,\,\,{(-1)'})_{2n}$.
We also extend $\chi$ to maps, with the same name, on the larger sets of affine signed permutations of $\integers'$ and of affine barred signed permutations of~$\integers'$.
In the latter setting, we must write $\chi(\pi)=(\!(1'\,\,\,{(-1)'})\!)_{2n}\cdot\pi\cdot(\!(1'\,\,\,{(-1)'})\!)_{2n}$, where the double parentheses have the expanded meaning defined in Section~\ref{barred subsec}. 

The action of $\chi$ swaps $1'$ and~$(-1)'$ in the cycle notation of $\pi\in W'$, and similarly ``modulo~$2n$''.
The map $\chi$ is a diagram automorphism of $W'$, in the sense that it permutes $S'$ and preserves the Coxeter diagram.
Also, $\chi$ takes reflections to reflections, so it is an automorphism of the absolute order on $W'$.
Moreover, $\chi$ takes loops to loops, so it is also an automorphism of the partial order relative to the alphabet of reflections and loops in the larger group of jointly even-signed permutations of $\integers'$.
Finally, $\chi$ also preserves the set $F_\doub$ of factored loops on double points (whether $1'$ or $1$ is a double point), so it is an automorphism of the partial order on affine barred even signed permutations relative to the alphabet of reflections and factored translations.
The following fact is immediate.

\begin{lemma}\label{fixed}
An affine barred even-signed permutation of $\integers'$ is fixed by $\chi$ if and only if it fixes $1'$ and $(-1)'$ or has $2$-cycles $(\!(1'\,\,\,(-1)')\!)_{2n}$.
\end{lemma}

Let $W$ be the subgroup of $W'$ consisting of elements fixed by $\chi$.
We define an isomorphism $\eta:\Stildeses\to W$ that sends each singly even-signed permutation of~$\integers$ to a doubly even-signed permutation of $\integers'$ (fixed by $\chi$) by adjoining the singleton cycles $(1')_{2n}$ and $((-1)')_{2n}$ or the $2$-cycles $(1'\,\,\,{(-1)'})_{2n}$ to make the result doubly even-signed.
The inverse takes an element fixed by $\chi$ to a singly even-signed permutation of $\integers$ by deleting the cycles $(1')_{2n}$ and $((-1)')_{2n}$ or the cycle $(1'\,\,\,{(-1)'})_{2n}$.

The subgroup $W$ is a Coxeter group with simple reflections $\eta(S)$.
These simple reflections are $\eta(s_0)=s'_0s'_{1'}=(1 \,\, {-1})_{2n}(1' \,\,\,{(-1)'})_{2n}$, $\eta(s_i)=s'_i=(\!(i \,\,\,i+1)\!)_{2n}$ for $i=1,\ldots,n-2$, and $\eta(s_{n-1})=s'_{n-1}=(\!(n-1 \,\,\,n+2)\!)_{2n}$.
The reflections $\eta(T)$ and loops $\eta(L)$ in $W$ are obtained by applying $\eta$ to the reflections and loops in $\Stildeses$.

The Coxeter group $W$ is a folding of $W'$ in the sense of \cite[Section~2.3]{affncA}.
Each Coxeter element of $W$ is also a Coxeter element of $W'$, because the simple reflection~$\eta(s_0)$ of $W$ is the product of the two simple reflections $s'_0$ and $s'_{1'}$ of~$W'$.
Thus we can (in the first instance) encode a Coxeter element $c$ of $W$ as a partition of $\set{1',(-1)',\pm1,\ldots,\pm(n-1)}$ into outer, inner, and double points.
Since $\eta(s_0)=s'_0s'_{1'}$ the Coxeter element $c$ of $W'$ has $s_0'$ and $s'_{1'}$ either both preceding~$s'_1$ or both following $s'_1$, so $1'$ and $(-1)'$ are the upper double points.
Since the choice of upper double points never varies, in fact we encode a Coxeter element $c$ of $W$ as a partition of $\set{\pm1,\ldots,\pm(n-1)}$ into outer, inner, and double points.
The numbers $\pm(n-1)$ are lower double points if and only if $s_{n-1}$ and $s_{n-2}$ either both precede or both follow $s_{n-3}$ in $c$, and otherwise, $\pm(n-2)$ are lower double points.
The remaining numbers in $\set{\pm1,\ldots,\pm(n-1)}$ are inner or outer points.
For $i\in\set{1,\ldots,n-2}$, if $\pm i$ is not a double point, then $i$ is outer and~$-i$ is inner if and only if $s_{i-1}$ precedes~$s_i$ in $c$.
Otherwise, $i$ is inner and $-i$ is outer.
If $\pm(n-1)$ is not a double point, then $n-1$ is outer and $-(n-1)$ is inner if and only if $s_{n-3}$ precedes $s_{n-1}$ in $c$.
Otherwise $n-1$ is inner and $-(n-1)$ is outer.
The following lemma is now immediate from Lemma~\ref{d annulus label} by the isomorphism~$\eta$.

\begin{lemma}\label{b annulus label}
Let $c$ be a Coxeter element of $\Stildeses$, represented as a partition of $\set{\pm1,\ldots,\pm(n-1)}$ into outer, inner, and double points.
If $a_1, \ldots, a_{n-2}$ are the outer points in increasing order and~$p$ is the positive label of the lower double point, then $c$ is the permutation
\[
(\!(\cdots\,a_1\,\,\,a_2\,\cdots\,a_{n-2} \,\,\,a_1+2n\,\cdots)\!)(p \,\,\,{-p+2n})_{2n}
\]
The corresponding permutation $\eta(c)\in W$ has an additional factor $(1'\,\,\,(-1)')_{2n}$.
\end{lemma}

We extend $\eta$ to an isomorphism (also called $\eta$) from $\Stildes$ to the group of jointly even-signed permutations of $\integers'$ fixed by $\chi$.
The map acts by adjoining the singleton cycles $(1')_{2n}$ and $((-1)')_{2n}$ or the $2$-cycles $(1'\,\,\,{(-1)'})_{2n}$ as needed to make the result jointly even-signed.
The inverse deletes the cycles $(1')_{2n}$ and $((-1)')_{2n}$ or deletes the cycle $(1'\,\,\,{(-1)'})_{2n}$.


We write $[1,c]^B_T$ for the interval in the absolute order on $\Stildeses$ and $[1,c]^B_{T\cup L}$ for the interval in the prefix/postfix/subword order on $\Stildes$ relative to the alphabet~$T\cup L$.
(See Proposition~\ref{big group gen B}.)
Recalling that $\eta(c)$ is a Coxeter element of $W$ and thus a Coxeter element of $W'$, we write $[1,\eta(c)]^D_{T'}$, $[1,\eta(c)]^D_{T'\cup L'}$, and $[1,\eta(c)]^D_{T'\cup F'}$ for the type-$\afftype{D}$ intervals defined earlier in the paper, but with elements that are affine permutations of $\integers'$ or affine barred permutations of $\integers'\cup\set{\bar\imath:i\in\integers'}$ .
We gather two facts that will be useful later.

\begin{lemma}\label{1cT B}
Let $c$ be a Coxeter element of $\Stildeses$, represented as a partition of $\set{\pm1,\ldots,\pm(n-1)}$ into outer, inner, and double points.
Then $\eta$ is an isomorphism from $[1,c]^B_T$ to the subposet of $[1,\eta(c)]^D_{T'}$ induced by elements fixed by~$\chi$.
\end{lemma}
\begin{proof}
An easy general fact about folding \cite[Proposition~2.2]{affncA} says that the interval $[1,\eta(c)]_{\eta(T)}$ in $W$ is the subposet of the interval $[1,\eta(c)]^D_{T'}$ induced by the set of elements fixed by $\chi$.
The lemma follows.
\end{proof}

\begin{remark}\label{harder}
The result about $[1,c]^B_{T\cup L}$ and $[1,\eta(c)]^D_{T'\cup L'}$ that is analogous to Lemma~\ref{1cT B} will fall out as a corollary (Corollary~\ref{isom b too}) as we prove our main results.
\end{remark}

\begin{lemma}\label{1cTF B}
Let $c$ be a Coxeter element of $\Stildeses$, represented as a partition of $\set{\pm1,\ldots,\pm(n-1)}$ into outer, inner, and double points.
Then the subposet of $[1,\eta(c)]^D_{T'\cup L'}$ induced by elements fixed by~$\chi$ equals the subposet of $[1,\eta(c)]^D_{T'\cup F'}$ induced by elements fixed by~$\chi$.
\end{lemma}

\begin{proof}
None of the factored translations $\fuu,\fud,\fdu,\fdd$ defined in Sections~\ref{barred subsec} and~\ref{change cox sec} are fixed by $\chi$, and furthermore, no element of $[1,\eta(c)]_{T'\cup F'}$ is fixed by $\chi$ if it is a product of elements of $T'\cup F'$ using exactly one of these factors $\fuu,\fud,\fdu,\fdd$.
Thus every element of $[1,\eta(c)]_{T'\cup F'}$ fixed by $\chi$ is actually in $[1,\eta(c)]_{T'\cup L'}$.
\end{proof}

\subsection{Symmetric noncrossing partitions of an annulus with one double point}\label{sym one sec}
We now introduce the combinatorial model for $[1,c]^B_{T\cup L}$ and $[1,c]_T$.
As in Section~\ref{nc D sec}, the definitions and basic results on noncrossing partitions are a special case of definitions and results in \cite[Section~3]{surfnc}.

In Section~\ref{aff sing sec}, we encoded a choice of Coxeter element $c$ of $\Stildeses$ as a choice of whether $\pm(n-1)$ or $\pm(n-2)$ are the bottom double points and a choice, for the remaining elements $i\in\set{1,\ldots(n-1)}$, of whether $i$ or $-i$ is outer.
For some fixed~$c$, let $B$ be the annulus with one double point with these numbered points and let~$\phi$ be the symmetry (which also exchanges two bottom double points).
The symmetry $\phi$ also has a second fixed point in $B$, besides the bottom double points, corresponding to the top double points in the type-$\afftype{D}$ model.
We refer to this as the ``non-double fixed point'' of $\phi$.

\begin{remark}[Projecting to the Coxeter plane]\label{B proj rem}
%
%
%
One can prove the type-$\afftype{B}$ analog of Theorem~\ref{d orb proj} by similar arguments and thereby see how the symmetric annulus~$B$ with one double point arises naturally by projecting $\set{\e_i:\,i\in\integers,\,i\not\equiv0\mod{n}}$ (an orbit of the action of $\Stildeses$ on $V$) to the Coxeter plane in $V^*$.
We omit the details.
\end{remark}

We define \newword{boundary segments} of $B$ and \newword{symmetric ambient isotopy} in $B$ just as in Section~\ref{nc D defs sec}.
An \newword{arc} in $B$ is a non-oriented curve $\alpha$ in $B$, having endpoints at numbered points and satisfying these requirements:
\begin{itemize}
\item
$\alpha$ does not intersect itself except possibly at endpoints.
\item
$\alpha$ is disjoint from double points and the boundary of $B$, except at endpoints.
\item
$\alpha$ does not bound a monogon in $B$ (even if its endpoints are the two double points).
\item 
$\alpha$ does not combine with a boundary segment to bound a digon in $B$. 
\item
Either $\alpha=\phi(\alpha)$ or $\alpha$ and $\phi(\alpha)$ don't intersect, except possibly at endpoints.
\item
$\alpha$ and $\phi(\alpha)$ do not combine to form a digon in $B$ unless that digon contains a double point.
However, we do not rule out $\alpha$ and $\phi(\alpha)$ if, at a vertex of the digon, the two edges incident to that point are actually attached to opposite double points.
The left picture of Figure~\ref{arc digon} shows a digon that is ruled out.
(At one vertex, both edges attach to $+$ and at the other vertex, both edges attach to $-$.)
The right picture shows a digon that is allowed.
\begin{figure}
\begin{tabular}{cc}
\scalebox{0.5}{\includegraphics{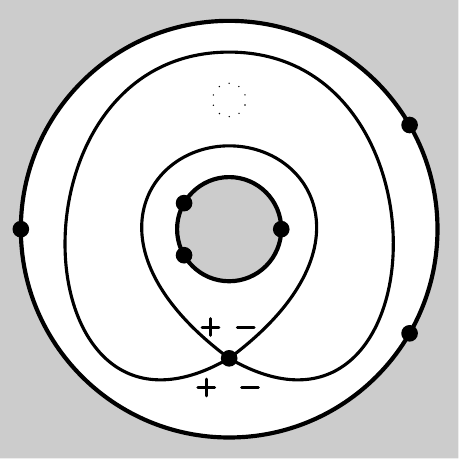}}&
\scalebox{0.5}{\includegraphics{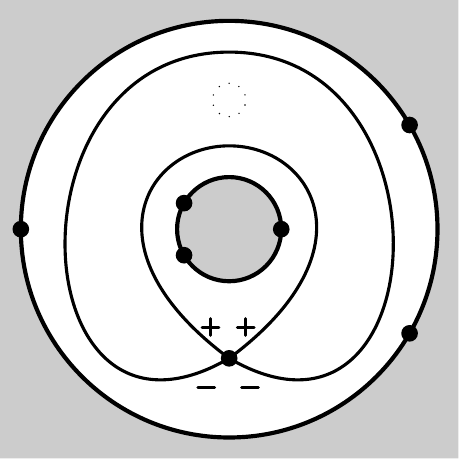}}
\end{tabular}
\caption{Left:  Not a symmetric pair of arcs;  Right:  A symmetric pair of arcs}
\label{arc digon}
\end{figure}
\end{itemize}
When $\alpha=\phi(\alpha)$, we call $\alpha$ a \newword{symmetric arc}.
When $\alpha$ and $\phi(\alpha)$ are disjoint except possibly at endpoints, we call $\alpha,\phi(\alpha)$ a \newword{symmetric pair of arcs}.
Symmetric arcs/pairs are considered up to isotopy and up to swapping $\alpha$ with~$\phi(\alpha)$.

An \newword{embedded block} in $B$ defined just as in Section~\ref{nc D defs sec} for embedded blocks in $D$, except that there is one more kind of symmetric block:
A degenerate \newword{curve block} consisting of a symmetric arc.
We note that in the annulus with only one double point, a stitched disk block must contain more numbered points than just the two double points on its boundary---otherwise the curves forming its boundary are not arcs.
In fact, the stitched disk block containing only double points is disallowed because its part is played by a symmetric degenerate block consisting of a symmetric arc with endpoints at the two copies of the double point.
A symmetric block contains either the double point or the non-double fixed point of $\phi$ in its interior, but not both.
A symmetric annular block contains both.

A \newword{(symmetric) noncrossing partition} of $B$ is defined exactly as in type $\afftype{D}$:  It is a collection $\P$ of disjoint embedded blocks such that every numbered point is in some block of $\P$, such that the action of $\phi$ permutes the blocks of $\P$, and having at most two annular blocks.
We also define the partial order on noncrossing partitions exactly as in type $\afftype{D}$.
We write $\tNCBc$ for the poset of noncrossing partitions of~$B$.
An annular block is \newword{non-dangling} if both components of its boundary contain numbered points.
Otherwise it is \newword{dangling}.
Let $\tNCBcircc$ stand for the subposet of $\tNCBc$ induced by noncrossing partitions with no dangling annular blocks.

\begin{example}\label{B exs}
Figure~\ref{nc ex fig B} shows examples of symmetric noncrossing partitions of an annulus with one double point in the case where $n=7$ and $c=s_2s_5s_1s_0s_4s_6s_3$.
\end{example}

\begin{figure}
\scalebox{0.47}{\includegraphics{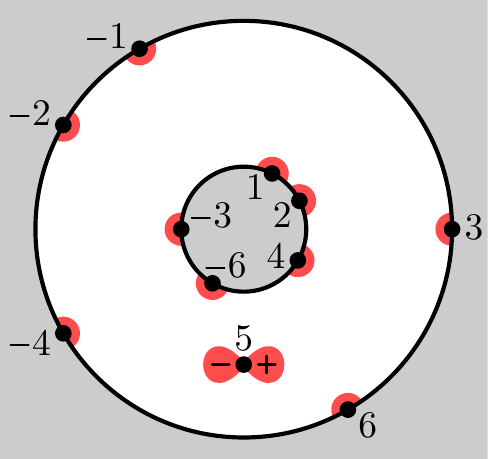}}
\quad 
\scalebox{0.47}{\includegraphics{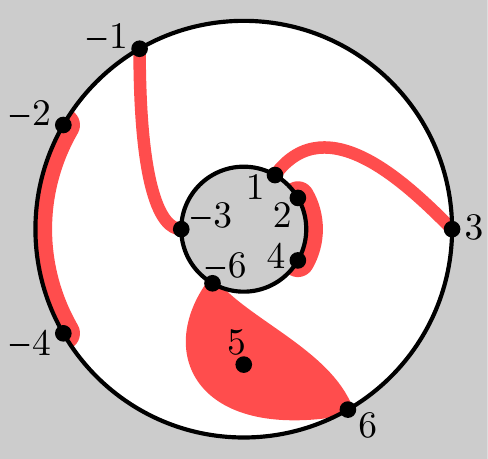}}
\quad 
\scalebox{0.47}{\includegraphics{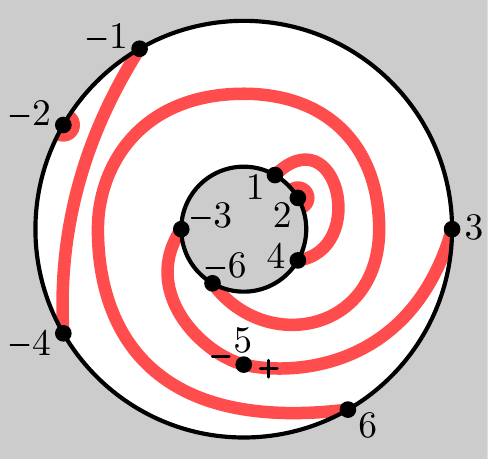}}\\[12pt]
\scalebox{0.47}{\includegraphics{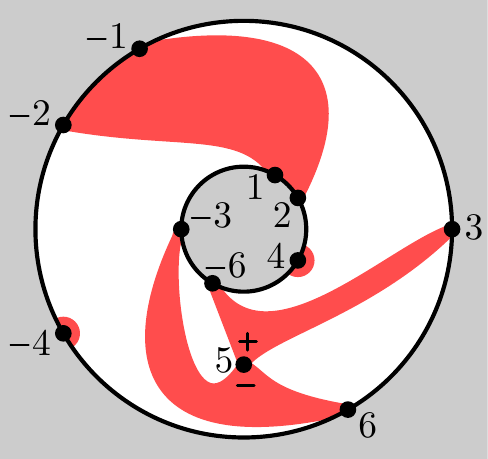}}
\quad 
\scalebox{0.47}{\includegraphics{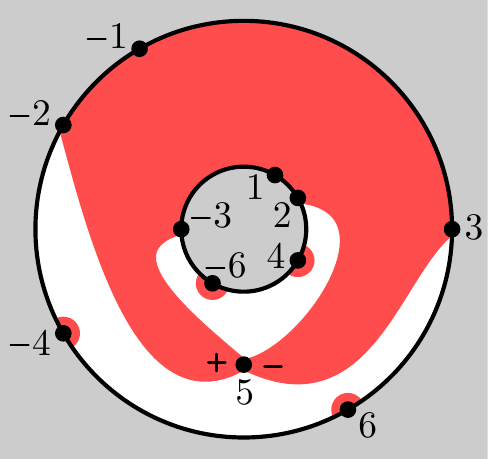}}
\quad 
\scalebox{0.47}{\includegraphics{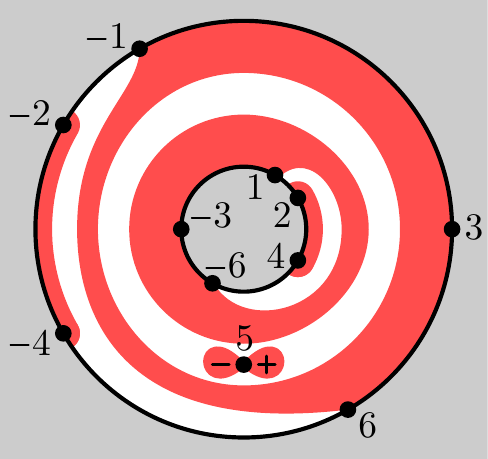}}\\[12pt]
\scalebox{0.47}{\includegraphics{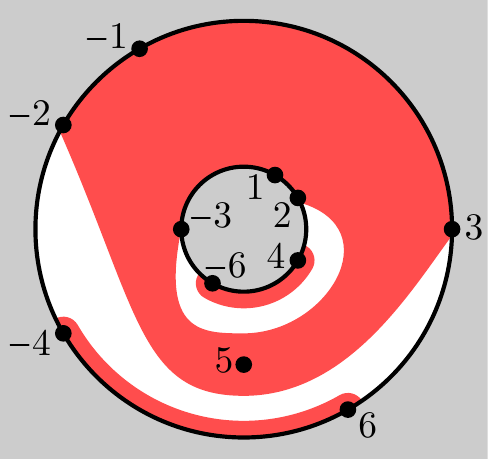}}
\quad 
\scalebox{0.47}{\includegraphics{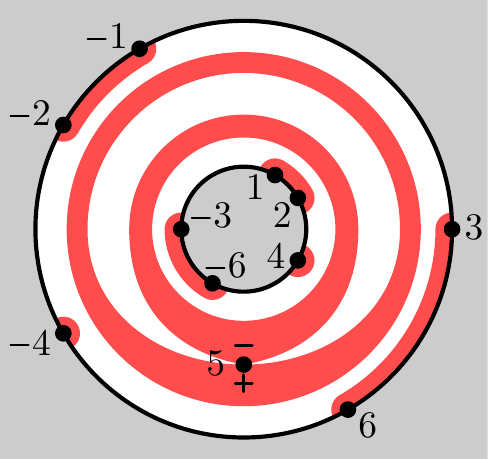}}
\quad 
\scalebox{0.47}{\includegraphics{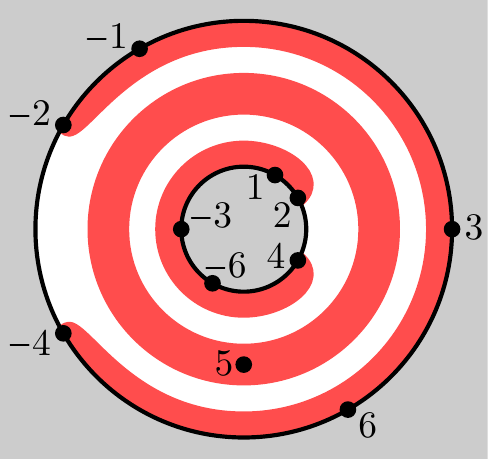}}\\[12pt]
\scalebox{0.47}{\includegraphics{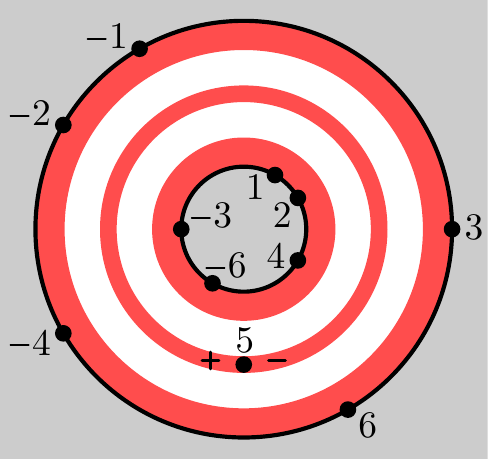}}
\caption{Symmetric noncrossing partitions of an annulus with one double point}
\label{nc ex fig B}
\end{figure}

The following theorem is \cite[Theorem~4.6]{surfnc}, a special case of \cite[Theorem~3.18]{surfnc}.

\begin{theorem}\label{B tilde main}
The poset $\tNCBc$ of symmetric noncrossing partitions of an annulus with $n-2$ marked points on each boundary and one pair of double points is graded, with rank function given by $n-1$ minus the number of symmetric pairs of distinct non-annular blocks plus the number of symmetric annular blocks.
\end{theorem}

As in type $\afftype{D}$, one may go up by a cover from $\P$ by finding a simple symmetric pair of connectors and forming the augmentation of $\P$ along the pair.
However, unlike in type $\afftype{D}$, the annulus $B$ with one double point admits symmetric arcs, and there is a notion of augmentation along a simple symmetric connector.
Full details on cover relations in $\tNCBc$ are in \cite[Section~3]{surfnc} in the generality of symmetric marked surfaces with double points.
In particular \cite[Figure~11]{surfnc} and \cite[Figure~12]{surfnc} include examples specific to $\tNCBc$.


\subsection{Isomorphisms}
We define a map $\perm^B:\tNCBc\to\Stildes$ by slightly modifying the definition of $\perm^D$.
We read cycles from each block as for $\perm^D$, except that for symmetric blocks containing the non-double fixed point of $\phi$, we don't record the tiny cycles that would come from the upper double points in type~$\afftype{D}$.
We treat symmetric curve blocks as we would treat non-degenerate symmetric disks.
We will prove the following theorem.

\begin{theorem}\label{isom b}
The map $\perm^B:\tNCBc\to\Stildes$ is an isomorphism from $\tNCBc$ to the interval $[1,c]^B_{T\cup L}$ in $\Stildes$.   
It restricts to an isomorphism from $\tNCBcircc$ to the interval $[1,c]^B_T$ in $\Stildeses$.
\end{theorem}

\begin{example}\label{perm B ex}
We apply $\perm^B$ to the elements of $\tNCBc$ shown in Figure~\ref{nc ex fig B}.
Label the top row of pictures left to right as $\P_1$, $\P_2$, $\P_3$, and label the next row left to right as $\P_4$, $\P_5$, $\P_6$, and so forth.  
{\allowdisplaybreaks
\begin{align*}
\perm^B(\P_1)&=(\!(1)\!)_{14}(\!(2)\!)_{14}(\!(3)\!)_{14}(\!(4)\!)_{14}(\!(5)\!)_{14}(\!(6)\!)_{14}\\
\perm^B(\P_2)&=(\!(1\,\,\,3)\!)_{14}(\!(2\,\,\,4)\!)_{14}(6\,\,\,8)_{14}(5\,\,\,9)_{14}\\
\perm^B(\P_3)&=(\!(1\,\,\,4)\!)_{14}(\!(2)\!)_{14}(\!(3\,\,\,5)\!)_{14}(6\,\,\,22)_{14}\\
\perm^B(\P_4)&=(1\,\,{-2}\,\,{-1}\,\,\,2)_{14}(\!(3\,\,\,5\,\,\,8)\!)_{14}(\!(4)\!)_{14}\\
\perm^B(\P_5)&=(1\,\,{-3}\,\,{-9}\,\,{-2}\,\,{-1}\,\,\,3\,\,\,9\,\,\,2)_{14}(\!(4)\!)_{14}(\!(6)\!)_{14}\\
\perm^B(\P_6)&=(\!(\cdots\,1\,\,{-3}\,\,{-6}\,\,{-13}\,\cdots)\!)_{14}(\!(2\,\,\,4)\!)_{14}(\!(5)\!)_{14}\\
\perm^B(\P_7)&=(\!(\cdots\,1\,\,{-3}\,\,{-12}\,\,{-13}\,\cdots)\!)_{14}(\!(4\,\,\,8)\!)_{14}(5\,\,\,9)_{14}\\
\perm^B(\P_8)&=(\!(1\,\,\,2)\!)_{14}(\!(3\,\,\,6)\!)_{14}(\!(4)\!)_{14}(\!(\cdots\,5\,\,\,{-9}\,\cdots)\!)_{14}\\
\perm^B(\P_9)&=(\!(1\,\,\,{-3}\,\,\,{-6}\,\,\,{-10}\,\,\,2)\!)_{14}(5\,\,\,9)_{14}\\
\perm^B(\P_{10})&=(\!(\cdots\,1\,\,{-3}\,\,{-6}\,\,{-10}\,\,\,{-12}\,\,\,{-15}\,\cdots)\!)_{14}(5\,\,\,23)_{14}
\end{align*}
}
\end{example}

To prove Theorem~\ref{isom b}, we give an alternative description of $\perm^B$ and prove results analogous to (and using) results from the proof of Theorem~\ref{isom d} (type $\afftype D$).

Let $c$ be a Coxeter element of $\Stildeses$, so that $\eta(c)$ is a Coxeter element of $W$ and of~$W'$.
Write $\tNCDetac$ for the noncrossing partitions of affine type $\afftype D$ with numbered points $\set{1',(-1)',\pm1,\ldots,\pm(n-1)}$, designated outer, inner, and double according to the Coxeter element $\eta(c)$.
The upper double points are $1'$ and $(-1)'$.
Similarly, write $\tNCDcircetac$ for the subposet of $\tNCDetac$ consisting of elements with no dangling annular blocks.
Reuse the symbol $\chi$ for the map on $\tNCDetac$ that swaps the double points $1'$ and $(-1)'$.
A noncrossing partition in $\tNCDetac$ is fixed by $\chi$ if and only if it either has singleton blocks at $1'$ and $(-1)'$ or has a symmetric disk or annular block containing the upper double points in its interior.
The map $\perm^D$ restricts to an isomorphism from the subposet of $\tNCDetac$ induced by noncrossing partitions fixed by $\chi$ to the subposet of $[1,\eta(c)]_{T'\cup L'}$ consisting of elements fixed by $\chi$.

There is an isomorphism $\zeta$ from $\tNCBc$ to the subposet of $\tNCDetac$ consisting of noncrossing partitions that are fixed by $\chi$:
Given $\P\in\tNCBc$, if $\P$ has a symmetric block containing the non-double $\phi$-fixed point of $B$, in $\zeta(\P)$ that block becomes a symmetric block containing the upper double points of $D$ in its interior and the other blocks of $\P$ are blocks of $\zeta(\P)$.
Otherwise, $\zeta(\P)$ is identical with $\P$ except that $\zeta(\P)$ has singleton blocks at the upper double points in $D$.
We see that $\perm^B=\eta^{-1}\circ\perm^D\circ\zeta$.
Since $\eta$ and $\zeta$ are bijections, Proposition~\ref{one to one d} implies the following proposition.

\begin{prop}\label{one to one b}
The map $\perm^B$ is one-to-one.
\end{prop}

\begin{prop}\label{cov ref loop b}
Suppose $\P,\Q\in\tNCBc$ have $\P\covered\Q$.
Then there exists a reflection or loop $\tau\in T\cup L$ such that $\perm^B(\Q)=\tau\cdot\perm^B(\P)$.
\end{prop}
\begin{proof}
If $\zeta(\P)\covered\zeta(\Q)$ in $\tNCDetac$, then Proposition~\ref{cov ref loop} says that there exists $\tau\in T'\cup L'$ such that $\perm^D(\zeta(\Q))=\tau\cdot\perm^D(\zeta(\P))$.
Since both $\zeta(\P)$ and $\zeta(\Q)$ are fixed by $\chi$, the reflection or loop $\tau$ fixes $1'$ and is thus in $T\cup L$.
If $\zeta(\P)$ is not covered by $\zeta(\Q)$ in $\tNCDetac$, then since $\P\covered\Q$ in $\tNCBc$, we compare Theorem~\ref{B tilde main} with Theorem~\ref{D tilde main} (replacing $n$ by $n+1$ in the latter).
The difference in rank between $\P$ and $\Q$ is $1$, so two symmetric pairs of nonannular blocks in $\P$ were combined into one symmetric pair of nonannular blocks in $\Q$, a symmetric pair of nonannular blocks was combined to form a symmetric disk block, or a symmetric block was changed to a symmetric annular block.
The only way for the difference in rank between $\zeta(\P)$ and $\zeta(\Q)$ to be greater than $1$ is if also $\zeta(\P)$ has a pair of trivial blocks at the upper double point that are contained in a symmetric block in $\zeta(\Q)$.
Thus $\Q$ was obtained from $\P$ by combining a symmetric pair of non-annular blocks into a symmetric disk block containing the upper (non-double) fixed point of $\phi$ or by changing a symmetric block containing the (lower) double point into a symmetric annular block.
Therefore, there exists $\R\in\tNCDetac$ such that $\zeta(\P)\covered\R\covered\zeta(\Q)$.
Indeed, there are two choices of $\R$:
One is obtained from $\zeta(\P)$ by a simple symmetric pair of connectors that connects the upper double points to a symmetric pair of blocks (or a symmetric disk block) and the other is the same but with the upper double points swapped.
The reflections $\tau_1$ and $\tau_2$ guaranteed by Proposition~\ref{cov ref loop} for $\zeta(\P)\covered\R$ with these two choices of $\R$ are also related by swapping the upper double points.
They commute, and $\tau_1\cdot\tau_2\cdot\zeta(\P)=\zeta(\Q)$.
But $\tau_1\cdot\tau_2$ is $\eta(\tau)$ for a reflection in $T$ of the form $(i\,\,\,-i)_{2n}$.
\end{proof}

\begin{prop}\label{perm inv cov b}
Suppose $\sigma\covered\pi$ in $[1,c]_{T\cup L}$ and $\pi=\perm^B(\Q)$ for some $\Q\in\tNCBc$.
Then there exists $\P\in\tNCBc$ such that $\sigma=\perm^B(\P)$ and $\P\covered\Q$.
\end{prop}
\begin{proof}
Let $\tau$ be the element of $T\cup L$ such that $\pi=\tau\sigma$.
If $\tau\in T'\cup L'$, then $\eta(\sigma)\covered\eta(\pi)=\tau\eta(\sigma)$.
Since $\eta(\pi)=\perm^D(\zeta(\Q))$, by Proposition~\ref{perm inv cov}, there exists $\R\in\tNCDetac$ such that $\eta(\sigma)=\perm^D(\R)$ and $\R\covered\zeta(\Q)$.
Taking $\P=\zeta^{-1}(\R)\in\tNCBc$, we have $\sigma=\perm^D(\P)$ and $\P\covered\Q$.

If $\tau\not\in T'\cup L'$, then $\eta(\tau)=\tau_1\tau_2$ for $\tau_1$ and $\tau_2$ as in the proof of Proposition~\ref{cov ref loop b}.
Arguing as in the paragraph above, twice, we obtain $\R_1,\R_2\in\tNCDetac$ such that $\eta(\sigma)=\perm^D(\R_1)$, $\eta(\tau_2\pi)=\perm^D(\R_2)$, and $\R_1\covered\R_2\covered\zeta(\Q)$.
Indeed, there are precisely two choices for $\R_2$, given by reversing the roles of $\tau_1$ and $\tau_2$.
Since $\tau_1$ and $\tau_2$ are related by $\chi$, these two choices for $\R_2$ are related by~$\chi$.
Furthermore, $\R_1$ is fixed by $\chi$.
The two choices of $\R_2$ are the only elements of $\tNCDetac$ strictly between $\R_1$ and $\zeta(\Q)$, and they are not fixed by $\chi$.
Taking $\P=\zeta^{-1}(\R_1)\in\tNCBc$, we see that $\sigma=\perm^B(\P)$ and $\P\covered\Q$ in $\tNCBc$.
\end{proof}

\begin{lemma}\label{b at least n}
Every word for $c$ in the alphabet $T\cup L$ has at least $n$ letters.
\end{lemma}
\begin{proof}
Given $\tau\in T\cup L$, either $\eta(\tau)=\tau\in T'\cup L'$ or $\eta(\tau)$ is a product of two elements of $T'$, namely two reflections $\tau_1=(\!(1'\,\,\,i)\!)$ and $\tau_2=(\!({(-1)'}\,\,\,i)\!)$.
If $\pi\in[1,c]^B_{T\cup L}$ and $\tau\in T\cup L$, then $\eta(\tau\pi)=\eta(\tau)\eta(\pi)$, so by Lemma~\ref{only one}, $\vr(\eta(\tau\pi))\ge\vr(\pi)-2$ and also $\vr(\eta(\tau\pi))\ge\vr(\pi)-1$ whenever $\eta(\tau)=\tau$.
The case where $\vr$ decreases by $2$ is when $(\perm^D)^{-1}(\pi)$ has the upper double point in the interior of a block and $(\perm^D)^{-1}(\tau\pi)$ has trivial blocks at the upper double point.
Thus the decrease by~$2$ can happen only once when following a maximal chain in $[1,c]_{T\cup L}$ downwards from $c$.
Since $\vr(\eta(c))=n+1$, the lemma follows.
\end{proof}

\begin{proof}[Proof of Theorem~\ref{isom b}]
Using Proposition~\ref{cov ref loop b} and Lemma~\ref{b at least n}, we can argue as in the proof of Theorem~\ref{isom d} that $\perm^B$ is an order-preserving map from $\tNCBc$ into $[1,c]_{T\cup L}$.
Proposition~\ref{one to one b} says that $\perm^B$ is one-to-one.
By Proposition~\ref{perm inv cov b} and an easy induction, $\perm^B$ maps $\tNCBc$ onto $[1,c]_{T\cup L}$, and thus $\perm^B$ is a bijection from $\tNCBc$ to $[1,c]_{T\cup L}$.
Proposition~\ref{perm inv cov b} shows that the inverse of $\perm^B$ is also order-preserving, and thus $\perm^B$ is an isomorphism of posets.

The map $\zeta$ restricts to an isomorphism from $\tNCBcircc$ to the subposet of $\tNCDcircetac$ consisting of noncrossing partitions that are fixed by $\chi$.
Theorem~\ref{isom d} says that $\perm^D$ restricts to an isomorphism from $\tNCDcircetac$ to $[1,c]^D_T$.
It is immediate that it restricts further to an isomorphism from the subposet of $\tNCDcircetac$ consisting of elements fixed by $\chi$ to the subposet of $[1,c]^D_{T'}$ consisting of elements fixed by $\chi$.
Now by Lemma~\ref{1cT B}, the second assertion of the theorem follows from the first.
\end{proof}

Since $\perm^B=\eta^{-1}\circ\perm^D\circ\zeta$, the fact that $\perm^B$, $\perm^D$, and $\zeta$ are isomorphisms implies the following corollary, promised in Remark~\ref{harder}.

\begin{cor}\label{isom b too}
The map $\eta$ is an isomorphism from $[1,c]_{T\cup L}$ to the subposet of $[1,\eta(c)]_{T'\cup L'}$ consisting of noncrossing partitions fixed by $\chi$.
\end{cor}

Recall that Corollary~\ref{D big lattice} says that $[1,\eta(c)]^D_{T'\cup F'}$ is a lattice.
By the usual easy lattice-theoretic fact, the subposet of $[1,\eta(c)]^D_{T'\cup F'}$ consisting of elements fixed by~$\chi$ is a lattice.
Lemma~\ref{1cTF B} says that the subposet of $[1,\eta(c)]^D_{T'\cup L'}$ consisting of elements fixed by $\chi$ is the same lattice.
Corollary~\ref{isom b too} and Theorem~\ref{isom b} now imply the following corollaries.

\begin{cor}\label{B lattice}
The interval $[1,c]^B_{T\cup L}$ in $\Stildes$ is a lattice.
\end{cor}

\begin{cor}\label{B NC lattice}
$\tNCBc$ is a lattice.
\end{cor}

\subsection{Factored translations in affine type B}
We now show that $\tNCBc$ is isomorphic to McCammond and Sulway's lattice in affine type~$\afftype{B}_{n-1}$.
We follow the outline of Section~\ref{barred sec}, omitting some details.
We fix the Coxeter element 
\[c=s_0\cdots s_{n-1}=(\!(\cdots\,1\,\,\,2\,\cdots\,{n-2}\,\,\,\,\,{1+2n}\,\cdots)\!)\,({n-1}\,\,\,{n+1})_{2n},\] 
corresponding to the choice of $\pm(n-1)$ to be double points, $1,\ldots,n-2$ to be outer points, and $-1,\ldots,-n+2$ to be inner points.
As in type~$\afftype D$, statements for general~$c$ can be recovered using source-sink moves.
(See Section~\ref{change cox sec}.)

For this choice of Coxeter element, the horizontal reflections in $[1,c]_T$ are $(\!(i\,\,\,\,j)\!)_{2n}$ and $(\!(i\,\,\,\,{j-2n})\!)_{2n}$ for $1\le i<j\le n-2$ as well as $(\!({-n+1}\,\,\,\,{n-1})\!)_{2n}$ and $(\!({-n-1}\,\,\,\,{n+1})\!)_{2n}$.
These are the elements $\perm^B(\P)$ where $\P$ is of one of the following forms:
Either
$\P$ has exactly two nontrivial blocks, a pair of nonsymmetric disks, each containing exactly two numbered points, one containing two outer points and one containing two inner points; or~$\P$ has exactly one nontrivial block, a stitched disk containing no numbered points except the double points $\pm(n-1)$.

The orthogonal decomposition of $E_0$ is $U_0\oplus U_1\oplus U_2$ where $U_0$ is the span of $-2\rho_0+\rho_{n-2}+\rho_{n-1}$, $U_0\oplus U_1$ is the set of vectors $\sum_{i=0}^{n-1}c_i\rho_i$ with $c_{n-2}=c_{n-1}$, and~$U_2$ is the span of $\rho_{n-2}-\rho_{n-1}$.     


The translations in $[1,c]_T$ are ${(\!(\cdots\,a\,\,\,a+2n\,\cdots)\!)}(\!(\cdots\,b\,\,\,b-2n\,\cdots)\!)$ such that $a\in\set{1,\ldots,n-2}$ and $b\in\set{\pm(n-1)}$.
In $E$, the translations vectors are
\[
\rho_a-\rho_{a-1}
+
\begin{rcases}\begin{dcases}
-\rho_0&\text{if }a=1,\\
\rho_{n-1}&\text{if }a=n-2\\
\end{dcases}\end{rcases}
+
\{\pm\rho_{n-2}\mp\rho_{n-1}\text{ if }b=\pm(n-1)\}
\]

The factorization of translations in $[1,c]_T$ into two loops ${(\!(\cdots\,a\,\,\,a+2n\,\cdots)\!)}$ and $(\!(\cdots\,b\,\,\,b-2n\,\cdots)\!)$ corresponds to translations by 
\[{\rho_a-\rho_{a-1}+\begin{rcases}\begin{dcases}
-\rho_0&\text{if }a=1,\\
\rho_{n-1}&\text{if }a=n-2\\
\end{dcases}\end{rcases}
}\in U_0\oplus U_1\]
 and ${\{\pm\rho_{n-2}\mp\rho_{n-1}\text{ if }b=\pm(n-1)\}\in U_2}$.
This is the factorization of translations in $[1,c]_T$ into factors $\lambda_i+q_i\lambda_0$ for $i=1,2$ with $q_1=1$ and $q_2=0$.
Thus the factored translations are loops $\ell_a$ for $a$ outer (i.e.\ $a=1,\ldots,n-1$) and $\ell_b$ for $b$ double (i.e.\ $b=\pm(n-1)$).
We see that $[1,c]_{T\cup L}$ coincides with $[1,c]_{T\cup F}$ for $q_1=1$ and $q_2=0$.
As in Section~\ref{change cox sec}, the same is true for any choice of Coxeter element~$c$.
In light of Theorems~\ref{same interval group} and~\ref{isom b}, we have proved the following theorem.

\begin{theorem}\label{gen cox b}
For any Coxeter element $c$ of $\Stildeses$ and any $q_1$ and $q_2$, construct factored translations~$F$.
Then $\tNCBc$ is isomorphic to the interval $[1,c]_{T\cup F}$ in $\Stildes$.
\end{theorem}

\bigskip

\noindent
\textbf{Acknowledgments.}
The author thanks Laura Brestensky for helpful conversations and two anonymous referees for helpful corrections and suggestions.


\begin{thebibliography}{27}

\bibitem{Ath-Rei}
C. A. Athanasiadis and V. Reiner,
\textit{Noncrossing partitions for the group $D_n$.}
SIAM J. Discrete Math \textbf{18} (2004), no. 2, 397--417.

%
%
\bibitem{Bessis}
D. Bessis,
\textit{The dual braid monoid.}
Ann. Sci. \'{E}cole Norm. Sup. (4) \textbf{36} (2003) no. 5, 647--683.

%
%
\bibitem{Bj-Br}
A. Bj\"{o}rner and F. Brenti,
\textit{Combinatorics of Coxeter groups.}
Graduate Texts in Mathematics, \textbf{231},
Springer, New York, 2005.

%
%
%
\bibitem{Bra-Wa}
T. Brady and C. Watt,
\textit{$K(\pi,1)$'s for Artin groups of finite type.}
Geom. Dedicata \textbf{94} (2002), 225--250.

%
\bibitem{BThesis}
L. Brestensky.
\textit{Planar Models for Noncrossing Partitions in Affine Type.}
Ph.D. Thesis, North Carolina State University, June 2022.

\bibitem{affncA}
L. Brestensky and N. Reading,
\textit{Noncrossing partitions of an annulus.}
Preprint, 2022. (\href{http://arxiv.org/abs/2212.14151}{\texttt{arXiv:2212.14151}}),
to appear in Comb. Theory.





%
%
%
%

\bibitem{Humphreys}
J. Humphreys,
\textit{Reflection Groups and Coxeter Groups.} 
Cambridge Studies in Advanced Mathematics, \textbf{29},
Cambridge Univ. Press, 1990.

%
%
\bibitem{McFailure}
J. McCammond,
\textit{Dual euclidean Artin groups and the failure of the lattice property.}
J. Algebra \textbf{437} (2015), 308--343. 

\bibitem{McSul}
J. McCammond and R. Sulway,
\textit{Artin groups of Euclidean type.}
Invent. Math. \textbf{210} (2017) no.~1, 231--282.

\bibitem{NicOan}
A. Nica and I. Oancea,
\textit{Posets of annular non-crossing partitions of types B and D.}
Discrete Math. \textbf{309} (2009), no~6, 1443--1466.


\bibitem{plane}
N. Reading.
\textit{Noncrossing partitions, clusters and the Coxeter plane.}
S\'em. Lothar. Combin. \textbf{63} (2010) Art. B63b, 32 pages.

\bibitem{surfnc}
N. Reading.
\textit{Noncrossing partitions of a marked surface.}
Preprint, 2022 (\href{http://arxiv.org/abs/2212.13799}{\texttt{arXiv:2212.13799}}),
to appear in SIAM J. Discrete Math.

\bibitem{typefree}
N. Reading and D. Speyer,
\textit{Sortable elements in infinite Coxeter groups.}
Trans. Amer. Math. Soc. \textbf{363} (2011) no.~2, 699--761. 

%
%
%

%





 


\end{thebibliography}
\end{document}